\documentclass[11pt,letterpaper]{article}

\usepackage[normalem]{ulem}
\usepackage{algpseudocode}
\usepackage{algorithm}
\usepackage{arxiv}
\usepackage{subcaption}
\usepackage{siunitx}
\usepackage{threeparttable}
\usepackage{array}
\usepackage{tabularx}
\usepackage{xcolor}
\usepackage[normalem]{ulem} 
\usepackage{booktabs,tikz,caption}
\definecolor{lbboxc}{RGB}{174,214,241}
\definecolor{mbboxc}{RGB}{79,151,214}
\definecolor{dbboxc}{RGB}{44,79,140}
\definecolor{orboxc}{RGB}{224,123,57}
\tikzset{
  bxp/.style={line width=0.6pt},
  medline/.style={line width=1.1pt},
  boxfill/.style={fill=black!12},
  outl/.style={line width=0.5pt},
  axisline/.style={line width=0.3pt, black!45},
  tick/.style={font=\tiny, inner sep=0pt},
  lbluefill/.style={fill=lbboxc},
  mbluefill/.style={fill=mbboxc},
  dbluefill/.style={fill=dbboxc},
  orangefill/.style={fill=orboxc},
  leg/.style={font=\scriptsize, inner sep=0pt},
}
\usepackage[utf8]{inputenc}
\usepackage[T1]{fontenc}
\usepackage{hyperref}
\usepackage{url}
\usepackage{booktabs}
\usepackage{amsfonts}
\usepackage{nicefrac}
\usepackage{amsmath, amsthm, amssymb, bm, mathtools}
\usepackage{graphicx}

\newtheorem{proposition}{Proposition}[section]

\newtheorem{lemma}{Lemma}[section]

\newtheorem{assumption}{Assumption}[section]

\newcommand{\R}{\mathbb{R}}
\newcommand{\C}{\mathbb{C}}
\newcommand{\Gam}{\Gamma}
\newcommand{\GamN}{\Gamma_N}
\newcommand{\PiN}{\Pi_{\Gamma_N}}
\newcommand{\rhoN}{\rho_{\theta,N}}

\newcommand{\rtN}{r_{\theta,N}}
\newcommand{\rstar}{\rho^\infty}             
\newcommand{\rNstar}{\rho^\infty_N}           
\newcommand{\rNlstar}{\rho^\infty_{N_\ell}}   
\newcommand{\tstar}{\theta^\infty}            
\newcommand{\rhotstar}{\rho_{\theta^\infty}}  

\newenvironment{revision}{\begingroup}{\endgroup}

\title{A multilevel stochastic-gradient neural solver for boundary integral equations}
\author{
    Bing-Ze Lu \thanks{Department of Mathematics, National Chung Cheng University, Minhsiung, Chiayi 100190, Taiwan. \texttt{bingzelu.math@gmail.com}}
    \quad and \quad
    Richard Tsai \thanks{Department of Mathematics and Oden Institute for Computational Engineering and Sciences, The University of Texas at Austin, Austin, Texas 78712, U.S.A. \texttt{ytsai@math.utexas.edu}}
}
\date{}

\renewcommand{\headeright}{}
\renewcommand{\undertitle}{}
\renewcommand{\shorttitle}{Neural BIE solver}

\hypersetup{
pdftitle={A multilevel stochastic-gradient neural solver for second-kind boundary integral equations},
pdfauthor={Bing-Ze Lu, Richard Tsai},
pdfkeywords={Boundary integral equations, implicit boundary integral method,
             multilevel training, neural network solvers, stochastic optimization,
             neural tangent kernel, Helmholtz equation, Nystr\"om discretization},
}

\begin{document}
\maketitle
\setcounter{footnote}{0}

\begin{abstract}
We propose a multilevel stochastic-gradient neural solver (MLSG) for second-kind boundary integral equations. MLSG represents the unknown boundary density using a neural network optimized by stochastic residual minimization over a hierarchy of successively refined Nyström discretizations. Upon transitioning from one level to the next, the network parameters obtained on the previous level initialize training on the current one. This coarse-to-fine strategy retains a continuous, grid-independent density representation and is designed to reduce the total computational effort required to reach a prescribed residual tolerance at the target resolution. The algorithm avoids grid-transfer operators and hierarchical fast-summation machinery, relying instead on batched kernel evaluations and standard network forward and backward passes that map efficiently onto modern GPU architectures. For uniformly stable second-kind discretizations, 
strongly nonuniform contraction rates originate in the empirical neural tangent kernel (NTK) rather than in the discretized operator. Within each level, parameter updates can reshape the NTK, while refinement re-samples the tangent kernel on a richer discrete space and reveals directions that were not adequately resolved on coarser levels. A cross-level estimate bounds the warm-start loss in terms of the preceding training tolerance and quadrature error, motivating a tolerance schedule that balances optimization and discretization errors. Experiments on Laplace/Poisson and Helmholtz problems in two and three dimensions, together with an exterior Robin problem on a hypersurface in $\R^4$, demonstrate the method under both parametric and signed-distance surface representations at up to million-scale discretizations. Under the tested hardware and software configuration, MLSG running on a single NVIDIA H200 GPU achieved substantially shorter wall-clock times than a dense-GMRES reference running on 64 cores of Intel Xeon Platinum 8480+ processors. In the schedule-comparison experiments, level revisiting also markedly reduced the run-to-run variability in the number of epochs required for convergence.
\end{abstract}


\section{Introduction}

Boundary integral equations (BIEs) replace an elliptic boundary-value
problem in a domain by an integral equation posed on its boundary.  Besides
reducing the geometric dimension by one, a second-kind formulation has a
particularly useful numerical property: after a stable discretization, its
conditioning remains essentially independent of the quadrature resolution.
Large-scale discretizations nevertheless produce dense systems whose efficient
solution typically relies on specialized fast-summation or hierarchical
techniques.  This paper investigates a different approach: whether a simple
neural residual solver, organized around stochastic row sampling and a hierarchy
of increasingly fine Nystr\"om discretizations, can exploit modern GPU hardware
effectively, and how its coarse-to-fine optimization interacts with the
resolution-independent conditioning of second-kind BIEs and the spectral bias,
or frequency principle, of gradient-based neural training~\cite{rahaman2019spectral,xu2020frequency}.

Let $\Omega \subset \R^d$, $d \geq 2$, be a bounded domain with $C^2$
boundary $\Gam=\partial\Omega$.  We consider Fredholm integral equations of
the second kind
\begin{equation}\label{eq:BIE}
    \mathcal{A} \rho := \tfrac12\rho+\mathcal{K}\rho=g,
    \qquad
    (\mathcal{K}\rho)(x):=\int_\Gam k(x,y)\rho(y)\,dS(y),
    \qquad x\in\Gam,
\end{equation}
where $\rho:\Gam\to\C$ is the unknown density, $g$ is prescribed
boundary data, assumed sufficiently regular for the $L^2$ residual and
quadrature used below, and $\mathcal{K}$ is compact on the function space
under consideration.  Such equations arise from elliptic problems
posed on $\Omega$ and on its exterior~\cite{kress2014linear,colton2013integral}.
Once $\rho$ has been found, the PDE solution is recovered away from the
boundary through a layer potential
\begin{equation}\label{eq:representation}
    u(x)=\int_\Gam \widetilde{k}(x,y)\rho(y)\,dS(y),
    \qquad x\in\R^d\setminus\Gam.
\end{equation}

The discretized system of linear equations
$$
    A_N\rho_N=g_N,
    \qquad
    A_N=\tfrac12 I+K_N\in\C^{N\times N},
$$
is typically solved by an iterative method.  For the stable Nyström
discretizations considered here, the conditioning of $A_N$ remains uniformly
bounded under refinement.  The matrix $A_N$ is nevertheless dense, so a
matrix-vector product costs $\mathcal{O}(N^2)$ without compression.
Fast multipole and hierarchical-matrix techniques reduce this cost to nearly
linear complexity in many standard settings
~\cite{greengard1987fmm,cheng1999fast,hackbusch1999}
and provide the asymptotic benchmark for large-scale BIE solution.

The proposed \emph{multilevel stochastic-gradient neural solver} (MLSG) follows a different design built from three ingredients: a neural representation of the boundary density, stochastic mini-batch residual minimization, and a hierarchy of successively refined Nyström discretizations. Since the neural network represents a continuous surface function independently of the quadrature, refinement merely alters the discretization used to evaluate the residual, allowing training to continue on the finer discretization with the same network. Moreover, because each stochastic update uses only a mini-batch of collocation equations, only the corresponding rows of the dense Nyström operator are evaluated on demand, enabling efficient parallel execution on GPUs. The discretization hierarchy thereby becomes an optimization mechanism rather than merely a device for resolving the final BIE. This sacrifices the asymptotic advantage of fast summation in exchange for a simple and regular computational pattern. The numerical question is therefore not whether MLSG improves the asymptotic complexity of a BIE solve (it does not), but whether this simpler pattern can be practically effective at large scale.

{{Two strands of neural boundary-integral methods are particularly relevant here. For single boundary-value problems, BINet~\cite{lin2023binet} represents the boundary unknown by a neural network and analyzes convergence in the infinite-width NTK regime, while BINN~\cite{sun2023binn} minimizes BIE residuals for potential and elasticity problems; related boundary-only solvers include~\cite{qu2024binnacoustic,zhang2024binnmechanics,ma2025blnn}. Most numerical studies in this line have focused on two space dimensions, although a recent deep-learning-enhanced boundary element formulation treats three-dimensional elasticity~\cite{gu2026deepbem}. Separately, boundary-integral structure has been incorporated into operator learning, both from boundary data alone~\cite{fang2024boundaries} and over families of two-dimensional domains~\cite{meng2025operatorbie}. }}

To understand the coarse-to-fine optimization mechanism, we examine the induced density and residual dynamics and isolate the role of the neural tangent kernel (NTK).  For the parameter gradient flow the sampled density satisfies
\[
  \dot\rho_{\theta,N}
  =-T_{\theta,N}A_N^\ast(A_N\rho_{\theta,N}-g_N),
\]
so the residual is driven by $A_NT_{\theta,N}A_N^\ast$.
We prove that the
eigenvalues of this operator are bounded above and below by the eigenvalues
of $T_{\theta,N}$ times resolution-independent constants.  
The BIE therefore
changes the instantaneous NTK contraction rates only by uniformly bounded
factors; the strongly nonuniform part of the dynamics comes from the
network parametrization.

This mechanism also clarifies the multilevel analogy. MLSG is closer to cascadic or nested iteration than to a classical V-cycle~\cite{hackbusch1985multigrid}.
{{Training first reduces the residual components to which the current tangent space is most responsive; the quadrature is then refined, and the same network is re-evaluated on a richer discrete space. No intergrid transfer operator is required: after refinement, training continues with the same network on the finer discretization. We do not identify the neural optimizer with a classical high-frequency smoother. Instead, different resolutions change which residual components are visible to the evolving NTK, producing the coarse-to-fine spectral progression observed numerically.}}

The same coarse-to-fine construction also gives a quantitative relation
between optimization and discretization error.  Under explicit regularity,
consistency, and sampling assumptions, we show that a level-$\ell$ iterate
satisfying
\[
  L^{(\ell)}(\theta^{(\ell,\mathrm{end})})\leq\varepsilon_\ell
\]
provides a warm start on the next level satisfying
\[
  L^{(\ell+1)}(\theta^{(\ell,\mathrm{end})})
  \leq C\bigl(\varepsilon_\ell+h_\ell^{2m}\bigr).
\]
Thus there is little benefit, from the standpoint of the next refinement, in
solving a coarse discretization substantially beyond its discretization scale.
This leads to the matched tolerance $\varepsilon_\ell\asymp h_\ell^{2m}$ used
to guide the multilevel schedule.

The numerical experiments examine both this optimization mechanism and the
practical capability of the resulting solver.  NTK spectra and residual
projections show the evolution of the effective tangent directions across
refinement levels, while comparisons of progressive and level-revisiting
schedules show that revisiting can substantially reduce run-to-run variability
in the optimization.  
More broadly, we apply the same basic algorithm to
interior Laplace/Poisson and exterior Helmholtz problems on parametric and
implicitly represented surfaces, including nonconvex and multiply connected geometries.  
For implicit geometries, the implicit boundary integral method (IBIM)~~\cite{KublikTanushevTsai2013,ChenTsai2017} constructs the quadrature directly from a signed-distance function sampled on a Cartesian grid (Appendix~\ref{app:surf-disc}), so no surface mesh is required. The same framework is also applied to an exterior Robin problem for the Laplace equation on a ring-shaped hypersurface in $\R^4$.

The implementation consists primarily of on-demand batched kernel evaluations and standard neural-network forward and backward passes, without hierarchical fast-summation data structures.
On a single NVIDIA H200 GPU it handles discretizations reaching more than one
million unknowns in the three-dimensional Helmholtz tests and approximately
$3.3$ million quadrature points in the four-dimensional Robin example.  In the
four-sphere benchmarks, under the hardware and software configurations tested
here, its wall-clock times are substantially shorter than those of the 64-core
dense-MATLAB-GMRES reference.  These measurements demonstrate the practical
effectiveness of the simple GPU-oriented computational pattern; they are not
hardware-normalized comparisons with optimized GPU Krylov or FMM-accelerated BIE
solvers.

The remainder of the paper is organized as follows.
Section~\ref{sec:prelim} introduces the Nyström discretization, the
least-squares residual, and the induced gradient flows.
Section~\ref{sec:Algs} develops the spectral mechanism and the multilevel
algorithm.  Section~\ref{sec:Examples} presents the numerical experiments,
and Section~\ref{sec:Conclusion} summarizes the conclusions and limitations.

\section{Preliminaries}\label{sec:prelim}

\subsection{Notation}\label{sec:notation}

Throughout, $\Omega \subset \R^d$, {\color{black}$d \geq 2$}, is a bounded domain with
$C^2$ boundary $\Gam = \partial\Omega$, equipped with the surface measure
$dS$. Function spaces over $\Gam$ are written without ambient brackets:
$L^2(\Gam)$, $C(\Gam)$, $H^s(\Gam)$, with the $L^2$ inner product
$\langle u, v\rangle_{L^2(\Gam)} := \int_\Gam u(y)\, v(y)\,dS(y)$.  
\paragraph{Surface and quadrature.}
$\GamN := \{x_1,\ldots,x_N\} \subset \Gam$ denotes a set of $N$ quadrature
points on $\Gam$, with associated positive weights
$\{w_1,\ldots,w_N\}$, $w_j > 0$. The quadrature rule
$Q_N(f) := \sum_{j=1}^N w_j\, f(x_j)$ approximates the surface integral
$\int_\Gam f\,dS$ for continuous integrands $f$. We collect the weights into
the diagonal matrix $W := \mathrm{diag}(w_1,\ldots,w_N) \in \R^{N\times N}$.

\paragraph{Weighted inner product and adjoints on $\R^N$.}
We endow $\R^N$ with the discrete weighted inner product
\begin{equation}\label{eq:weighted-inner}
    \langle u, v\rangle_W := \sum_{j=1}^N w_j\, u_j v_j
    \;=\; v^\top W u, \qquad u, v \in \R^N,
\end{equation}
which is the natural quadrature image of the $L^2(\Gam)$ inner product.
Because the quadrature weights are positive, $\langle\cdot,\cdot\rangle_W$
is positive definite and defines a genuine inner product on $\R^N$. The
adjoint of an operator $T \in \R^{N \times N}$ with respect to
\eqref{eq:weighted-inner} is $T^\ast = W^{-1} T^\top W$.  When a Jacobian
$J\in\R^{N\times p}$ maps Euclidean parameter space into the weighted sample
space, its adjoint is $J^\ast=J^\top W\in\R^{p\times N}$.

\paragraph{Restriction operator.}
The restriction operator
\[
    \PiN \colon C(\Gam) \longrightarrow \R^N,
    \qquad
    \PiN[\rho]_j := \rho(x_j), \quad j=1,\ldots,N,
\]
maps a continuous function on $\Gam$ to its values at the quadrature nodes.

\paragraph{Densities.}
$\rho \colon \Gam \to \R$ denotes a continuous (typically $C^1$) density on
$\Gam$, and $\rho_N := \PiN[\rho] \in \R^N$ its sample on the quadrature
nodes.  The exact continuum solution of \eqref{eq:BIE} is denoted $\rstar$;
its nodal sample is $\PiN[\rstar]$.  The exact solution of the discrete
Nystr\"om system is denoted separately by $\rNstar$ below.

\paragraph{Operators.}
We use calligraphic font for operators acting on function spaces over $\Gam$
and ordinary math font for their finite-dimensional discretizations:
$\mathcal{A}, \mathcal{K}$ are the continuum BIE and integral operators in
\eqref{eq:BIE}; $A_N, K_N$ are the corresponding $N \times N$ matrices.
Adjoints in $L^2(\Gam)$ are denoted $\mathcal{A}^\ast, \mathcal{K}^\ast$.
We reserve the asterisk $\ast$ for adjoints throughout, and the star
$\infty$ for exact or optimal quantities (densities and parameters); the
two symbols are kept visually distinct on purpose.

\paragraph{Restriction to real-valued densities.}
The analysis in Sections~\ref{sec:prelim}--\ref{sec:Algs} is written for
real-valued densities and operators.  The Helmholtz examples require
$\rho\colon\Gam\to\C$; we handle them by realification, identifying
$\C\cong\R^2$ and applying the same arguments to the associated
$2N$-dimensional real system.  The discretization, second-kind stability, and
gradient-flow analysis then carry over with $A_N$ replaced by its
realification.

\subsection{The Nystr\"om method}\label{sec:nystrom}

Applying the quadrature rule $Q_N$ to the integral operator $\mathcal{K}$ in
\eqref{eq:BIE} produces the Nystr\"om operator
\begin{equation}\label{eq:Knystrom}
    (K_N \phi)_i \;:=\; \sum_{j=1}^N w_j\, k(x_i, x_j)\,\phi_j,
    \qquad i = 1,\ldots,N,
\end{equation}
acting on $\phi \in \R^N$, and the discrete BIE operator
\begin{equation}\label{eq:Anystrom}
    A_N \;:=\; \tfrac12 I + K_N \;\in\; \R^{N\times N}.
\end{equation}
Setting $g_N = \PiN[g]$, the Nystr\"om discretization of
\eqref{eq:BIE} is the linear system
\begin{equation}\label{eq:linear-system}
    A_N \rNstar \;=\; g_N,
    \qquad \rNstar:=A_N^{-1}g_N,
\end{equation}
whenever $A_N$ is invertible.  In general $\rNstar$ is not identical to the
sample $\PiN[\rstar]$ of the continuum solution; their difference is the
Nystr\"om discretization error estimated below.
For singular kernels, we assume that the chosen quadrature provides a
convergent approximation, including any required near-diagonal
regularization.  We suppress this regularization in the notation and
continue to write the kernel as $k(x,y)$.

\paragraph{Consistency, stability, and spectral convergence.}

The classical collectively-compact Nystr\"om theory of
\cite{anselone1971collectively,atkinson1997numerical} motivates the stability
picture used here.  Under the usual consistency assumptions on the
quadrature (and the corresponding regularization for weakly singular
kernels), $K_N\to\mathcal K$ collectively compactly, and invertibility of
$\mathcal A=\tfrac12\mathcal I+\mathcal K$ yields eventual invertibility of
the discrete operators in the standard Nystr\"om setting.  The gradient-flow
analysis below is formulated in the quadrature-induced $W$-norm, so we state
the required discrete $L^2$ stability separately rather than infer it
formally from sup-norm collectively-compact theory.

\begin{assumption}[Uniform $W$-stability of the Nystr\"om family]
\label{ass:nystrom-W-stability}
For all sufficiently fine quadratures used in the analysis, $A_N$ is
invertible on $(\R^N,\langle\cdot,\cdot\rangle_W)$ and there exist constants
$0<c_A\le C_A<\infty$, independent of $N$, such that
\begin{equation}\label{eq:uniform-A-bounds}
  \|A_N\|_W\le C_A,
  \qquad
  \|A_N^{-1}\|_W\le c_A^{-1}.
\end{equation}
Equivalently, every $W$-singular value of $A_N$ lies in $[c_A,C_A]$.
\end{assumption}

Assumption~\ref{ass:nystrom-W-stability} is the precise conditioning property
used in the remainder of the paper; neither normality of $A_N$ nor an
eigenvalue-clustering hypothesis is required for the gradient-flow
arguments.  It is consistent with the stable second-kind discretizations
used in our experiments and can be checked directly for a chosen quadrature
family.  Standard Nystr\"om spectral-convergence results for isolated
nonzero eigenvalues and spectral projectors remain relevant background, but
are not used as a substitute for the $W$-stability assumption above.

\paragraph{Convergence of the discrete solution.}

Under the corresponding consistency and regularity assumptions, the
Nystr\"om solution satisfies the standard estimate
\begin{equation}\label{eq:nystrom-conv}
    \|\PiN[\rstar] - \rNstar\|_W \;\leq\; C\, h^\nu,
\end{equation}
where $\rstar$ solves \eqref{eq:BIE}, $\rNstar$ solves
\eqref{eq:linear-system}, $h$ is the characteristic quadrature spacing, and
$\nu$ is the convergence order of the chosen quadrature for the regularity
of $\rstar$.

\subsection{Least-squares residual functionals}\label{sec:loss}

Our solver determines an approximate density by minimizing the
$L^2$-residual of the BIE \eqref{eq:BIE}, evaluated either continuously on
$\Gam$ or discretely through the quadrature.

\paragraph{Continuum loss.}
For $\rho \in L^2(\Gam)$ define the continuum residual
\begin{equation}\label{eq:residual-cont}
    r[\rho](x) \;:=\; (\mathcal{A}\rho)(x) - g(x),
    \qquad x \in \Gam,
\end{equation}
and the continuum least-squares functional
\begin{equation}\label{eq:loss-cont}
    \mathcal{L}[\rho;\Gam]
    \;:=\; \tfrac12 \int_\Gam |r[\rho](x)|^2\,dS(x)
    \;=\; \tfrac12 \|\mathcal{A}\rho - g\|_{L^2(\Gam)}^2.
\end{equation}

\paragraph{Discrete loss.}
For $\rho_N \in \R^N$ define the discrete residual
\begin{equation}\label{eq:residual-disc}
    r_N[\rho_N] \;:=\; A_N \rho_N - g_N \;\in\; \R^N,
\end{equation}
and the discrete least-squares functional
\begin{equation}\label{eq:loss-disc}
    L(\rho_N;\GamN) \;:=\; \tfrac12 \sum_{j=1}^N w_j\,|r_N[\rho_N]_j|^2
    \;=\; \tfrac12 \langle r_N[\rho_N],\, r_N[\rho_N]\rangle_W.
\end{equation}
$L(\,\cdot\,;\GamN)$ is the quadrature image of
$\mathcal{L}[\,\cdot\,;\Gam]$ when the input is the restriction of a
continuous density.

\paragraph{Convexity and uniqueness.}
Both functionals are convex quadratic in their respective arguments. The
gradient of $\mathcal{L}[\,\cdot\,;\Gam]$ with respect to $\rho \in
L^2(\Gam)$ is
\begin{equation}\label{eq:grad-cont-rho}
    \nabla_\rho \mathcal{L}[\rho;\Gam]
    \;=\; \mathcal{A}^\ast(\mathcal{A}\rho - g),
\end{equation}
and the gradient of $L(\,\cdot\,;\GamN)$ with respect to $\rho_N \in \R^N$
in the $W$-weighted inner product is
\begin{equation}\label{eq:grad-disc-rho}
    \nabla_{\rho_N} L(\rho_N;\GamN)
    \;=\; A_N^\ast(A_N \rho_N - g_N),
\end{equation}
where $A_N^\ast$ is the $W$-adjoint introduced after
\eqref{eq:weighted-inner}.

The first-order condition for the continuum problem is
$\mathcal A^\ast\mathcal A\rho=\mathcal A^\ast g$.  Bounded invertibility
of $\mathcal A$ gives the coercivity estimate
\[
  \langle \rho,\mathcal A^\ast\mathcal A\rho\rangle_{L^2(\Gam)}
  =\|\mathcal A\rho\|_{L^2(\Gam)}^2
  \ge \|\mathcal A^{-1}\|^{-2}\|\rho\|_{L^2(\Gam)}^2,
\]
so the unique minimizer is $\rstar=\mathcal A^{-1}g$.  Likewise,
\eqref{eq:uniform-A-bounds} implies that, for all sufficiently large $N$,
$A_N^\ast A_N$ is uniformly positive definite in the $W$-inner product and
$L(\cdot;\GamN)$ has the unique minimizer
$\rNstar=A_N^{-1}g_N$ with zero residual.

\paragraph{Network parametrization, parametrized loss, and residual.}
The unknown density on $\Gam$ is represented by a multilayer perceptron
(MLP) $\rho_\theta \colon \Gam \to \R$ with trainable parameters
$\theta \in \R^p$, and we write 
\begin{equation}
 \rhoN := \PiN[\rho_\theta] \in \R^N   
\end{equation}
for its quadrature sample. The concrete architectural choices depend
on how $\Gam$ is described in the application and are deferred to
Section~\ref{sec:Examples}; the analysis of this section applies to
any such parametrization. The loss functionals \eqref{eq:loss-cont} and
\eqref{eq:loss-disc} pull back through the parametrization to functions
of $\theta$,
\begin{equation}\label{eq:loss-theta}
    L(\theta;\Gam) \;:=\; \mathcal{L}[\rho_\theta;\Gam],
    \qquad
    L(\theta;\GamN) \;:=\; L(\rhoN;\GamN),
\end{equation}
and the corresponding residuals are
\begin{equation}\label{eq:res-theta}
    r_\theta \;:=\; \mathcal{A}\rho_\theta - g \;\in\; L^2(\Gam),
    \qquad
    \rtN \;:=\; A_N\,\rhoN - g_N \;\in\; \R^N.
\end{equation}
Although $\mathcal{L}[\,\cdot\,;\Gam]$ and $L(\,\cdot\,;\GamN)$ are
convex in their direct arguments, the composition with the nonlinear
map $\theta \mapsto \rho_\theta$ generally destroys convexity, so
$L(\theta;\Gam)$ and $L(\theta;\GamN)$ are not convex in $\theta$. The
remainder of the section studies the geometry of these landscapes and
the dynamics of the parameter flow on them.

\subsection{Gradient flows}\label{sec:gflow-rho}

We first examine gradient flow in the density itself and then pull that flow
back through the neural parametrization.  The unparametrized density flow
provides a resolution-independent reference, while the parameter-space flow
introduces the neural tangent kernel that governs the anisotropy of the
training dynamics.

\paragraph{Reference gradient flows.}
The $L^2(\Gam)$ gradient flow of $\mathcal{L}[\,\cdot\,;\Gam]$ is
\begin{equation}\label{eq:flow-cont-rho}
    \frac{d}{dt}\rho \;=\; -\,\mathcal{A}^\ast \bigl(\mathcal{A}\rho - g\bigr),
\end{equation}
and the $W$-weighted gradient flow of $L(\,\cdot\,;\GamN)$ on $\R^N$ is
\begin{equation}\label{eq:flow-disc-rho}
    \frac{d}{dt}\rho_N \;=\; -\,A_N^\ast \bigl(A_N \rho_N - g_N\bigr).
\end{equation}
Bounded invertibility of $\mathcal A$ and the uniform discrete bounds
\eqref{eq:uniform-A-bounds} imply
\[
  c_A^2 I \preceq_W A_N^\ast A_N \preceq_W C_A^2 I.
\]
Consequently the unparametrized discrete flow contracts every density-error
direction at a rate between two positive constants independent of $N$; the
continuum statement is analogous.  This resolution-independent reference
rate is the baseline against which the parametrized flow is compared.

\paragraph{Parameter-space flow.}
The gradient flow on the parameters reads
\begin{equation}\label{eq:flow-theta}
    \frac{d}{dt}\theta \;=\; -\,\nabla_\theta L(\theta;\Gam)
    \quad\text{(continuum)},
    \qquad
    \frac{d}{dt}\theta \;=\; -\,\nabla_\theta L(\theta;\GamN)
    \quad\text{(discrete)}.
\end{equation}
The continuum and discrete parameter flows are parallel.  We derive the
continuum form first, because it exposes the operator structure most
transparently, and then state the corresponding discrete formulas used in
the computation.

Applying the chain rule componentwise to $L(\theta;\Gam) =
\mathcal{L}[\rho_\theta;\Gam]$ writes each parameter-gradient component
as an $L^2(\Gam)$ inner product of the density-side residual gradient
$\mathcal{A}^\ast r_\theta$ with the network's parameter sensitivity,
\begin{equation}\label{eq:grad-cont-theta}
    \bigl(\nabla_\theta L(\theta;\Gam)\bigr)_k
    \;=\;
    \Bigl\langle \mathcal{A}^\ast r_\theta,\;
                 \frac{\partial \rho_\theta}{\partial \theta_k}
    \Bigr\rangle_{L^2(\Gam)},
    \qquad k = 1,\ldots,p.
\end{equation}
The {{sensitivity vectors}}
$\{\partial \rho_\theta / \partial \theta_k\}_{k=1}^p$ span the
\emph{neural tangent space} at $\theta$ in $L^2(\Gam)$.

Pulling \eqref{eq:flow-theta} back to the density via the chain rule
gives the induced dynamics for $\rho_\theta$:
\begin{equation}\label{eq:rho-flow-cont}
    \frac{d}{dt}\rho_\theta(x)
    \;=\; \sum_{k=1}^p \partial_{\theta_k}\rho_\theta(x)\,\frac{d}{dt}\theta_k
    \;=\; - \int_\Gam T_\theta(x,y)\,
            \mathcal{A}^\ast \bigl(\mathcal{A}\rho_\theta - g\bigr)(y)\,dS(y),
\end{equation}
where 
\begin{equation}\label{eq:NTK-kernel}
    T_\theta(x,y)
    \;:=\;
    \bigl\langle \nabla_\theta \rho_\theta(x),\,
                 \nabla_\theta \rho_\theta(y)\bigr\rangle_{\R^p}
    \;=\;
    \sum_{k=1}^p \partial_{\theta_k}\rho_\theta(x)\,
                 \partial_{\theta_k}\rho_\theta(y).
\end{equation}
Following~\cite{jacot2018neural}, we call $T_\theta$ the neural tangent
kernel.  Here the network has finite width, so $T_\theta$ evolves with
$\theta$ rather than remaining frozen at its initialization value.  This
evolution is one of the mechanisms examined by the multilevel analysis in
Section~\ref{sec:Algs}.

The corresponding integral operator on $L^2(\Gam)$ is
\begin{equation}\label{eq:NTK-operator}
    (\mathcal{T}_\theta u)(x)
    \;:=\; \int_\Gam T_\theta(x,y)\,u(y)\,dS(y).
\end{equation}
Substituting \eqref{eq:NTK-kernel} into the right-hand side and
exchanging the integral with the finite sum gives the factorization
\begin{equation}\label{eq:NTK-finite-rank}
    (\mathcal{T}_\theta u)(x)
    \;=\; \sum_{k=1}^{p}
        \bigl\langle \partial_{\theta_k}\rho_\theta,\, u
        \bigr\rangle_{L^2(\Gam)}\,
        \partial_{\theta_k}\rho_\theta(x),
    \qquad u \in L^2(\Gam),
\end{equation}
from which several properties of $\mathcal{T}_\theta$ follow directly.
The range of $\mathcal{T}_\theta$ lies in
$\mathrm{span}\{\partial_{\theta_k}\rho_\theta\}_{k=1}^{p}$, a subspace
of dimension at most $p$. Cauchy--Schwarz on each inner product gives
$\|\mathcal{T}_\theta u\|_{L^2(\Gam)}
 \leq \bigl(\sum_{k=1}^{p}\|\partial_{\theta_k}\rho_\theta\|_{L^2(\Gam)}^{2}\bigr)\,
      \|u\|_{L^2(\Gam)}$,
so $\mathcal{T}_\theta$ is bounded. Pairing \eqref{eq:NTK-finite-rank}
against $v \in L^2(\Gam)$,
\[
    \bigl\langle \mathcal{T}_\theta u, v\bigr\rangle_{L^2(\Gam)}
    \;=\; \sum_{k=1}^{p}
        \bigl\langle \partial_{\theta_k}\rho_\theta,\, u
        \bigr\rangle_{L^2(\Gam)}
        \bigl\langle \partial_{\theta_k}\rho_\theta,\, v
        \bigr\rangle_{L^2(\Gam)},
\]
which is symmetric in $(u,v)$ and reduces to
$\sum_{k=1}^{p}\langle \partial_{\theta_k}\rho_\theta, u\rangle_{L^2(\Gam)}^{\,2} \geq 0$
when $v = u$, so $\mathcal{T}_\theta$ is self-adjoint and positive
semi-definite. Bounded subsets of the finite-dimensional range are
precompact, hence $\mathcal{T}_\theta$ is compact on $L^2(\Gam)$. We
summarize:

\begin{proposition}[Properties of the continuum NTK operator]
\label{prop:ntk-compact}
Let $\rho_\theta \colon \Gam \to \R$ be a network parametrization with
$p$ trainable parameters whose sensitivities
$\partial_{\theta_k}\rho_\theta$ are continuous on $\Gam$. Then the
NTK integral operator $\mathcal{T}_\theta \colon L^2(\Gam) \to
L^2(\Gam)$ defined by \eqref{eq:NTK-operator} is bounded,
self-adjoint, positive semi-definite, and of rank at most $p$. In
particular, $\mathcal{T}_\theta$ is compact on $L^2(\Gam)$.
\end{proposition}

The same calculation on $\GamN$ replaces the $L^2(\Gam)$ pairing by the
$W$-pairing and the operator $\mathcal{A}^\ast$ by its $W$-adjoint
$A_N^\ast$:
\begin{equation}\label{eq:grad-disc-theta}
    \bigl(\nabla_\theta L(\theta;\GamN)\bigr)_k
    \;=\;
    \Bigl\langle A_N^\ast \rtN,\;
                 \frac{\partial \rhoN}{\partial \theta_k}
    \Bigr\rangle_W,
    \qquad k = 1,\ldots,p,
\end{equation}
or equivalently 
\begin{equation}
 \nabla_\theta L(\theta;\GamN) = J_{\theta,N}^\ast\,
A_N^\ast \rtN,   
\end{equation}
where $J_{\theta,N}$ is the Jacobian
\begin{equation}
 J_{\theta,N} \in \R^{N\times p},\qquad
 (J_{\theta,N})_{i,k} := \partial_{\theta_k}\rho_\theta(x_i),
 \qquad J_{\theta,N}^\ast=J_{\theta,N}^\top W.
\end{equation}

The induced flow on the sampled density is
\begin{equation}\label{eq:rho-flow-disc}
    \frac{d}{dt}\rhoN
    \;=\; J_{\theta,N}\,\frac{d}{dt}\theta
    \;=\; -\,T_{\theta,N}\,A_N^\ast(A_N\,\rhoN - g_N),
\end{equation}
with the \emph{empirical neural tangent kernel} on $\GamN$
\begin{equation}\label{eq:eNTK}
    T_{\theta,N} \;:=\; J_{\theta,N}\, J_{\theta,N}^\ast \;\in\; \R^{N\times N}.
\end{equation}
Writing $G_N$ for the Gram matrix $(G_N)_{ij} = T_\theta(x_i,x_j) = \langle \nabla_\theta\rho_\theta(x_i), \nabla_\theta\rho_\theta(x_j)\rangle_{\R^p}$, {\color{black}{the empirical NTK factors as
\begin{equation}\label{eq:ntk-nystrom}
    T_{\theta,N} \;=\; G_N\,W,
    \qquad
    (T_{\theta,N}\,u)_i \;=\; \sum_{j=1}^{N} w_j\,T_\theta(x_i,x_j)\,u_j;
\end{equation}
that is, $T_{\theta,N}$ is the Nystr\"om discretization of
$\mathcal{T}_\theta$ on $\GamN$, produced from the kernel $T_\theta$ by
the same quadrature rule that produces $A_N$ from $\mathcal{A}$ in
Section~\ref{sec:nystrom}. The weight factor distinguishes it from the
symmetric Gram matrix $G_N$.}} The same argument as in the
proof of Proposition~\ref{prop:ntk-compact}, applied to the
factorization $T_{\theta,N} = J_{\theta,N} J_{\theta,N}^\ast$, shows
that $T_{\theta,N}$ is $W$-self-adjoint, positive semi-definite on
$(\R^N, \langle\cdot,\cdot\rangle_W)$, and of rank at most $\min(p,N)$.
It is rank-deficient whenever $p < N$, the regime of interest at the
finer levels.

\paragraph{Comparison with the un-parametrized flow.}
The reference flows~\eqref{eq:flow-cont-rho}
and~\eqref{eq:flow-disc-rho} correspond formally to
$T_\theta(x,y)=\delta(x-y)$ (equivalently $\mathcal T_\theta=\mathcal I$
and $T_{\theta,N}=I$).  Because $A_N^\ast A_N$ is uniformly well
conditioned in $N$ (Section~\ref{sec:nystrom}), every density-error
direction then contracts at a rate bounded above and below independently
of the resolution.  For a neural parametrization,
$\mathcal T_\theta$ has rank at most $p$, while $T_{\theta,N}$ has rank
at most $\min(p,N)$ and is rank-deficient whenever $p<N$.  Small or zero
NTK eigenvalues therefore create slow or invisible tangent directions in
the induced flow.  Section~\ref{sec:Algs} develops this observation as the
basis for the multilevel strategy.

\paragraph{Residual dynamics.}
Applying $A_N$ to \eqref{eq:rho-flow-disc} yields the residual flow
\begin{equation}\label{eq:r-flow-disc}
    \frac{d}{dt}\rtN \;=\; -\,A_N\,T_{\theta,N}\,A_N^\ast\, \rtN.
\end{equation}
The driving operator
$A_N\,T_{\theta,N}\,A_N^\ast$ is $W$-self-adjoint and positive
semi-definite. 

Equation~\eqref{eq:r-flow-disc} makes the source of nonuniform contraction
explicit: the empirical NTK appears inside the residual-flow operator
$A_NT_{\theta,N}A_N^\ast$.  Section~\ref{sec:Algs} quantifies this
dependence spectrally.

Training terminates at finite time, so the residual must also serve as an
a posteriori indicator of density error.  Assumption~\ref{ass:nystrom-W-stability}
provides the required uniform inverse bound.

\begin{lemma}[Residual controls error]\label{lem:res-to-err}
Let $N$ be large enough that $A_N$ is invertible. For every
$\rho_N \in \R^N$,
\begin{equation}\label{eq:res-to-err-disc}
    \|\rho_N - \rNstar\|_W
    \;\leq\;
    \|A_N^{-1}\|_W\,
    \|A_N \rho_N - g_N\|_W,
\end{equation}
and Assumption~\ref{ass:nystrom-W-stability} bounds
$\|A_N^{-1}\|_W$ uniformly in $N$.  In particular, for the network sample
$\rhoN=\PiN[\rho_\theta]$, the observable residual $\|\rtN\|_W$
controls the distance to the discrete solution.
\end{lemma}

We remark here that if $A_N$ comes from a single layer formulation for the Neumann Helmholtz problem, $||A^{-1}_N||_W$ formally depends on the wave number and will increase significantly
when the system is near resonance. We do not address such situations in this paper.

\section{The multilevel stochastic-gradient neural solver}\label{sec:Algs}

We develop the MLSG solver by minimizing the discrete residual across a hierarchy of quadrature nodes with sizes $N_1<N_2<\cdots<N_{\ell_{\text F}}$, while transferring the network parameters from one training stage to the next. Moreover, the training procedure is flexible and can re-use quadrature points. Accordingly, a \emph{progressive} schedule traverses the hierarchy once from coarse to fine, whereas a \emph{cyclic} schedule returns to selected coarser levels before performing the final refinement.

Sections~\ref{sec:Algs:motivation}--\ref{sec:Algs:expansion} explain why MLSG strategy can accelerate the solution of the boundary integral equations considered here. First, the dense BIE evaluations at each level are mini-batched, and each batch is a dense matrix product of exactly the shape GPU execution favors. {{Second, the scheme is cascadic rather than multigrid with smoothing, so each level only refines the initial guess inherited from the coarser level and no relaxation sweeps are needed. }}Third, re-sampling changes the geometry of the neural tangent kernel, which helps training escape the plateau reached at the end of the previous resolution.
Sections 3.4 and 3.5 define the multilevel training algorithms and quantify the residual transferred between grid levels to balance the optimization and discretization errors.

\subsection{Computational accounting}\label{sec:Algs:motivation}

\emph{(i) Dense row operations.}  A residual evaluation for a mini-batch
uses dense dot products against selected rows of $A_N$, followed by a
network forward/backward pass.  These operations are regular and highly
parallel, although their wall-clock behavior depends on memory traffic,
kernel generation, and the network architecture.

\emph{(ii) Mini-batches change the per-update arithmetic.}  
{{If a batch contains $b$ collocation rows, each stochastic update requires only $\mathcal O(bN)$ kernel interactions, compared with the $\mathcal O(N^2)$ interactions of a full residual evaluation. Since $b \ll N$, this substantially reduces the cost of each optimization step.}}
\emph{(iii) The represented object is independent of the grid.}  The same
function $\rho_\theta$ can be evaluated on every quadrature in the ladder.
Thus interlevel transfer is the identity on $\theta$; no explicit
restriction or prolongation of a grid function is required.

\emph{(iv) Warm starts move work away from the finest grid.}  A refined
stage begins from a network that has already reduced the residual visible on
a coarser quadrature.  The analysis in Section~\ref{sec:Algs:analysis}
quantifies the size of the residual passed to the next level, while the
experiments measure how much optimization work remains after the transfer.

\subsection{Spectral bias and the effective neural tangent space}\label{sec:Algs:mg-analogy}

A multilevel hierarchy can help when different resolutions change the optimization problem seen by the network, thereby alleviating training bottlenecks. {{Classical geometric multigrid achieves this via smoothing sweeps and coarse-grid correction, while MLSG operates through a different spectral mechanism: grid refinement re-samples the residual and redefines the empirical NTK $T_{\theta,N}$ at a fixed parameter set $\theta$,  thereby revealing previously unresolved tangent directions. Consequently, while MLSG shares the concept of spectral separation with multigrid, it lacks explicit intergrid correction sweeps. Algorithmically, MLSG is therefore closer to a cascadic or nested iteration solver.}}

The residual flow is
\begin{equation}\label{eq:r-flow-spectral-repeat}
  \dot r_{\theta,N}=-B_{\theta,N}r_{\theta,N},
  \qquad
  B_{\theta,N}:=A_NT_{\theta,N}A_N^\ast.
\end{equation}
The following estimate makes precise the statement that a well-conditioned
second-kind BIE changes the NTK spectrum only by bounded factors.

\begin{proposition}[Spectral equivalence of the residual-flow and NTK operators]
\label{prop:bie-ntk-spectral-equivalence}
Let $A_N$ be invertible on $(\R^N,\langle\cdot,\cdot\rangle_W)$ and let
$T_{\theta,N}$ be $W$-self-adjoint and positive semidefinite.  Write
\[
  \alpha_N:=\|A_N^{-1}\|_W^{-1},\qquad
  \beta_N:=\|A_N\|_W,
\]
and order the eigenvalues of $T_{\theta,N}$ and $B_{\theta,N}$ as
$\sigma_{1,N}\ge\cdots\ge\sigma_{N,N}\ge0$ and
$\mu_{1,N}\ge\cdots\ge\mu_{N,N}\ge0$, respectively.  Then
\begin{equation}\label{eq:bie-ntk-spectral-equivalence}
  \alpha_N^2\sigma_{j,N}
  \le \mu_{j,N}
  \le \beta_N^2\sigma_{j,N},
  \qquad j=1,\ldots,N.
\end{equation}
In particular, under \eqref{eq:uniform-A-bounds}, the equivalence constants
are independent of the quadrature resolution.
\end{proposition}

\begin{proof}
After the isometry $v\mapsto W^{1/2}v$, the statement is an ordinary
Euclidean matrix inequality.  Let $T_{\theta,N}^{1/2}$ be the positive
square root.  The nonzero eigenvalues of
$B_{\theta,N}=A_NT_{\theta,N}A_N^\ast$ are the squared singular values of
$A_NT_{\theta,N}^{1/2}$, while the eigenvalues of $T_{\theta,N}$ are the
squared singular values of $T_{\theta,N}^{1/2}$.  The standard product
inequalities for singular values give
$\alpha_N s_j(T_{\theta,N}^{1/2})\le
s_j(A_NT_{\theta,N}^{1/2})\le
\beta_N s_j(T_{\theta,N}^{1/2})$; squaring yields
\eqref{eq:bie-ntk-spectral-equivalence}.
\end{proof}

Proposition~\ref{prop:bie-ntk-spectral-equivalence} requires no simultaneous
diagonalization of $A_N$ and $T_{\theta,N}$.  At fixed $\theta$, any large
spread in the nonzero residual-contraction rates is therefore inherited
from the empirical NTK up to constants that remain uniform with respect to
quadrature refinement.

\paragraph{NTK spectrum and the meaning of frequency.}
The tendency of neural networks to fit lower-frequency components before
more oscillatory ones is commonly described as \emph{spectral bias} or the
\emph{frequency principle}~\cite{rahaman2019spectral,xu2020frequency}.

The empirical NTK has the decomposition
\begin{equation}\label{eq:ntk-spec}
    T_{\theta,N}=V_N\Sigma_NV_N^\ast,
    \qquad
    \Sigma_N=\operatorname{diag}(\sigma_{1,N},\ldots,\sigma_{N,N}),
\end{equation}
with $W$-orthonormal eigenvectors $\tau_{j,N}$.  At the continuum level,
Proposition~\ref{prop:ntk-compact} shows that $\mathcal T_\theta$ has rank
at most $p$.  It therefore has at most $p$ positive eigenvalues and vanishes
on the orthogonal complement of its range.

Compactness alone does \emph{not} provide a canonical Fourier ordering of
the eigenfunctions on a general surface.  Accordingly, throughout the
analysis we use the ordered eigenvalues only as tangent-kernel directions.
When the numerical experiments refer to low- and high-frequency NTK modes,
that terminology describes an additional empirical observation: for the
smooth sinusoidal networks used there, the leading computed eigenvectors are
less oscillatory and increasingly oscillatory structure appears deeper in
the spectrum.  Figures~\ref{fig:ntk_multi_initial}--\ref{fig:single_level} test this
ordering directly.

\paragraph{Exact instantaneous loss identity.}
Project the back-projected residual onto the NTK eigenbasis,
\begin{equation}\label{eq:proj-coords}
    A_N^\ast r_{\theta,N}
    =\sum_{j=1}^N\eta_j(t)\tau_{j,N}(\theta(t)),
    \qquad
    \eta_j(t)=\langle\tau_{j,N}(\theta(t)),A_N^\ast r_{\theta,N}(t)\rangle_W.
\end{equation}
Then
\begin{equation}\label{eq:proj-flow}
  \frac{d}{dt}L(\theta(t);\GamN)
  =-\langle A_N^\ast r_{\theta,N},
          T_{\theta,N}A_N^\ast r_{\theta,N}\rangle_W
  =-\sum_{j=1}^N\sigma_{j,N}(\theta)|\eta_j(t)|^2.
\end{equation}
This identity is exact even though the eigenbasis evolves with $\theta$.
It shows that a large component of $A_N^\ast r_{\theta,N}$ can contribute
very little to the instantaneous loss decrease if it lies in small-$\sigma$
directions.

\paragraph{Residual projections used in the numerical diagnostics.}
The lower panels of Figures~\ref{fig:ntk_multi_initial}--\ref{fig:single_level}
serve a different purpose from the exact decomposition
\eqref{eq:proj-flow}.  They project the residual itself onto the NTK
eigenbasis,
\begin{equation}\label{eq:res-proj-coords}
    r_{\theta,N}
    =\sum_{j=1}^N \zeta_j(t)\tau_{j,N}(\theta(t)),
    \qquad
    \zeta_j(t)=\langle\tau_{j,N}(\theta(t)),r_{\theta,N}(t)\rangle_W.
\end{equation}
These coefficients should not be confused with the $\eta_j$ in
\eqref{eq:proj-coords}: the adjoint $A_N^\ast$ may rotate the residual, so
$\zeta_j$ is not the modal coefficient in the instantaneous loss-decay
identity.  We nevertheless use $\zeta_j$ in the figures because it displays
the unresolved BIE residual directly.  If
$e_N:=\rho_{\theta,N}-\rho_N^\infty$, then $r_{\theta,N}=A_Ne_N$, and
Assumption~\ref{ass:nystrom-W-stability} gives
\begin{equation}\label{eq:res-error-equivalence}
    c_A\|e_N\|_W
    \le \|r_{\theta,N}\|_W
    \le C_A\|e_N\|_W.
\end{equation}
Thus the residual norm is uniformly equivalent to the density error.  The
coefficients $\zeta_j$ show how this unresolved residual is distributed
relative to the tangent directions, while \eqref{eq:res-error-equivalence}
gives the residual magnitude a direct error interpretation.  Equation~\eqref{eq:proj-flow},
by contrast, describes the instantaneous optimization dynamics.

For comparison, the unparametrized density flow corresponds to
$T_{\theta,N}=I$ and gives
$\dot L=-\|A_N^\ast r_{\theta,N}\|_W^2$.  This motivates the algebraic
definition
\begin{equation}\label{eq:eff-tangent-space}
  \mathcal E_\theta^{(N)}
  :=\operatorname{span}\{\tau_{j,N}(\theta):\sigma_{j,N}(\theta)\ge1\}.
\end{equation}
The threshold $1$ is only a normalization relative to the unparametrized
loss identity; it is not a universal physical frequency cutoff.  Directions
inside $\mathcal E_\theta^{(N)}$ contribute at least as strongly as the
corresponding unit-NTK direction in \eqref{eq:proj-flow}, whereas directions
deep in the small-eigenvalue tail contribute much less.

For a frozen $\theta$, explicit Euler discretization of the residual flow is
stable under a step size controlled by $\mu_{1,N}(B_{\theta,N})$, which
Proposition~\ref{prop:bie-ntk-spectral-equivalence} bounds by
$C_A^2\sigma_{1,N}$.  Quadrature weights are intrinsic to the discretization
rather than a normalization applied afterwards: on a quasi-uniform grid
where $w_j=\mathcal O(N^{-1})$, $T_{\theta,N}=G_NW$ is a Nystr\"om
discretization of the continuum kernel $T_\theta$, so its eigenvalues track
those of $T_\theta$ instead of scaling with $N$ (tubular weights fulfill the
same role in IBIM).  Without the weights, $T_{\theta,N}$ would reduce to
$G_N$, whose leading eigenvalues scale linearly with $N$, and the spectra at
consecutive levels would differ by a factor of $N$ rather than
characterizing a single continuum operator.  Quadrature weighting thus
eliminates grid-dependent step-size penalties and lets the warm-start
analysis in Section~\ref{sec:Algs:analysis} operate on a single fixed
spectrum.

In practice, there are two effects that modify this idealized picture.  First of all, parameter $\sigma_{1,N}$ changes during training. Experiments in~Figures~\ref{fig:ntk_multi_initial}--\ref{fig:single_level} show that they are not static.  Second, Adam's
parameter preconditioning changes the quantitative damping rates, so
\eqref{eq:proj-flow} and
Proposition~\ref{prop:bie-ntk-spectral-equivalence} characterize the loss
geometry rather than the exact discrete optimizer path.
Section~\ref{sec:Examples:spectral} shows that this spectral organization
remains visible along actual training trajectories.

\subsection{Expansion of the effective neural tangent space}\label{sec:Algs:expansion}

As mentioned above, the NTK evolves throughout training.
 Within a fixed level, changing $\theta$ changes the continuum tangent kernel
itself. Across levels, changing the quadrature re-samples the same kernel at the transferred parameter value, and the change can reveal discrete directions that were not
resolved accurately on the coarser grid even when $\theta$ is unchanged.

\paragraph{A finite-width within-level mechanism.}
Consider the one-hidden-layer reduction
\begin{equation}\label{eq:one-layer-reduction}
  \rho_\theta(x)=\sum_m a_m\psi(z_m(x)),
  \qquad z_m(x)=w_m^\top x+b_m.
\end{equation}
Let $\Psi_m=\PiN[\psi(z_m)]$ and let
$\Phi_m=\PiN[\psi'(z_m)\widetilde x]$ denote the sampled sensitivity to
$(w_m,b_m)$, with $\widetilde x=(x,1)$.  The empirical NTK splits as
\begin{equation}\label{eq:ntk-split}
  T_{\theta,N}
  =\underbrace{\sum_m\Psi_m\Psi_m^\ast}_{\text{output-weight component}}
   +\underbrace{\sum_m a_m^2\Phi_m\Phi_m^\ast}_{\text{inner-weight component}}.
\end{equation}
Hence
\begin{equation}\label{eq:rayleigh-amp}
  \langle u,T_{\theta,N}u\rangle_W
  =\sum_m|\langle\Psi_m,u\rangle_W|^2
   +\sum_m a_m^2\|\Phi_m^\ast u\|_2^2.
\end{equation}
For fixed $(w_m,b_m)$ and all other parameters, increasing $a_m^2$ adds a
positive-semidefinite term to $T_{\theta,N}$; by the min--max principle its
ordered eigenvalues cannot decrease.  Thus growth of an output amplitude is
one explicit mechanism by which the tangent spectrum can be lifted.

The output-weight dynamics are
\begin{equation}\label{eq:am-dyn}
  \frac{d}{dt}a_m
  =-\langle A_N^\ast r_{\theta,N},\Psi_m\rangle_W.
\end{equation}
Equation~\eqref{eq:am-dyn} determines whether $|a_m|$ actually grows; mere
nonzero overlap is not enough, because the sign relative to $a_m$ matters.
The point of \eqref{eq:ntk-split} is therefore not to assert monotone NTK
growth along every trajectory, but to identify a concrete route by which
residual-driven parameter motion can strengthen previously weak tangent
directions.

The activation controls how the two factors in this mechanism are related.
Formally, differentiation multiplies Fourier content by frequency,
$\widehat{\psi'}(\xi)=i\xi\widehat\psi(\xi)$.  For sinusoidal activations,
$\sin' = \cos$, so the activation and its inner-weight sensitivity retain
the same frequency scale up to phase and polynomial coordinate factors.
This provides a useful reason for using sinusoidal networks in the spectral
experiments, but it does not imply that an arbitrary residual is aligned
with a single NTK eigenmode.

\paragraph{Cross-level re-sampling.}
Suppose the parameter $\theta$ is frozen at the end of a level and the grid is
refined.  The continuum kernel $T_\theta$ is unchanged, while
$T_{\theta,N}=G_NW$ is replaced by a finer Nystr\"om sample.  On nonnested
grids the discrete eigenspaces are not literal supersets of one another and
the dimension of \eqref{eq:eff-tangent-space} need not be monotone as a
matter of algebra.  What refinement does provide is a more resolved
discretization of the continuum tangent operator and a new sampling of the
residual.  In the experiments, this re-sampling reveals additional
directions with eigenvalues above the effective threshold and moves a
substantial part of the warm-start residual into those directions.  We
refer to that observed phenomenon as \emph{expansion of the effective
neural tangent space}.

The two effects then interact.  Refinement changes what the optimizer sees;
subsequent parameter motion can change the NTK itself through mechanisms
such as \eqref{eq:ntk-split}, rotation of the sensitivities under
$w_m$-dynamics, and deeper-layer tangent-kernel alignment
\cite{atanasov2022silent}.  The one-layer calculation isolates only the
first of these mechanisms and is not intended as a complete dynamical model
of the deep networks used numerically.

In the spectral experiment, the effective space
$\mathcal E_\theta^{(N)}$ at the coarser levels is dominated by the
less-oscillatory leading NTK directions.  Training therefore constructs the
coarse-level density primarily through these smooth tangent directions,
rather than by an arbitrary nodewise interpolation.  As refinement exposes
more oscillatory effective directions, the same network acquires the
additional structure needed to represent finer scales.  We regard this as
an empirical smoothness mechanism of the multilevel representation, not as
a general interpolation theorem.

\paragraph{Plateaus.}
Equation~\eqref{eq:proj-flow} also gives a conservative interpretation of a
training plateau.  If the back-projected residual lies mainly in directions
with very small $\sigma_{j,N}$, the loss decreases slowly.  During such an
interval the parameters may still move and reshape the NTK; if this motion
raises the relevant tangent directions, faster contraction can follow.  If,
on the other hand, the back-projected residual is nearly orthogonal both to
the effective tangent directions and to the sensitivities that can alter
them, then the plateau can persist.  In the one-layer model a sufficient
signature for the amplitude mechanism to be inactive is
\begin{equation}\label{eq:stall}
  |\eta_j|\approx0\quad\text{for the effective directions},
  \qquad
  \langle A_N^\ast r_{\theta,N},\Psi_m\rangle_W\approx0
  \quad\text{for all }m.
\end{equation}
This is a diagnostic, not a theorem that all observed plateaus have two
universal phases.  Section~\ref{sec:Examples:spectral} examines the actual
spectral trajectories and shows why refinement is effective in the tested
problem.

\subsection{The multilevel algorithm}\label{sec:Algs:smt}

\paragraph{Ladder of quadrature grids.}
Fix levels $\ell = 1,\ldots,\ell_{\text{F}}$. At level $\ell$ we
apply the chosen quadrature rule on $\Gam$ with $N_\ell$ nodes and positive
weights, with node set and weight matrix
\begin{equation}\label{eq:level-grid}
    \Gam^\ell \;:=\; \Gamma_{N_\ell}
        \;=\; \{x_1^\ell,\ldots,x_{N_\ell}^\ell\},
    \qquad
    W^\ell \;:=\; \mathrm{diag}\bigl(w_1^\ell,\ldots,w_{N_\ell}^\ell\bigr).
\end{equation}
We do not assume the grids are nested.

The levels are indexed by increasing resolution,
$N_1<N_2<\cdots<N_{\ell_{\text F}}$.  They form a fixed ladder; the
order in which the solver visits them is specified separately by the
schedule below and need not be monotone.


\paragraph{Per-level residual and loss.}
At level $\ell$ the discrete least-squares functional \eqref{eq:loss-disc} is
\begin{equation}\label{eq:loss-level}
    L^{(\ell)}(\theta)
    \;:=\;
    \tfrac12\,\bigl\| A_{N_\ell}\, \Pi_{\Gam^\ell}[\rho_\theta]
                     \,-\, g_{N_\ell} \bigr\|_{W^\ell}^{2},
\end{equation}
with discrete residual and parameter Jacobian
\begin{equation}\label{eq:res-level}
    r_\theta^{(\ell)}
    \;:=\;
    A_{N_\ell}\, \Pi_{\Gam^\ell}[\rho_\theta] \,-\, g_{N_\ell},
    \qquad
    \bigl(J_\theta^{(\ell)}\bigr)_{ik}
    \;:=\;
    \partial_{\theta_k}\rho_\theta(x_i^\ell).
\end{equation}
The parameter-space gradient is
\begin{equation}\label{eq:grad-level}
    \nabla_\theta L^{(\ell)}(\theta)
    \;=\;
    \bigl(J_\theta^{(\ell)}\bigr)^{\ast}\, A_{N_\ell}^{\ast}\,
        r_\theta^{(\ell)}
    \;\in\; \R^p,
\end{equation}
where both $\bigl(J_\theta^{(\ell)}\bigr)^\ast$ and $A_{N_\ell}^\ast$
denote $W^\ell$-adjoints, in the sense of Section~\ref{sec:notation}. The
empirical NTK at level $\ell$ is
$T_\theta^{(\ell)} := J_\theta^{(\ell)} (J_\theta^{(\ell)})^{\ast} \in
\R^{N_\ell \times N_\ell}$.

\paragraph{Per-level minimizers.}
$L^{(\ell)}$ is convex in $\Pi_{\Gam^\ell}[\rho_\theta]$ and bounded below
by $0$. When the network has enough capacity to interpolate the discrete
solution $\rNlstar$ at the $N_\ell$ nodes of
$\Gam^\ell$, the lower bound is attained at some level-$\ell$ minimizer
$\tstar$, and the corresponding network output satisfies
\begin{equation}\label{eq:level-minimizer}
    \Pi_{\Gam^\ell}\bigl[\rhotstar\bigr] \;=\; \rNlstar.
\end{equation}
$L^{(\ell)}$ pins $\rho_\theta$ only on $\Gam^\ell$, so the values of
$\rhotstar$ on $\Gam \setminus \Gam^\ell$ are unconstrained, and
$\rhotstar$ is in general not a minimizer of either $L^{(\ell+1)}$
or the continuum loss \eqref{eq:loss-cont}.

\paragraph{Warm-started traversal and visiting schedules.}
The solver visits the levels of the ladder in an order fixed in advance by
a \emph{schedule}
\begin{equation}\label{eq:schedule}
    \mathcal{S} \;=\; (s_1, s_2, \ldots, s_T),
    \qquad s_t \in \{1,\ldots,\ell_{\text{F}}\},
\end{equation}
a sequence of $T$ \emph{stages}, each naming the grid used at that stage.
Stage $t$ trains on $\Gam^{s_t}$, warm-started from the parameters returned
by the previous stage,
\begin{equation}\label{eq:warm-start-stage}
    \theta^{(t,\mathrm{init})} \;:=\; \theta^{(t-1)},
    \qquad
    \theta^{(t)} \;\approx\; \arg\min_\theta L^{(s_t)}(\theta),
\end{equation}
with $\theta^{(0)}$ a fresh random initialization and $\theta^{(t)}$ the
parameter returned after a finite number of stochastic gradient steps.

The progressive schedule is
$\mathcal S=(1,2,\ldots,\ell_{\mathrm F})$: each stage refines the
quadrature, so the stage and level indices coincide and the hierarchy is
traversed once from coarse to fine.  This is the schedule most directly
related to cascadic multigrid~\cite{bornemann1996cascadic}.  A cyclic
schedule is nonmonotone; for example,
$\mathcal S=(1,2,3,4,3,4,5)$ revisits levels~3 and~4 before the final
refinement.  The purpose of such revisits is empirical: a return to a
coarser discretization can reshape the parameters before the next ascent.
We do not identify this schedule with a classical FMG cycle.

Algorithm~\ref{alg:slsg} serves as a building block of the main MLSG algorithm~\ref{alg:smt}.
\textsc{SLSG} performs stochastic-gradient residual minimization on a
fixed quadrature, while \textsc{MLSG} calls \textsc{SLSG} according to
$\mathcal S$, warm-starting every stage from the parameters returned by the
preceding one.

\paragraph{Mini-batch loss.}
The implementation minimizes the same quadrature-weighted loss
$L^{(\ell)}$ defined in \eqref{eq:loss-level}.  For a uniformly sampled
mini-batch $\mathcal B\subset\{1,\ldots,N_\ell\}$, the optimizer uses
\begin{equation}\label{eq:minibatch-loss}
    L_{\mathcal B}^{(\ell)}(\theta)
    :=\frac{N_\ell}{2|\mathcal B|}\sum_{i\in\mathcal B}
       w_i^\ell\bigl|r_{\theta,i}^{(\ell)}\bigr|^2,
\end{equation}
which is an unbiased stochastic estimator of the deterministic quadrature
loss \eqref{eq:loss-level}.  When no level index is needed, we write this
mini-batch loss as $L_{\mathcal B}(\theta;\GamN)$.  Computing
$\nabla_\theta L_\mathcal B^{(\ell)}$ requires only the
$|\mathcal B|$ rows of $A_{N_\ell}$ indexed by $\mathcal B$, together
with one forward and one backward pass of $\rho_\theta$ at the $N_\ell$
quadrature nodes, for a per-step arithmetic count of
$\mathcal O(|\mathcal B|N_\ell)$.  A complete epoch visits all rows and
therefore still entails $\mathcal O(N_\ell^2)$ dense kernel interactions.

\begin{algorithm}[h]
\caption{\textsc{SLSG}: single-level stochastic-gradient neural solver}
\label{alg:slsg}
\begin{algorithmic}[1]
\Require \textbf{Inputs:}
\Statex \hspace{1.2em}$\bullet$ Quadrature grid $\GamN$; discrete BIE operator $A_N$, available through rowwise evaluation, and data $g_N$
\Statex \hspace{1.2em}$\bullet$ Network $\rho_\theta$; initial parameters $\theta^{\mathrm{init}}$
\Statex \hspace{1.2em}$\bullet$ Maximum epochs $M$; batch size $b$
\Statex \hspace{1.2em}$\bullet$ Learning-rate schedule $\{\eta_m\}_{m=1}^{M}$
\Statex \hspace{1.2em}$\bullet$ Stopping tolerance $\varepsilon$
\Ensure Trained parameters $\theta^{\mathrm{out}}$
\State $\theta \gets \theta^{\mathrm{init}}$
\For{$m = 1, 2, \ldots, M$}
    \State Draw a uniform random permutation $\pi$ of $\{1,\ldots,N\}$
    \State Set $K \gets \lceil N / b \rceil$
    \For{$k = 1, 2, \ldots, K$}
        \State $\mathcal{B}_k \gets
                \bigl\{\pi((k-1)b+1),\ldots,\pi(\min(kb, N))\bigr\}$
        \State $\hat g \gets \nabla_\theta L_{\mathcal{B}_k}(\theta)$
        \State $\theta \gets \mathrm{Adam}\bigl(\theta,\,\hat g;\, \eta_m\bigr)$
    \EndFor
    \If{$L(\theta;\GamN) \leq \varepsilon$}
        \State \textbf{break}
    \EndIf
\EndFor
\State \Return $\theta^{\mathrm{out}} \gets \theta$
\end{algorithmic}
\end{algorithm}

\begin{algorithm}[h]
\caption{\textsc{MLSG}: multilevel stochastic-gradient neural solver}
\label{alg:smt}
\begin{algorithmic}[1]
\Require \textbf{Inputs:}
\Statex \hspace{1.2em}$\bullet$ Grid ladder $\bigl\{(\Gam^\ell, W^\ell, A_{N_\ell}, g_{N_\ell})\bigr\}_{\ell=1}^{\ell_{\text{F}}}$
\Statex \hspace{1.2em}$\bullet$ Visiting schedule $\mathcal{S} = (s_1,\ldots,s_T)$ with $s_t \in \{1,\ldots,\ell_{\text{F}}\}$
\Statex \hspace{2.6em}(progressive: $\mathcal{S}=(1,\ldots,\ell_{\text{F}})$; cyclic: non-monotone)
\Statex \hspace{1.2em}$\bullet$ Network $\rho_\theta$; initial parameters $\theta^{(0)}$
\Statex \hspace{1.2em}$\bullet$ Per-stage epoch budgets $\{M_t\}_{t=1}^T$; batch size $b$
\Statex \hspace{1.2em}$\bullet$ Per-stage learning-rate schedules $\bigl\{\{\eta_m^{(t)}\}_{m=1}^{M_t}\bigr\}_{t=1}^T$
\Statex \hspace{1.2em}$\bullet$ Per-stage tolerances $\{\varepsilon_t\}_{t=1}^T$
\Ensure Trained parameters $\theta^{(T)}$
\For{$t = 1, 2, \ldots, T$}
    \State $\ell \gets s_t$
        \Comment{grid visited at stage $t$}
    \State $\theta^{(t)} \gets
            \textsc{SLSG}\bigl(\Gam^\ell,
                                  A_{N_\ell}, g_{N_\ell},
                                  \rho_\theta, \theta^{(t-1)},
                                  M_t, b,
                                  \{\eta_m^{(t)}\}_{m=1}^{M_t},
                                  \varepsilon_t\bigr)$
\EndFor
\State \Return $\theta^{(T)}$
\end{algorithmic}
\end{algorithm}


\paragraph{Tolerance schedule.}
Lemma~\ref{lem:res-to-err} converts the level loss into a bound on the
distance from the network sample to the discrete Nystr\"om solution, while
\eqref{eq:nystrom-conv} contributes a discretization error of order
$h_\ell^\nu$.  This discretization estimate alone suggests that there is
little benefit in driving the weighted level loss far below
$h_\ell^{2\nu}$.  The cross-level analysis in
Section~\ref{sec:Algs:analysis} refines this heuristic by including the
sampling regularity and yields the matched scale
$\varepsilon_\ell\asymp h_\ell^{2m}$ with $m=\min(s,\nu)$; see
\eqref{eq:matched-tolerance}.  Under a cyclic schedule, a revisited level may
therefore be assigned a tighter target when it is encountered again.

The remaining optimizer parameters can be tuned inexpensively on coarse
levels.  In the experiments we use simple prescribed per-stage learning-rate
and batch-size schedules, stated with each example.

\paragraph{Two scenarios at near-stationarity.}
Write $\theta$ for the current parameter, and suppose the level-$\ell$
optimization has run long enough that
\begin{equation}\label{eq:near-stationary}
    \nabla_\theta L^{(\ell)}(\theta)
    \;=\;
    \bigl(J_\theta^{(\ell)}\bigr)^{\ast}\,
        A_{N_\ell}^{\ast}\, r_\theta^{(\ell)}
    \;\approx\; 0 .
\end{equation}
A small parameter gradient has two qualitatively different explanations.
\textbf{(S1) Small residual:} $r_\theta^{(\ell)}\approx0$, so the iterate is
close to the discrete level solution.  \textbf{(S2) Tangent-space stall:}
the residual is not small, but $A_{N_\ell}^\ast r_\theta^{(\ell)}$ lies
approximately in $\ker(T_\theta^{(\ell)})$, equivalently in the
$W^\ell$-orthogonal complement of the sampled tangent space.  In~(S2) the
current parametrization is locally insensitive to the remaining residual.
Refinement changes both the sampled residual and tangent space, so it can
restore a nonzero gradient in either case, although this is a mechanism
rather than a guarantee.  The next subsection quantifies only the size of
the residual transferred across a refinement step.

\subsection{Cross-level residual transfer and balance of target loss tolerance}\label{sec:Algs:analysis}

The network parameters pass unchanged from level $\ell$ to level
$\ell+1$, but the discrete residual does not: both the Nystr\"om operator
and the quadrature nodes change.  A small level-$\ell$ loss therefore gives
no cross-level bound without additional control between the sampled nodes.
The following sampling assumption supplies that control.

Let $h_\ell$ denote the grid size
 of $\Gam^\ell$.  We make the
cross-level regularity needed for this step explicit.

\begin{assumption}[Uniform residual regularity and stable sampling]\label{ass:cross-level-sampling}
There is a normed function space $X(\Gam)$, an exponent $s>0$, and constants
$C_X,C_S,C_Q$ independent of the level such that every parameter encountered
in the multilevel run satisfies
\[
  \|r_\theta\|_{X(\Gam)}\le C_X,
\]
and the quadrature nodes satisfy the two sampling estimates
\begin{align}
  \|v\|_{L^2(\Gam)}^2
  &\le C_S\Bigl(\|\Pi_{\Gam^\ell}v\|_{W^\ell}^2
      +h_\ell^{2s}\|v\|_{X(\Gam)}^2\Bigr),
      \label{eq:sampling-ineq}\\
  \|\Pi_{\Gam^\ell}v\|_{W^\ell}^2
  &\le C_Q\Bigl(\|v\|_{L^2(\Gam)}^2
      +h_\ell^{2s}\|v\|_{X(\Gam)}^2\Bigr),
      \label{eq:sampling-stability}
\end{align}
for all $v\in X(\Gam)$.  The first inequality controls what can occur
between the nodes; the second prevents the discrete sampled norm from
amplifying a continuum residual without bound.  For quasi-uniform nodes and
Lipschitz functions one may take $s=1$ by cellwise estimates; higher-order
values of $s$ require corresponding higher-order sampling or zeros
estimates together with compatible quadrature estimates, and are not
inferred from smoothness alone.
\end{assumption}

Assume also that the level-$\ell$ Nystr\"om rule has consistency order
$\nu>0$ on the residuals generated during training.  Then
\begin{equation}\label{eq:consist-norm}
  \|\Pi_{\Gam^\ell}[r_\theta]-r_\theta^{(\ell)}\|_{W^\ell}
  \le C_1h_\ell^\nu.
\end{equation}
Combining \eqref{eq:sampling-ineq} and \eqref{eq:consist-norm}, with
$m:=\min(s,\nu)$, yields
\begin{equation}\label{eq:cont-residual-bound}
  \|r_\theta\|_{L^2(\Gam)}^2
  \le C\Bigl(\|r_\theta^{(\ell)}\|_{W^\ell}^2+h_\ell^{2m}\Bigr).
\end{equation}
If the level-$\ell$ stopping rule gives
$L^{(\ell)}(\theta^{(\ell,\mathrm{end})})\le\varepsilon_\ell$, then
$\|r_\theta^{(\ell)}\|_{W^\ell}^2\le2\varepsilon_\ell$, and
\eqref{eq:cont-residual-bound} yields
\[
  \|r_\theta\|_{L^2(\Gam)}^2
  \le C\bigl(\varepsilon_\ell+h_\ell^{2m}\bigr)
  \qquad\text{at }\theta=\theta^{(\ell,\mathrm{end})}.
\]
At the next level, operator consistency and
\eqref{eq:sampling-stability} give
\[
  \|r_\theta^{(\ell+1)}\|_{W^{\ell+1}}
  \le \|\Pi_{\Gam^{\ell+1}}r_\theta\|_{W^{\ell+1}}
       +C_1h_{\ell+1}^{\nu}.
\]
Using the uniform $X(\Gam)$ bound, the preceding continuum estimate, and
$h_{\ell+1}\le h_\ell$, we obtain
\begin{equation}\label{eq:warm-start-bound}
  L^{(\ell+1)}(\theta^{(\ell,\mathrm{end})})
  \le C\bigl(\varepsilon_\ell+h_\ell^{2m}\bigr).
\end{equation}
The constant is independent of the level as long as the consistency,
regularity, and two sampling constants above are uniform.  This estimate
applies to refinement steps from level $\ell$ to level $\ell+1$; the benefit
of the fine-to-coarse revisits in a cyclic schedule is empirical and is not
covered by this bound.

Estimate~\eqref{eq:warm-start-bound} has a direct algorithmic consequence.
Once $\varepsilon_\ell$ is appreciably smaller than $h_\ell^{2m}$, further
optimization on level $\ell$ cannot improve the guaranteed next-level warm
start by the same factor, because the cross-level discretization term
already dominates.  This motivates a matched schedule
\begin{equation}\label{eq:matched-tolerance}
  \varepsilon_\ell\asymp h_\ell^{2m}.
\end{equation}
If $h_\ell$
is reduced successively by a fixed factor, then the optimizer
enters each refined level within a level-independent constant factor of its
new target tolerance.

The norm estimate should be kept separate from the spectral interpretation
used to motivate MLSG.  Equations
\eqref{eq:cont-residual-bound}--\eqref{eq:warm-start-bound} control the
\emph{size} of the transferred residual; they do not prove that the
$h_\ell^{2m}$ contribution is localized in a particular band of NTK
eigenvectors.  The statement that the inherited component is largely
coarse and that refinement exposes a new spectral band is therefore an
empirical claim in this paper, tested by the projections in
Section~\ref{sec:Examples:spectral}.  Keeping these two statements separate
allows the cross-level estimate to remain valid without imposing an
unjustified Fourier decomposition on a general surface.

The estimate also does not by itself yield an end-to-end complexity theorem
for Adam.  Per-update cost depends on the number of sampled kernel
interactions, network forward/backward work, memory traffic, and device
utilization, while the number of updates required to reach a target loss is
optimizer- and problem-dependent.  We therefore report epoch counts and wall-clock behavior in
Section~\ref{sec:Examples}, rather than infer a formal linear-complexity
result from \eqref{eq:warm-start-bound}.

\section{Numerical examples}\label{sec:Examples}

The experiments serve two complementary purposes.  We first examine the
spectral mechanism directly, comparing single-level and multilevel training
through NTK spectra and modal residual diagnostics.  We then test the solver
on three classes of BIEs: interior Dirichlet Laplace/Poisson problems,
exterior Neumann Helmholtz problems, and an exterior Robin Laplace problem
on a hypersurface in $\R^4$.  Across these examples we use both parametric
and volumetric surface representations and both real- and complex-valued
densities.

\paragraph{Problem 1: Interior Dirichlet problem for Poisson's equation. }\label{para:poisson}
On a bounded domain $\Omega \subset \R^d$, $d \in \{2,3\}$, with $C^2$
boundary $\Gam = \partial\Omega$, find $u$ satisfying
\[
    \Delta u \;=\; f \quad \text{in } \Omega,
    \qquad
    u \;=\; g \quad \text{on } \Gam.
\]
Writing $u = u_p + v$ with $u_p$ a known particular solution of
$\Delta u_p = f$, the harmonic correction $v$ inherits
Dirichlet trace $g - u_p|_\Gam$. We represent $v$ as a double-layer
potential
\begin{equation}\label{eq:lap-rep}
    v(x) \;=\; \int_\Gam \frac{\partial G_{0, d}(x,y)}{\partial \mathbf{n}_y}\,
                          \rho(y)\,dS(y),
    \qquad x \in \Omega,
\end{equation}
with $G_{0, d}$ the Laplace fundamental solution
\[
    G_{0, d}(x,y) \;=\;
    \begin{cases}
        -\tfrac{1}{2\pi}\,\log|x-y|, & d = 2,\\[2pt]
        \tfrac{1}{4\pi}\,|x-y|^{-1}, & d = 3.
    \end{cases}
\]
The interior-Dirichlet jump relation~\cite{kress2014linear}
then yields the second-kind BIE with kernel
\begin{equation}\label{eq:lap-kernel}
    k_{\mathrm{Lap}}(x,y) \;=\; -\frac{\partial G_{0, d}(x,y)}{\partial \mathbf{n}_y}.
\end{equation}


\paragraph{Problem 2: Exterior Neumann problem for the Helmholtz equation.}\label{para:exhelm}
On the exterior $\R^3 \setminus \overline\Omega$, find $u$ satisfying
\[
    \Delta u + \kappa^2 u \;=\; 0
        \quad \text{in } \R^3 \setminus \overline\Omega,
    \qquad
    \frac{\partial u}{\partial n} \;=\; g \quad \text{on } \Gam,
\]
together with the Sommerfeld radiation condition at infinity. We seek
$u$ as a single-layer potential
\begin{equation}\label{eq:helm-rep}
    u(x) \;=\; \int_\Gam G_{\kappa}(x,y)\,\rho(y)\,dS(y),
    \qquad x \in \R^3 \setminus \overline\Omega,
\end{equation}
with the radiating Helmholtz Green's function
\[
    G_{\kappa}(x,y) \;=\; \frac{e^{i\kappa|x-y|}}{4\pi\,|x-y|}.
\]
Taking the normal derivative on the exterior side and applying the jump
relation gives the second-kind BIE \eqref{eq:BIE} with right-hand side
proportional to $g$ and kernel
\begin{equation}\label{eq:helm-kernel}
    k_{\mathrm{Helm}}(x,y)
    \;=\;
    -\frac{\partial G_{\kappa}(x,y)}{\partial \mathbf{n}_x}.
\end{equation}
The sign of the jump is opposite to that in \eqref{eq:BIE}, which the
convention there absorbs.
\paragraph{Caveat: spurious resonance.}
The single-layer formulation \eqref{eq:helm-rep}--\eqref{eq:helm-kernel}
of the exterior Neumann Helmholtz problem loses unique solvability at
wavenumbers $\kappa$ for which $\kappa^2$ is an interior Dirichlet
eigenvalue of $-\Delta$ on $\Omega$, and is ill-conditioned for $\kappa$
near such a resonance. 
The combined-field formulation of Burton and
Miller~\cite{burton1971application} restores uniqueness.  We do not pursue
that extension here; all reported wavenumbers are chosen away from the
interior Dirichlet spectrum of $\Omega$.

\paragraph{Problem 3: Exterior Robin problem for the Laplace equation in $\mathbb R^4$.}\label{para:br}
On the exterior $D := \R^4 \setminus \overline\Omega$, find $\psi$
satisfying
\[
    \Delta \psi \;=\; 0 \quad \text{in } D,
    \qquad
    \frac{\partial \psi}{\partial \mathbf{n}} + \beta\,\psi \;=\; 0
        \quad \text{on } \Gam,
    \qquad
    \psi(x) \;=\; 1 + \mathcal{O}\bigl(|x|^{-2}\bigr)
        \;\text{ as } |x| \to \infty,
\]
with $\mathbf{n}$ the unit normal pointing out of $\Omega$ and
$\beta \in C(\Gam)$ the Robin coefficient.

We represent $ u = \psi -1 $ by the combined layer potential
\begin{equation}\label{eq:ring-rep}
    u(x) \;=\; \int_\Gam
    \left[
        \beta(y)\,G_{0, 4}(x,y)
        + \frac{\partial G_{0, 4}(x,y)}{\partial \mathbf{n}_y}
    \right] \rho(y)\,dS(y),
    \qquad x \in D,
\end{equation}
with the Laplace fundamental solution in $\R^4$
\[
    G_{0, 4}(x,y) \;=\; \frac{1}{4\pi^2\,|x-y|^{2}}.
\]
Green's representation of $u$ together with the jump
relations~\cite{kress2014linear} gives a second-kind BIE for the
boundary trace $\rho := \psi|_\Gam$ of the form \eqref{eq:BIE} with
right-hand side $g \equiv 1$ and kernel
\begin{equation}\label{eq:ring-kernel}
    k_{\mathrm{Lap},\R^4}(x,y)
    \;=\;-\left(
    \beta(y)\,G_{0, 4}(x,y)
    + \frac{\partial G_{0, 4}(x,y)}{\partial \mathbf{n}_y}\right).
\end{equation}

Both terms are $\mathcal O(|x-y|^{-2})$ near the diagonal.  This
singularity is integrable over a three-dimensional surface, so
$k_{\mathrm{Lap},\R^4}$ is weakly singular.  The density is real-valued,
and hence no realification is needed.  The geometric origin of the problem
and the specific ring configuration are described in
Section~\ref{sec:Examples:blackring}.

Across the numerical examples we use two representations of the density,
chosen according to how the surface is specified.  The full discretizations
are given in Appendix~\ref{app:surf-disc}; here we record only the
representation and quadrature structure needed to interpret the experiments.

\textit{Parametric representation $\rho_\theta^{\mathrm{param}}$.}

If $\Gam$ is described by a finite atlas
$\{(D_\alpha,\varphi_\alpha)\}_{\alpha=1}^{M_{\mathrm{chart}}}$,
with $\varphi_\alpha:D_\alpha\subset\R^{d-1}\to
\Gam_\alpha\subset\Gam$ and
$\Gam=\bigcup_\alpha\Gam_\alpha$, we assign one MLP to each chart and
write the resulting density representation as

\[
    \rho_\theta^{\mathrm{param}}
    \;=\; \bigl(\rho_{\theta(1)},\,\ldots,\,\rho_{\theta(M_{\mathrm{chart}})}\bigr),
    \qquad
    \rho_{\theta(\alpha)} \,:\, D_\alpha \to \C,
    \quad \alpha = 1, \ldots, M_{\mathrm{chart}},
\]

where $\theta(\alpha)\in\R^{p_\alpha}$ denotes the parameter block for
chart $\alpha$, and
$\theta=(\theta(1),\ldots,\theta(M_{\mathrm{chart}}))$.  Parentheses in
$\theta(\alpha)$ distinguish the chart index from the component index
$\theta_k$ used in Section~\ref{sec:loss}.  The corresponding surface
density is defined chartwise by

\[
    \rho(\varphi_\alpha(\xi)) \;=\; \rho_{\theta(\alpha)}(\xi),
    \qquad \xi \in D_\alpha,
\]

The data are pulled back as $g_\alpha=g\circ\varphi_\alpha$.  Each
surface quadrature node is an image
$x_j=\varphi_\alpha(\xi_j^\alpha)$ of a parameter-domain node.  With the
nodes ordered by chart, $A_N$ has an
$M_{\mathrm{chart}}\times M_{\mathrm{chart}}$ block structure: diagonal
blocks contain the singular same-chart interactions and use the corrected
quadrature, whereas off-diagonal blocks contain smooth interactions between
distinct chart images.  MLSG acts on the concatenated parameter vector, with
the residual loss accumulated over all charts.  Appendix~\ref{app:surf-disc:param}
gives the resulting discrete equations and the two-chart block system.

\textit{Volumetric (IBIM) representation $\rho_\theta^{\mathrm{IBIM}}$.}

If $\Gam$ is given as the zero level set of a signed-distance function
$d_\Gam$, we use the implicit boundary integral method (IBIM)
of~\cite{KublikTanushevTsai2013, ChenTsai2017} and represent the density by
a single ambient MLP,

\[
    \rho_\theta^{\mathrm{IBIM}} : \R^d \to \C,
\]

whose trace supplies the boundary density.  The surface nodes are the
closest-point projections $x_j=P_\Gam(z_j)$ of Cartesian grid nodes lying
in the tube
$T_\epsilon=\{x\in\R^d:|d_\Gam(x)|<\epsilon\}$; the quadrature weights
contain the regularized one-dimensional delta factor in $d_\Gam$.
Appendix~\ref{app:surf-disc:ibim} gives the tubular quadrature and the
corresponding discrete operator $K_N$.

\paragraph{Neural network architecture.}

Unless stated otherwise, we represent the density by a multilayer
perceptron with sinusoidal activations and SIREN
initialization~\cite{Sitzmann20}.  The first layer uses \begin{revision}
$\sin(\omega_0(W_0x+b_0))$\end{revision}with weights drawn from
$\mathcal U[-1/d,1/d]$, where $d$ is the input dimension; subsequent
layers use $\sin(Wx+b)$ with weights drawn from
$\mathcal U[-\sqrt{6/m},\sqrt{6/m}]$, where $m$ is the layer width.
For complex-valued densities the network has two output channels, for the
real and imaginary parts.  The depth, width, and value of $\omega_0$ are
reported for each experiment.  Unless specified otherwise, training uses
Adam~\cite{kingma2015adam} with initial learning rate $10^{-3}$, multiplied by $0.9$ every
$200$ epochs, mini-batches of $8192$ rows, and weight decay $10^{-6}$.

\paragraph{Practical loss implementation.}

The quadrature enters both the Nystr\"om operator and the outer
least-squares loss.  The source quadrature weights are retained in $A_N$, so
each collocation residual is the Nystr\"om residual defined in
\eqref{eq:res-theta}; the same target quadrature weights define the discrete
$L^2$ residual loss \eqref{eq:loss-level}.  Thus the optimization objective is simply the deterministic quadrature
discretization of the continuum residual loss.

For the parametric discretizations this is the usual surface quadrature.  For
IBIM, the weights in \eqref{eq:ibim-nodes} arise by applying the Cartesian
trapezoidal rule to the exact tubular-neighborhood representation of the
surface integral; see Appendix~\ref{app:surf-disc:ibim}.  Consequently, the
IBIM loss is simply the corresponding discretized volume integral.  The
mini-batch loss \eqref{eq:minibatch-loss} is a stochastic estimator of this
fixed deterministic quadrature loss: row sampling changes the optimization
procedure, not the underlying objective.

Retaining the quadrature weights also gives the level-independent NTK
scaling discussed in Section~\ref{sec:Algs:mg-analogy}.  In particular, for
quasi-uniform surface quadratures the weights scale as $N^{-1}$, while in
IBIM the tubular weights provide the corresponding geometric scaling.
Omitting these weights would replace the Nystr\"om-scaled tangent operator
by an unscaled Gram matrix and introduce an artificial dependence on the
number and distribution of quadrature nodes.

\paragraph{Computational facilities and timing convention.}

All neural networks are implemented in PyTorch and trained on a single
NVIDIA H200 GPU provided by the National Center for High-Performance
Computing (NCHC), National Institutes of Applied Research (NIAR), Taiwan.
The reference dense solves use \texttt{gmres()}~\cite{saad1986gmres} in MATLAB R2024b on the
NCHC Forerunner~1 system with 64 Intel Xeon Platinum 8480+ CPU cores and
275.2GB RAM; the restart parameter is $m=60$.

\subsection{Spectral bias and effective tangent-space expansion}
\label{sec:Examples:spectral}

We test the spectral picture of Sections~\ref{sec:Algs:mg-analogy}
and~\ref{sec:Algs:expansion} directly.  The diagnostics ask three questions:
whether the empirical NTK has a small leading band of effective directions,
whether refinement reveals additional directions above the effective
threshold, and how the unresolved residual is distributed relative to those
directions before and after refinement.  These are empirical questions; on
nonnested grids the effective spaces need not be algebraically nested.

For the spectral diagnostic we use the Laplace equation on a two-dimensional
flower-shaped domain with boundary parametrization

\begin{equation}\label{eq:flower}
    x(s) = A(s)\cos s,
    y(s) = A(s)\sin s,
    ~A(s) = \dfrac{\sqrt{5}}{50} + \bigl(1 + 0.1\sin(4s)\bigr),~s \in [0,2\pi).
\end{equation}

and Dirichlet data

\begin{equation}\label{eq:g}
    g\bigl(x(s),y(s)\bigr) \;=\; y(s)\bigl(1 - \sin({150} s)\bigr),
    \qquad s \in [0,2\pi).
\end{equation}

The density is represented by a single MLP with Cartesian input
$(x(s),y(s))$, $5$ hidden layers of $200$ neurons, cosine activations, and
Kaiming initialization.

The grid ladder has five levels,
$N_\ell=2^{6+\ell}$ for $\ell=1,\ldots,5$, ranging from $128$ to
$2048$ quadrature points; throughout the ladder,
$\kappa_W\approx{\color{black}1.11}$.  The stopping tolerance starts at
$1.6\times10^{-4}$ and is halved at each refinement, reaching
$10^{-5}$ on the finest level.  The multilevel runs consistently reach this
finest-level target.  By contrast, single-level training on the same
$2048$-point grid remains above the target after $20{,}000$ optimizer
iterations; see Figure~\ref{fig:single_level}.


\paragraph{Spectral expansion across levels.}

Figures~\ref{fig:ntk_multi_initial} and~\ref{fig:ntk_multi_final}
report the NTK spectra at the beginning and end of Levels~2--5, aggregated
over $20$ independent runs.  The upper rows show the spectrum of
$T_{\theta,N_\ell}$, and the lower rows show the residual coefficients
$\zeta_j$ from \eqref{eq:res-proj-coords}.  Two patterns are relevant:
refinement reveals additional directions above the effective threshold, and
a non-negligible part of the unresolved residual lies in directions that are
effective on the refined grid.  This is empirical support for the cross-level
spectral interpretation, not a consequence of the norm estimate in
Section~\ref{sec:Algs:analysis}; nor should the plotted coefficients be read
as the exact modal loss-decay rates of \eqref{eq:proj-flow}.

\begin{figure}[p]
\centering
\begin{subfigure}[t]{0.24\linewidth}
    \centering
    \includegraphics[width=\linewidth]{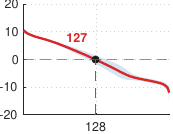}
    \caption{{Level 2}}
\end{subfigure}
\hfill
\begin{subfigure}[t]{0.24\linewidth}
    \centering
    \includegraphics[width=\linewidth]{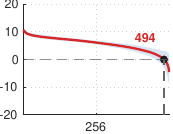}
    \caption{{Level 3}}
\end{subfigure}
\hfill
\begin{subfigure}[t]{0.24\linewidth}
    \centering
    \includegraphics[width=\linewidth]{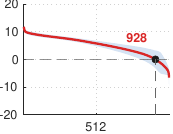}
    \caption{{Level 4}}
\end{subfigure}
\hfill
\begin{subfigure}[t]{0.24\linewidth}
    \centering
    \includegraphics[width=\linewidth]{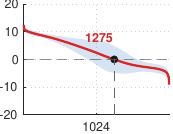}
    \caption{{Level 5}}
\end{subfigure}
\vspace{1.em}
\begin{subfigure}[t]{0.24\linewidth}
    \centering
    \includegraphics[width=\linewidth]{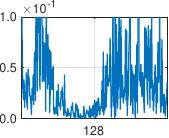}
    \caption{Level 2\\(Iteration 0)}
\end{subfigure}
\hfill
\begin{subfigure}[t]{0.24\linewidth}
    \centering
    \includegraphics[width=\linewidth]{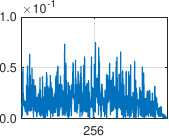}
    \caption{Level 3\\(Iteration 0)}
\end{subfigure}
\hfill
\begin{subfigure}[t]{0.24\linewidth}
    \centering
    \includegraphics[width=\linewidth]{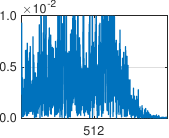}
    \caption{Level 4\\(Iteration 0)}
\end{subfigure}
\hfill
\begin{subfigure}[t]{0.24\linewidth}
    \centering
    \includegraphics[width=\linewidth]{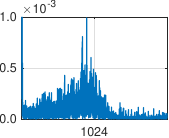}
    \caption{Level 5\\(Iteration 0)}
\end{subfigure}
\caption{Initial NTK spectral distributions at Levels~2--5,
aggregated over 20 independent runs. Top: NTK spectra. Bottom: the residual
$r_{\theta,N}$ projected onto the corresponding NTK eigenmodes.}
\label{fig:ntk_multi_initial}

\vspace{2.em}

\begin{subfigure}[t]{0.24\linewidth}
    \centering
    \includegraphics[width=\linewidth]{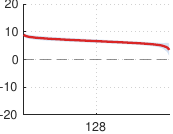}
    \caption{Level 2}
\end{subfigure}
\hfill
\begin{subfigure}[t]{0.24\linewidth}
    \centering
    \includegraphics[width=\linewidth]{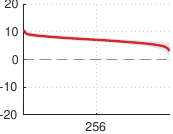}
    \caption{Level 3}
\end{subfigure}
\hfill
\begin{subfigure}[t]{0.24\linewidth}
    \centering
    \includegraphics[width=\linewidth]{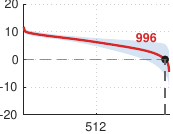}
    \caption{Level 4}
\end{subfigure}
\hfill
\begin{subfigure}[t]{0.24\linewidth}
    \centering
    \includegraphics[width=\linewidth]{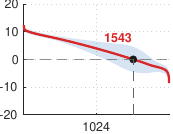}
    \caption{Level 5}
\end{subfigure}
\vspace{1.em}
\begin{subfigure}[t]{0.24\linewidth}
    \centering
    \includegraphics[width=\linewidth]{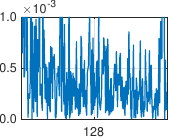}
    \caption{Level 2\\(Iteration 521)}
\end{subfigure}
\hfill
\begin{subfigure}[t]{0.24\linewidth}
    \centering
    \includegraphics[width=\linewidth]{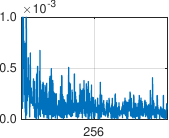}
    \caption{Level 3\\(Iteration 512)}
\end{subfigure}
\hfill
\begin{subfigure}[t]{0.24\linewidth}
    \centering
    \includegraphics[width=\linewidth]{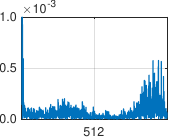}
    \caption{Level 4\\(Iteration 338)}
\end{subfigure}
\hfill
\begin{subfigure}[t]{0.24\linewidth}
    \centering
    \includegraphics[width=\linewidth]{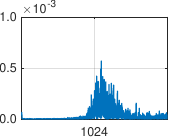}
    \caption{Level 5\\(Iteration 443)}
\end{subfigure}
\caption{NTK spectral distributions upon termination at Levels~2--5, aggregated
over 20 independent runs, under the same setup and layout as the
initial state shown in Fig.~\ref{fig:ntk_multi_initial}. Comparison of
the two reveals how the spectra and projected residuals evolve from
initialization (Iteration~0) to convergence (Iterations 521, 512,
338, and 443 for Levels~2--5, respectively). Top: NTK spectra.
Bottom: the residual $r_{\theta,N}$ projected onto the corresponding NTK
eigenmodes.}
\label{fig:ntk_multi_final}
\end{figure}
\begin{figure}[htbp]
    \centering
    \begin{subfigure}[b]{0.24\linewidth}
        \includegraphics[width=\linewidth]{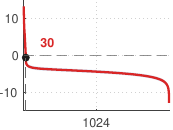}
    \end{subfigure}
    \hfill
    \begin{subfigure}[b]{0.24\linewidth}
        \includegraphics[width=\linewidth]{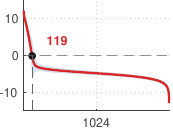}
    \end{subfigure}
    \hfill
    \begin{subfigure}[b]{0.24\linewidth}
        \includegraphics[width=\linewidth]{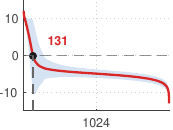}
    \end{subfigure}
    \hfill
    \begin{subfigure}[b]{0.24\linewidth}
        \includegraphics[width=\linewidth]{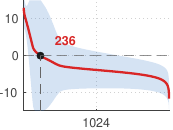}
    \end{subfigure}

    \vspace{0.5em}

    \begin{subfigure}[b]{0.24\linewidth}
        \includegraphics[width=\linewidth]{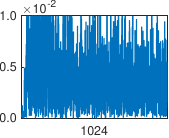}
        \caption{Iteration 1}
        \label{fig:single_level_col1}
    \end{subfigure}
    \hfill
    \begin{subfigure}[b]{0.24\linewidth}
        \includegraphics[width=\linewidth]{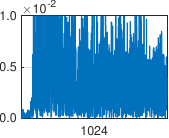}
        \caption{Iteration 1000}
        \label{fig:single_level_col2}
    \end{subfigure}
    \hfill
    \begin{subfigure}[b]{0.24\linewidth}
        \includegraphics[width=\linewidth]{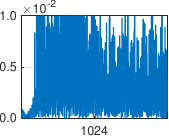}
        \caption{Iteration 2000}
        \label{fig:single_level_col3}
    \end{subfigure}
    \hfill
    \begin{subfigure}[b]{0.24\linewidth}
        \includegraphics[width=\linewidth]{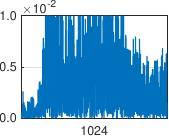}
        \caption{Iteration 20000}
        \label{fig:single_level_col4}
    \end{subfigure}

    \caption{Single-level training on the finest grid: NTK spectra at iterations 1, 1000, 2000, and the final step. Top: NTK spectra. Bottom: the residual $r_{\theta,N}$ projected onto the corresponding NTK eigenmodes. Loss at iteration 20,000 is $2.37\times 10^{-1}$.}
    \label{fig:single_level}
\end{figure}

\paragraph{Progressive versus cyclic schedules.}

For this schedule-comparison study we use the four-level ladder
$N=256,512,1024,2048$.  Table~\ref{tab:1d_cyclic_progressive} compares the
two visiting schedules of Section~\ref{sec:Algs:smt} over $50$ independent
runs. The two schedules reach comparable final effective tangent spaces.  Their
main difference is the optimization effort required to do so: the cyclic
schedule has a markedly smaller run-to-run spread in the required epoch
count.

\begin{table}[t]
\centering
\caption{{Epoch counts by training level for the progressive and cyclic schedules over 50 independent runs.  The first three levels ($N=256,512,1024$) are shared.  For the cyclic schedule, the revisited levels~3 and~4 are shown for three prescribed epoch budgets at the first visit to level~4 (150, 200, and 500 epochs; light to dark blue).  The progressive schedule is shown in orange.  Boxes span the interquartile range with the median marked; whiskers extend to $1.5\times\mathrm{IQR}$, with more extreme values shown as outliers.}}
\label{tab:1d_cyclic_progressive}
\medskip
\begin{tabular}{c r @{\hspace{2em}} r @{\hspace{4em}} c}
\toprule
Level & $N$ & $\varepsilon$ & \#\,epochs \\
\cmidrule(lr){4-4}
 & & & \begin{tikzpicture}[x=8cm, y=1ex, baseline=0.5ex]\draw[axisline] (0,0) -- (1,0);\draw[axisline] (0.1250,0) -- (0.1250,0.8);\node[tick] at (0.1250,2.6) {1000};\draw[axisline] (0.2500,0) -- (0.2500,0.8);\node[tick] at (0.2500,2.6) {2000};\draw[axisline] (0.3750,0) -- (0.3750,0.8);\node[tick] at (0.3750,2.6) {3000};\draw[axisline] (0.5000,0) -- (0.5000,0.8);\node[tick] at (0.5000,2.6) {4000};\draw[axisline] (0.6250,0) -- (0.6250,0.8);\node[tick] at (0.6250,2.6) {5000};\draw[axisline] (0.7500,0) -- (0.7500,0.8);\node[tick] at (0.7500,2.6) {6000};\draw[axisline] (0.8750,0) -- (0.8750,0.8);\node[tick] at (0.8750,2.6) {7000};\draw[axisline] (1.0000,0) -- (1.0000,0.8);\node[tick] at (1.0000,2.6) {8000};\end{tikzpicture} \\[2pt]
\midrule
\multicolumn{4}{c}{\textbf{(a) Shared levels}} \\
\midrule
1 & 256 & $10^{-2}$ & \begin{tikzpicture}[x=8cm, y=1ex, baseline=-0.5ex]\draw[axisline] (0,0) -- (1,0);\draw[bxp] (0.0842,0) -- (0.1202,0);\draw[bxp] (0.1488,0) -- (0.1909,0);\draw[bxp] (0.0842,-0.7) -- (0.0842,0.7);\draw[bxp] (0.1909,-0.7) -- (0.1909,0.7);\filldraw[boxfill,draw=black,line width=0.5pt] (0.1202,-1.1) rectangle (0.1488,1.1);\draw[medline] (0.1281,-1.1) -- (0.1281,1.1);\draw[outl,boxfill] (0.0744,0) circle (1.4pt);\draw[outl,boxfill] (0.2162,0) circle (1.4pt);\draw[outl,boxfill] (0.2369,0) circle (1.4pt);\end{tikzpicture} \\[1pt]
2 & 512 & $10^{-5}$ & \begin{tikzpicture}[x=8cm, y=1ex, baseline=-0.5ex]\draw[axisline] (0,0) -- (1,0);\draw[bxp] (0.0162,0) -- (0.0262,0);\draw[bxp] (0.0462,0) -- (0.0690,0);\draw[bxp] (0.0162,-0.7) -- (0.0162,0.7);\draw[bxp] (0.0690,-0.7) -- (0.0690,0.7);\filldraw[boxfill,draw=black,line width=0.5pt] (0.0262,-1.1) rectangle (0.0462,1.1);\draw[medline] (0.0369,-1.1) -- (0.0369,1.1);\draw[outl,boxfill] (0.0916,0) circle (1.4pt);\end{tikzpicture} \\[1pt]
3 & 1\,024 & $10^{-5}$ & \begin{tikzpicture}[x=8cm, y=1ex, baseline=-0.5ex]\draw[axisline] (0,0) -- (1,0);\draw[bxp] (0.0303,0) -- (0.0414,0);\draw[bxp] (0.0613,0) -- (0.0903,0);\draw[bxp] (0.0303,-0.7) -- (0.0303,0.7);\draw[bxp] (0.0903,-0.7) -- (0.0903,0.7);\filldraw[boxfill,draw=black,line width=0.5pt] (0.0414,-1.1) rectangle (0.0613,1.1);\draw[medline] (0.0514,-1.1) -- (0.0514,1.1);\draw[outl,boxfill] (0.0992,0) circle (1.4pt);\draw[outl,boxfill] (0.0995,0) circle (1.4pt);\draw[outl,boxfill] (0.1082,0) circle (1.4pt);\end{tikzpicture} \\[1pt]
\midrule
\multicolumn{4}{c}{\textbf{(b) Cyclic schedule}} \\
\midrule
4 & 2\,048 & -- & \begin{tikzpicture}[x=8cm, y=1ex, baseline=-0.5ex]\draw[axisline] (0,2.9) -- (1,2.9);\draw[axisline] (0,0) -- (1,0);\draw[axisline] (0,-2.9) -- (1,-2.9);\draw[bxp] (0.0188,2.9) -- (0.0188,2.9);\draw[bxp] (0.0188,2.9) -- (0.0188,2.9);\draw[bxp] (0.0188,2.2) -- (0.0188,3.6);\draw[bxp] (0.0188,2.2) -- (0.0188,3.6);\filldraw[lbluefill,draw=black,line width=0.5pt] (0.0188,1.8) rectangle (0.0188,4);\draw[medline] (0.0188,1.8) -- (0.0188,4);\draw[bxp] (0.0250,0) -- (0.0250,0);\draw[bxp] (0.0250,0) -- (0.0250,0);\draw[bxp] (0.0250,-0.7) -- (0.0250,0.7);\draw[bxp] (0.0250,-0.7) -- (0.0250,0.7);\filldraw[mbluefill,draw=black,line width=0.5pt] (0.0250,-1.1) rectangle (0.0250,1.1);\draw[medline] (0.0250,-1.1) -- (0.0250,1.1);\draw[bxp] (0.0625,-2.9) -- (0.0625,-2.9);\draw[bxp] (0.0625,-2.9) -- (0.0625,-2.9);\draw[bxp] (0.0625,-3.6) -- (0.0625,-2.2);\draw[bxp] (0.0625,-3.6) -- (0.0625,-2.2);\filldraw[dbluefill,draw=black,line width=0.5pt] (0.0625,-4) rectangle (0.0625,-1.8);\draw[medline] (0.0625,-4) -- (0.0625,-1.8);\end{tikzpicture} \\[10pt]
\textbf{3} & 1\,024 & $10^{-6}$ & \begin{tikzpicture}[x=8cm, y=1ex, baseline=-0.5ex]\draw[axisline] (0,2.9) -- (1,2.9);\draw[axisline] (0,0) -- (1,0);\draw[axisline] (0,-2.9) -- (1,-2.9);\draw[bxp] (0.0766,2.9) -- (0.1255,2.9);\draw[bxp] (0.1860,2.9) -- (0.2684,2.9);\draw[bxp] (0.0766,2.2) -- (0.0766,3.6);\draw[bxp] (0.2684,2.2) -- (0.2684,3.6);\filldraw[lbluefill,draw=black,line width=0.5pt] (0.1255,1.8) rectangle (0.1860,4);\draw[medline] (0.1341,1.8) -- (0.1341,4);\draw[outl,lbluefill] (0.3765,2.9) circle (1.4pt);\draw[bxp] (0.0736,0) -- (0.1241,0);\draw[bxp] (0.1815,0) -- (0.2560,0);\draw[bxp] (0.0736,-0.7) -- (0.0736,0.7);\draw[bxp] (0.2560,-0.7) -- (0.2560,0.7);\filldraw[mbluefill,draw=black,line width=0.5pt] (0.1241,-1.1) rectangle (0.1815,1.1);\draw[medline] (0.1323,-1.1) -- (0.1323,1.1);\draw[outl,mbluefill] (0.2681,0) circle (1.4pt);\draw[outl,mbluefill] (0.3791,0) circle (1.4pt);\draw[bxp] (0.0725,-2.9) -- (0.1100,-2.9);\draw[bxp] (0.1578,-2.9) -- (0.2080,-2.9);\draw[bxp] (0.0725,-3.6) -- (0.0725,-2.2);\draw[bxp] (0.2080,-3.6) -- (0.2080,-2.2);\filldraw[dbluefill,draw=black,line width=0.5pt] (0.1100,-4) rectangle (0.1578,-1.8);\draw[medline] (0.1235,-4) -- (0.1235,-1.8);\draw[outl,dbluefill] (0.2359,-2.9) circle (1.4pt);\draw[outl,dbluefill] (0.2491,-2.9) circle (1.4pt);\draw[outl,dbluefill] (0.2504,-2.9) circle (1.4pt);\draw[outl,dbluefill] (0.2504,-2.9) circle (1.4pt);\draw[outl,dbluefill] (0.2543,-2.9) circle (1.4pt);\draw[outl,dbluefill] (0.2581,-2.9) circle (1.4pt);\draw[outl,dbluefill] (0.3046,-2.9) circle (1.4pt);\end{tikzpicture} \\[10pt]
\textbf{4} & 2\,048 & $5\times10^{-7}$ & \begin{tikzpicture}[x=8cm, y=1ex, baseline=-0.5ex]\draw[axisline] (0,2.9) -- (1,2.9);\draw[axisline] (0,0) -- (1,0);\draw[axisline] (0,-2.9) -- (1,-2.9);\draw[bxp] (0.2504,2.9) -- (0.3569,2.9);\draw[bxp] (0.4383,2.9) -- (0.5549,2.9);\draw[bxp] (0.2504,2.2) -- (0.2504,3.6);\draw[bxp] (0.5549,2.2) -- (0.5549,3.6);\filldraw[lbluefill,draw=black,line width=0.5pt] (0.3569,1.8) rectangle (0.4383,4);\draw[medline] (0.4030,1.8) -- (0.4030,4);\draw[outl,lbluefill] (0.5865,2.9) circle (1.4pt);\draw[outl,lbluefill] (0.6057,2.9) circle (1.4pt);\draw[outl,lbluefill] (0.6208,2.9) circle (1.4pt);\draw[bxp] (0.2429,0) -- (0.3573,0);\draw[bxp] (0.4374,0) -- (0.5403,0);\draw[bxp] (0.2429,-0.7) -- (0.2429,0.7);\draw[bxp] (0.5403,-0.7) -- (0.5403,0.7);\filldraw[mbluefill,draw=black,line width=0.5pt] (0.3573,-1.1) rectangle (0.4374,1.1);\draw[medline] (0.3990,-1.1) -- (0.3990,1.1);\draw[outl,mbluefill] (0.5944,0) circle (1.4pt);\draw[outl,mbluefill] (0.6060,0) circle (1.4pt);\draw[outl,mbluefill] (0.6222,0) circle (1.4pt);\draw[bxp] (0.2310,-2.9) -- (0.3491,-2.9);\draw[bxp] (0.4337,-2.9) -- (0.5317,-2.9);\draw[bxp] (0.2310,-3.6) -- (0.2310,-2.2);\draw[bxp] (0.5317,-3.6) -- (0.5317,-2.2);\filldraw[dbluefill,draw=black,line width=0.5pt] (0.3491,-4) rectangle (0.4337,-1.8);\draw[medline] (0.3873,-4) -- (0.3873,-1.8);\draw[outl,dbluefill] (0.5772,-2.9) circle (1.4pt);\draw[outl,dbluefill] (0.5997,-2.9) circle (1.4pt);\draw[outl,dbluefill] (0.6235,-2.9) circle (1.4pt);\end{tikzpicture} \\[10pt]
\midrule
\multicolumn{4}{c}{\textbf{(c) Progressive schedule}} \\
\midrule
4 & 2\,048 & $5\times10^{-7}$ & \begin{tikzpicture}[x=8cm, y=1ex, baseline=-0.5ex]\draw[axisline] (0,0) -- (1,0);\draw[bxp] (0.2910,0) -- (0.4548,0);\draw[bxp] (0.5759,0) -- (0.7468,0);\draw[bxp] (0.2910,-0.7) -- (0.2910,0.7);\draw[bxp] (0.7468,-0.7) -- (0.7468,0.7);\filldraw[orangefill,draw=black,line width=0.5pt] (0.4548,-1.1) rectangle (0.5759,1.1);\draw[medline] (0.5012,-1.1) -- (0.5012,1.1);\draw[outl,orangefill] (0.7809,0) circle (1.4pt);\draw[outl,orangefill] (0.7967,0) circle (1.4pt);\draw[outl,orangefill] (0.9470,0) circle (1.4pt);\end{tikzpicture} \\[1pt]
\addlinespace[3pt]
\midrule
\multicolumn{4}{@{}l@{}}{\makebox[70mm][l]{\textbf{Final-level median (cyclic 150/200/500)}}\textbf{3224 / 3192 / 3099}} \\
\multicolumn{4}{@{}l@{}}{\makebox[70mm][l]{\textbf{Final-level median (progressive)}}\textbf{4010}} \\
\addlinespace[2pt]
\multicolumn{4}{c}{\begin{tikzpicture}[x=1cm, y=1ex, baseline=-0.5ex]\filldraw[boxfill,draw=black,line width=0.4pt] (0.00,-0.8) rectangle (0.28,0.8);\node[leg,anchor=west] at (0.36,0) {shared};\filldraw[lbluefill,draw=black,line width=0.4pt] (1.61,-0.8) rectangle (1.89,0.8);\node[leg,anchor=west] at (1.97,0) {cyclic (L4, 150)};\filldraw[mbluefill,draw=black,line width=0.4pt] (4.77,-0.8) rectangle (5.05,0.8);\node[leg,anchor=west] at (5.13,0) {cyclic (L4, 200)};\filldraw[dbluefill,draw=black,line width=0.4pt] (7.93,-0.8) rectangle (8.21,0.8);\node[leg,anchor=west] at (8.29,0) {cyclic (L4, 500)};\filldraw[orangefill,draw=black,line width=0.4pt] (11.09,-0.8) rectangle (11.37,0.8);\node[leg,anchor=west] at (11.45,0) {progressive};\end{tikzpicture}} \\
\bottomrule
\end{tabular}
\end{table}
\subsection{Poisson's equation in three dimensions}
\label{sec:Examples:poisson}

We next solve a Poisson problem on the Stanford bunny, a nonconvex
three-dimensional surface with regions of high curvature.  The geometry is
specified nonparametrically through a signed-distance representation, and
the density is represented volumetrically using the IBIM formulation of
Appendix~\ref{app:surf-disc:ibim}.  This example tests the method on a
surface for which a convenient global parametrization is unavailable.

We use IBIM to solve
\begin{equation}\label{eq:bunny-prob}
    \Delta u \;=\; 1 \quad \text{in } \Omega,
    \qquad
    u \;=\; \tfrac{1}{6}\bigl(x^2+y^2+z^2\bigr) \quad \text{on } \partial\Omega,
\end{equation}
whose exact solution is the boundary datum extended into $\Omega$, so
the relative $L^2$ error in the reconstructed $u$ admits a closed-form
reference.

The geometry is supplied as a dense point cloud.  On Cartesian nodes near
the point cloud, local piecewise-quadratic interpolation provides the
distance values and Jacobian factors required by IBIM.  The density network
has $10$ hidden layers of $500$ neurons and uses SIREN initialization.

Table~\ref{tab:bunny_training_strategy} reports the multilevel training
statistics over $20$ independent runs.  We also compare the learned
solution with the GMRES reference.  Both are evaluated using the squared
relative $L^2$ error with target $10^{-3}$, while
Figure~\ref{fig:bunny_results} shows the surface densities and pointwise
errors at sampled interior points.

\begin{table}[!h]
    \centering
    \caption{{MLSG training statistics for the Stanford bunny example under the progressive schedule.  Means and standard deviations of the epoch counts and wall-clock times are computed over 20 runs; the Total row reports statistics of the per-run totals.}}
    \label{tab:bunny_training_strategy}\medskip
    \begin{tabular}{c r r c r r r r}
        \toprule
         & & & & \multicolumn{2}{c}{Epochs} & \multicolumn{2}{c}{Time (sec)} \\
        \cmidrule(lr){5-6}\cmidrule(lr){7-8}
        Level & Grid & Problem size & Target loss & Mean & Std & Mean & Std \\
        \midrule
        1 & $40^3$  &  10\,000 & $8.0\times10^{-4}$ & $51.10$ & $3.78$  & $1.23$  & $0.26$ \\
        2 & $64^3$  &  26\,082 & $4.0\times10^{-4}$ & $53.40$ & $23.53$ & $5.61$  & $2.47$ \\
        3 & $81^3$  &  41\,974 & $2.0\times10^{-4}$ & $22.10$ & $4.86$  & $5.24$  & $1.16$ \\
        4 & $102^3$ &  66\,664 & $2.0\times10^{-4}$ & $16.45$ & $1.12$  & $9.16$  & $0.62$ \\
        5 & $128^3$ & 104\,663 & $1.5\times10^{-4}$ & $14.50$ & $0.74$  & $18.66$ & $0.95$ \\
        6 & $161^3$ & 164\,198 & $1.5\times10^{-4}$ & $10.45$ & $0.74$  & $32.78$ & $2.32$ \\
        7 & $203^3$ & 256\,242 & $1.5\times10^{-4}$ & $7.70$  & $0.46$  & $57.35$ & $3.43$ \\
        8 & $256^3$ & 394\,457 & $1.5\times10^{-4}$ & $5.70$  & $0.46$  & $99.36$ & $7.97$ \\
        \midrule
        \multicolumn{4}{l}{\textbf{Total}}
          & & & $\mathbf{229.38}$ & $\mathbf{12.06}$ \\
        \bottomrule
    \end{tabular}
\end{table}

\begin{figure}[!h]
    \centering
    \begin{subfigure}{0.22\linewidth}
        \centering
        \includegraphics[width=\linewidth]{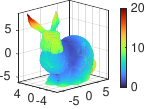}
        \caption{$\rho_\theta$ \\
        (MLSG, IBIM)}
        \label{fig:bunny_nn}
    \end{subfigure}
    \hfill
    \begin{subfigure}{0.22\linewidth}
        \centering
        \includegraphics[width=\linewidth]{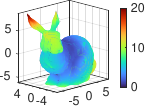}
        \caption{$\rNstar$ \\
        (GMRES, IBIM)}
        \label{fig:bunny_gmres}
    \end{subfigure}
    \hfill
    \begin{subfigure}{0.22\linewidth}
        \centering
        \includegraphics[width=\linewidth]{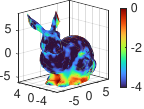}
        \caption{Error in $u$ \\
        (from NN)}
        \label{fig:bunny_nn_err}
    \end{subfigure}
    \hfill
    \begin{subfigure}{0.22\linewidth}
        \centering
        \includegraphics[width=\linewidth]{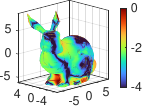}
        \caption{Error in $u$ \\(from GMRES)}
        \label{fig:bunny_gmres_err}
    \end{subfigure}
    \caption{{Stanford bunny example.  The first two panels show the surface densities obtained by MLSG and GMRES; the last two show the corresponding pointwise errors in the reconstructed field at sampled interior points, on a logarithmic scale.}}
    \label{fig:bunny_results}
\end{figure}

\subsection{Exterior Helmholtz problem in three dimensions}
\label{sec:Examples:helmholtz}

We next consider exterior Neumann Helmholtz problems over a range of
wavenumbers and geometries.  The multiple-sphere benchmark provides the
main wall-clock comparison with GMRES, while the coiled torus, linked tori,
and genus-three surface test the method on nonconvex and nontrivial
topologies.  These examples also examine progressive and cyclic training
schedules.

\subsubsection{Scattering by multiple spheres}
\label{sec:Examples:helmholtz:spheres}

The multiple-sphere geometry is a standard three-dimensional Helmholtz
benchmark.  
We analyze the runtime of the proposed approach across various resolutions and wavenumbers for a four-sphere configuration.
Here $n$ denotes the Cartesian
grid resolution in each coordinate direction and $N$ the resulting number
of boundary quadrature points, hence the dimension of the discrete BIE.

Each scatterer is a radius-$2$ ball placed inside the cubic
computational box $[-10,10]^3$; configurations are arranged so that
the shortest surface-to-surface distance between spheres is $2.57$.

The incident field is the plane wave
$u^{\mathrm{inc}}=\exp(i\kappa y)$.  The Neumann data on
$\bigcup_j\partial B_j$ are the negative normal derivative of this field,
and \eqref{eq:helm-rep}--\eqref{eq:helm-kernel} apply on each sphere;
off-diagonal blocks of $K_N$ account for inter-sphere coupling. All test wavenumbers $\kappa$ are chosen away
from the interior Dirichlet spectrum of $\bigcup_j B_j$.

The network has $5$ hidden layers of $200$ neurons with SIREN activations
($\omega_0=1$), and mini-batches contain $1024$ rows.  The grid ladder has
$\ell_{\mathrm F}=8$ levels when {\color{black}$n=256$} and
$\ell_{\mathrm F}=10$ when {\color{black}$n=512$}; the per-coordinate
resolution grows by a factor of $\sqrt2$ between successive levels, so the
number of boundary quadrature points approximately doubles.  The learning
rate is $10^{-3}$ on the first two levels and $10^{-4}$ thereafter.

\begin{figure}[t]
    \centering

    \includegraphics{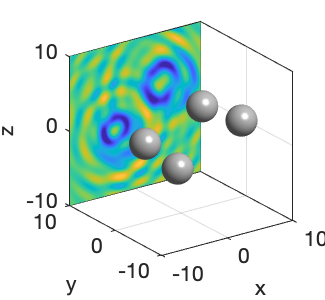}
    \qquad
    \includegraphics{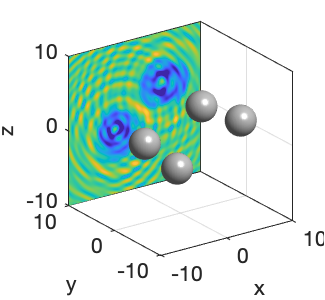}

    \caption{{Reconstructed total field $|u+u^{\mathrm{inc}}|$ for the four-sphere problem.  Left: $n=256$, $\kappa=4$.  Right: $n=512$, $\kappa=8$.}}
    \label{fig:helmholtz_grouped}
\end{figure}

\paragraph{Scaling with grid resolution and wavenumber.}

For the four-sphere configuration, we vary the per-coordinate grid
resolution $n\in\{256,512\}$ and the wavenumber $\kappa$ as listed in
Table~\ref{tab:helmholtz_result}.  The resulting BIE systems range from
$131\,030$ to $525\,182$ boundary unknowns.

\begin{table}[t]
\centering
\caption{Implementation-level wall-clock comparison for the
four-sphere Helmholtz benchmark at matched accuracy. MLSG uses one NVIDIA H200
GPU; the dense MATLAB GMRES reference uses 64 Intel Xeon Platinum 8480+ CPU
cores. The MLSG mean and standard deviation are reported over 20 trials. The
final column is the ratio of these measured wall-clock times and is not a
hardware-normalized algorithmic speedup. }
\label{tab:helmholtz_result}
\medskip
\small
\renewcommand{\arraystretch}{1.03}
\setlength{\tabcolsep}{7pt}

\begin{tabular}{r r r r r r r}
\toprule
\multicolumn{7}{c}{\textbf{(a) GMRES}} \\
\addlinespace[2pt]
$h$ & $N$ & $\kappa$
& \multicolumn{2}{r}{Time~($t_{G}$ in sec)}
& Relative error (squared)
& \\
\midrule
0.078 & 131\,030 & 1
& \multicolumn{2}{r}{3\,194.60}
& $3.33\times10^{-5}$ & \\
& & 2
& \multicolumn{2}{r}{5\,126.90}
& $3.63\times10^{-5}$ & \\
& & 3
& \multicolumn{2}{r}{9\,871.70}
& $9.83\times10^{-5}$ & \\
& & 4
& \multicolumn{2}{r}{12\,080.78}
& $6.79\times10^{-5}$ & \\
\cmidrule(lr){1-7}
0.039 & 525\,182 & 1
& \multicolumn{2}{r}{18\,949.67}
& $3.33\times10^{-5}$ & \\
& & 3
& \multicolumn{2}{r}{70\,434.94}
& $9.82\times10^{-5}$ & \\
& & 5
& \multicolumn{2}{r}{56\,394.54}
& $6.08\times10^{-5}$ & \\
& & 7
& \multicolumn{2}{r}{259\,602.63}
& $9.67\times10^{-5}$ & \\
& & 8
& \multicolumn{2}{r}{NA}
& NA & \\

\midrule

\addlinespace[2pt]
\multicolumn{7}{c}{\textbf{(b) MLSG}} \\
\addlinespace[2pt]
$h$ & $N$ & $\kappa$
& Mean ($t_{ML}$ in sec)
& Std (sec)
& Relative error (squared)
&  Ratio ($t_{G}/t_{ML}$)\\
\midrule
0.078 & 131\,030 & 1
& $\mathbf{31.76}$ & $\mathbf{3.55}$
& $\mathbf{(7.55\pm1.21)\times10^{-5}}$
& 100.58 \\
& & 2
& $\mathbf{24.69}$ & $\mathbf{2.43}$
& $\mathbf{(7.03\pm1.60)\times10^{-5}}$
& 207.67 \\
& & 3
& $\mathbf{35.63}$ & $\mathbf{3.79}$
& $\mathbf{(7.04\pm0.80)\times10^{-5}}$
& 277.07 \\
& & 4
& $\mathbf{59.19}$ & $\mathbf{3.28}$
& $\mathbf{(9.09\pm0.68)\times10^{-5}}$
& 204.10 \\
\cmidrule(lr){1-7}
0.039 & 525\,182 & 1
& $\mathbf{176.61}$ & $\mathbf{19.70}$
& $\mathbf{(2.27\pm0.54)\times10^{-5}}$
& 107.29 \\
& & 3
& $\mathbf{113.27}$ & $\mathbf{1.65}$
& $\mathbf{(4.26\pm0.40)\times10^{-5}}$
& 621.82 \\
& & 5
& $\mathbf{156.80}$ & $\mathbf{29.37}$
& $\mathbf{(5.18\pm2.21)\times10^{-5}}$
& 359.66 \\
& & 7
& $\mathbf{3\,948.78}$ & $\mathbf{151.16}$
&  $\mathbf{(8.46\pm1.91)\times10^{-5}}$
& 65.74 \\
& & 8
& $\mathbf{2\,248.81}$ & $\mathbf{236.40}$
& $\mathbf{(4.42\pm1.03)\times10^{-5}}$ 
& -- \\
\bottomrule
\end{tabular}
\end{table}
\begin{figure}[t]
    \centering
    \subfloat[Per-level wall time.\label{fig:time-vs-level}]{%
        \includegraphics[width=0.48\linewidth]{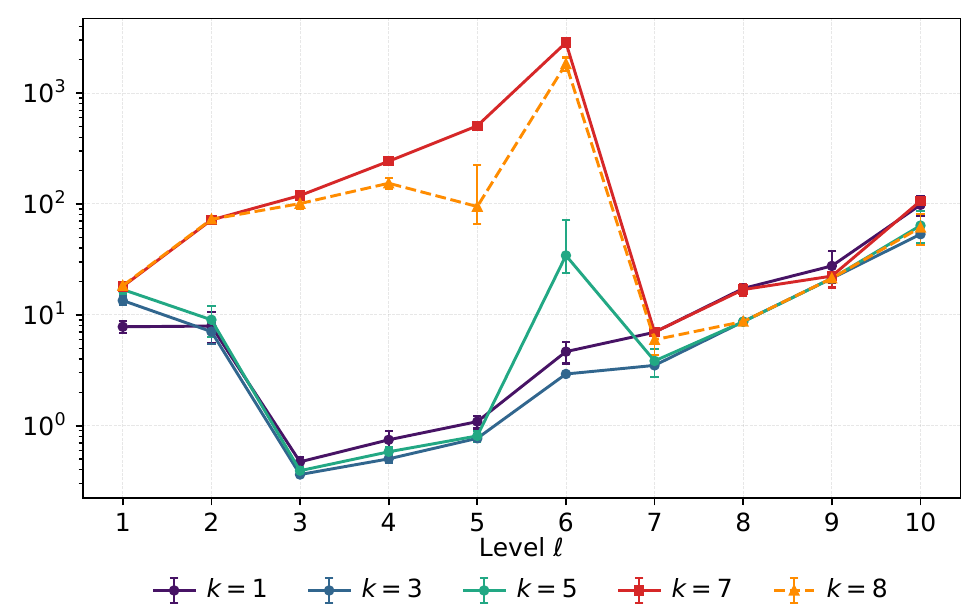}%
    }\hfill
    \subfloat[Per-level epoch count.\label{fig:epochs-vs-level}]{%
        \includegraphics[width=0.48\linewidth]{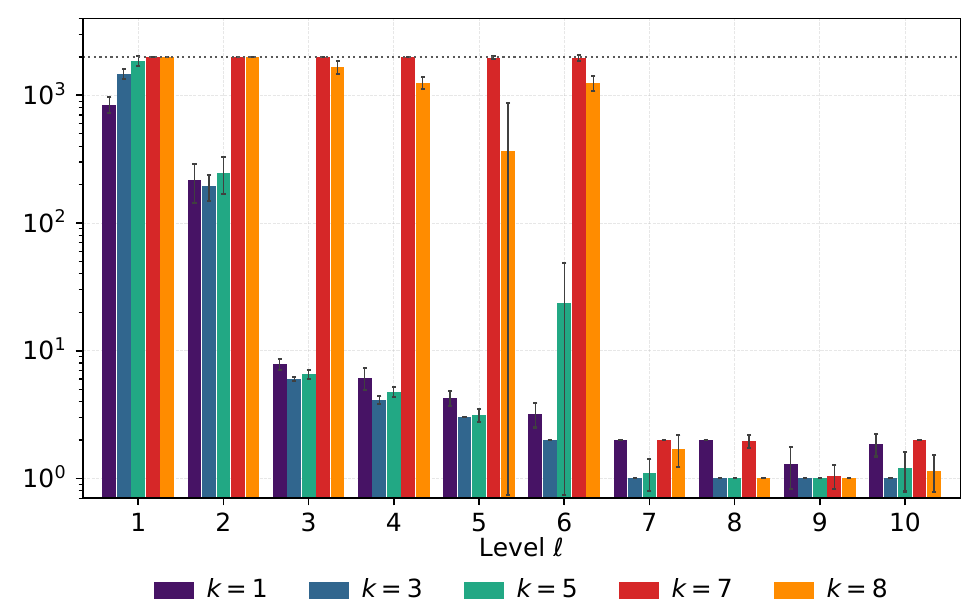}%
    }\\[4pt]
    
    \subfloat[Initial loss at each level.\label{fig:initloss-vs-level}]{%
        \includegraphics[width=0.48\linewidth]{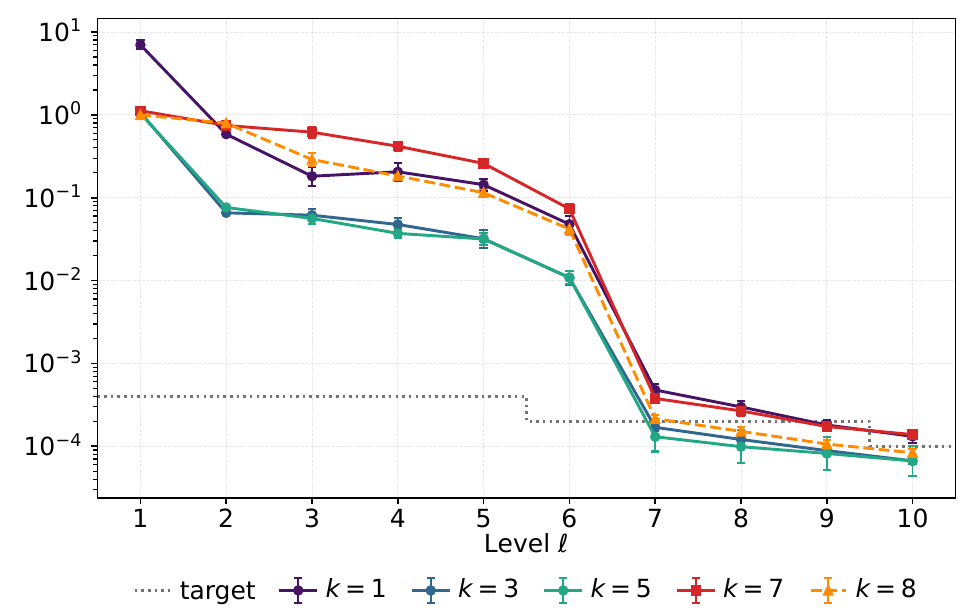}%
    }\hfill
    \subfloat[Per-level final training loss.\label{fig:loss-vs-level}]{%
        \includegraphics[width=0.48\linewidth]{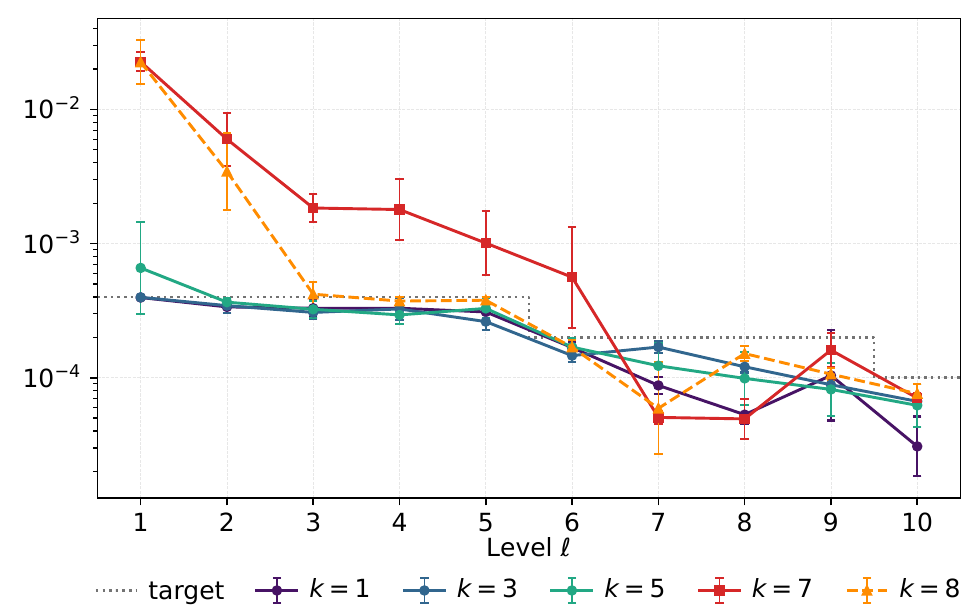}%
    }
    \caption{The number of computational epochs, the elapsed time, and the training loss for solving $\kappa = \left\{1, 3, 5, 7, 8\right\}$ with $N = 525\,182$ are shown. The $x$ axis denotes the level $\ell$. Panel (a) presents the time spent to meet the level termination condition; panel (b) shows the number of epochs used at each level. Panel (c) shows the loss at the beginning of each level, whereas panel (d) shows the loss at the end of each level, before any training is performed on that level. The dotted step curve in panels (c) and (d) denotes the prescribed level tolerance.}
    \label{fig:cost-vs-level}
\end{figure}
Table~\ref{tab:helmholtz_result} contrasts the two implementations under identical stopping criteria: the GMRES solver terminates when $\texttt{reltol}^2<10^{-4}$, while MLSG halts once the training error falls below $10^{-4}$.  Table 4 reports the measured wall-clock ratio. Under the tested hardware and software configurations, MLSG is substantially faster than the dense MATLAB GMRES reference for MLSG over the dense-GMRES
reference on the tested systems and shows that the mini-batched workload maps
efficiently to one H200 GPU. In particular, for $N = 525\,182$ and wavenumber $\kappa = 7$, the conventional GMRES solver requires over $3$ days to return the solution, even for higher wavenumber cases. This comparison indicates that our approach can be considered an alternative means of obtaining an initial understanding of a dense linear system. Subsequently, a more accurate solution can be achieved by using the neural network’s output as the initialization for the GMRES solver. 

{Note that, for $\kappa = 8$, Matlab did not finish within the four-day wall-time limit and was terminated.}
{We also noticed that there is a significant drop in the ratio for $\kappa=7$. 
The approximately fivefold increase in GMRES computation time at $\kappa=7$ suggests poorer conditioning of $A_{N}$. The increase in MLSG computation time is more likely attributable to 
the grid resolutions in coarser levels, which may limit the method’s ability to resolve higher modes of the solution.}
{To back this interpretation up, we plot the computational time and the epoch count at each level in Figure~\ref{fig:cost-vs-level}. For $\ell > 6$, both the runtime and the epoch count are largely consistent for different choices of $\kappa$, suggesting that, for $\kappa \ge 7$, a significant proportion of the solver's effort is used in coarser, under-resolving levels. In Figure~\ref{fig:initloss-vs-level}, at the coarsest level, the initial losses are comparable across all values of $\kappa$. By contrast, Figure~\ref{fig:loss-vs-level} shows that for large $\kappa$, the level loss fails to converge within the first six levels. Together, these results indicate that poor conditioning of the matrices $A_{N_\ell}$ leads to increased computational cost at high wave number $\kappa$.}

Figure~\ref{fig:helmholtz_grouped} shows representative slices of the
reconstructed total field.  When accuracy beyond the MLSG tolerance is
required, the learned density can also be used as an initial guess for a
conventional iterative solver.


\subsubsection{Scattering by nonconvex surfaces}
\label{sec:Examples:helmholtz:complex}

\paragraph{Coiled torus.}

Our first nonconvex geometry is a coiled torus: a circular tube whose
centerline winds six times around a larger circle.  The resulting smooth
surface is parametrized, for $(t,s)\in[0,2\pi]^2$, by
\begin{equation}\label{eq:coil-surface}
    \mathbf X(t,s)
    =
    \mathbf c(t)
    +
    r_{\mathrm{tube}}
    \bigl(
        \cos s\,\mathbf N(t)
        +
        \sin s\,\mathbf B(t)
    \bigr),
\end{equation}
where
\begin{equation*}
    \mathbf c(t)
    =
    \bigl(
        (5+0.9\cos 6t)\cos t,\,
        (5+0.9\cos 6t)\sin t,\,
        0.9\sin 6t
    \bigr),
\end{equation*}
$r_{\mathrm{tube}}=0.8$, and
\begin{equation*}
    \mathbf N(t)
    =
    \bigl(
        \cos 6t\cos t,\,
        \cos 6t\sin t,\,
        \sin 6t
    \bigr),
    \qquad
    \mathbf B(t)
    =
    \widehat{\mathbf T}(t)\times\mathbf N(t),
    \qquad
    \widehat{\mathbf T}(t)
    =
    \frac{\mathbf c'(t)}{\lvert\mathbf c'(t)\rvert}.
\end{equation*}

We rigidly rotate the surface so that the axis of the large circle is
aligned with $(3,4,5)/\sqrt{50}$, avoiding alignment with the coordinate
planes.  The incident field is
$u^{\mathrm{inc}}=\exp(i\kappa y)$.  A single doubly periodic chart covers
the surface, so one network module represents the density.  Dividing each
angular variable into {\color{black}$n$} intervals gives
{\color{black}$N=n^2$} quadrature nodes and a system
$A_N\in\mathbb C^{{\color{black}N\times N}}$.  The network has $10$
hidden layers of $500$ neurons and uses SIREN initialization.

Figure~\ref{fig:coil} shows the computed density and reconstructed total
field for $\kappa=6$ and {\color{black}$n=1024$}, corresponding to
{\color{black}$N=n^2=1\,048\,576$} surface quadrature nodes.  The MLSG
wall-clock time is $9979.1\,\mathrm{s}$.

\begin{figure}[h]
    \centering
    \begin{subfigure}[t]{0.4\linewidth}
        \centering
        \includegraphics[width=\linewidth]
        {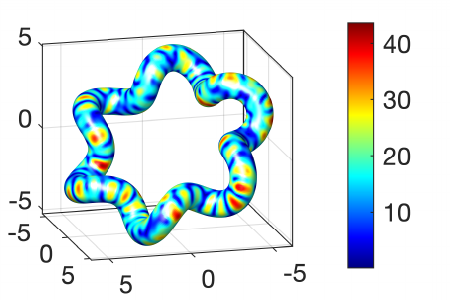}
        \caption{Magnitude of the computed density,
        $|\rho_\theta^{\mathrm{param}}|$.}
        \label{fig:coil_density}
    \end{subfigure}
    \hspace{2em}
    \begin{subfigure}[t]{0.4\linewidth}
        \centering
        \includegraphics[width=\linewidth]
        {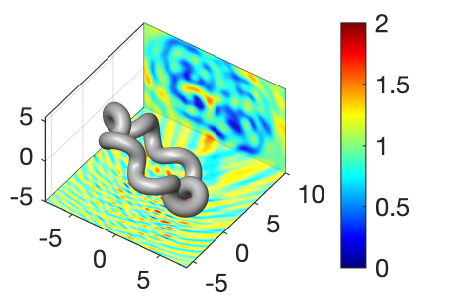}
        \caption{Magnitude of the total field,
        $|u+u^{\mathrm{inc}}|$, on a constant-$z$ plane.}
        \label{fig:coil_zslice}
    \end{subfigure}
    \caption{ The total field satisfies a homogeneous Neumann boundary
    condition, and the incident field is
    $u^{\mathrm{inc}}=\exp(i\kappa y)$.}
    \label{fig:coil}
\end{figure}


\paragraph{Linked tori.}
We next consider two congruent tori linked into a chain.  With major radius
$R=4$ and minor radius $r=1$, the two surfaces are parametrized by
\begin{eqnarray*}
    \mathbf X_1(t,s)
    &=&
    \left(
    (R+r\cos s)\cos t,\;
    r\sin s,\;
    (R+r\cos s)\sin t
    \right),\\
    \mathbf X_2(t,s)
    &=&
    \left(
    (R+r\cos s)\cos t+R,\;
    (R+r\cos s)\sin t,\;
    r\sin s
    \right),
    ~~ 0\leq t,s<2\pi ,
\end{eqnarray*}

The configuration lies in $[-20,20]^3$ and is illuminated by the plane wave
$u^{\mathrm{inc}}=\exp(i\kappa y)$.  Each torus is covered by one doubly
periodic chart, so the parametric density has two network modules,
$\rho_\theta^{\mathrm{param}}=(\rho_{\theta(1)},\rho_{\theta(2)})$.
Partitioning each angular variable into $n$ subintervals produces the
surface quadrature.

Discretizing the BIE yields a linear system in block form,
\begin{eqnarray*}
\begin{bmatrix}
A_{11} & A_{12} \\
A_{21} & A_{22}
\end{bmatrix}
\begin{bmatrix}
\rho_{1, N} \\
\rho_{2, N}
\end{bmatrix}
=
\begin{bmatrix}
f_{1} \\
f_{2}
\end{bmatrix},
\end{eqnarray*}

The diagonal blocks $A_{ii}$ describe self-interaction on each torus, while
the off-diagonal blocks $A_{ij}$, $i\ne j$, couple the two components; $f_i$
is the boundary data on the $i$th torus.  The module $\rho_{\theta(i)}$
represents the corresponding discrete density.  To provide sufficient
capacity for this coupled geometry, each module uses $10$ hidden layers of
$500$ neurons.

Table~\ref{tab:torii_training} summarizes the multilevel training, and
Figure~\ref{fig:tori} shows the learned surface density and reconstructed
total field on representative cross-sections.  The plotted quantities are
shown in absolute value.

{{Furthermore, level 9 demands more computation time and additional epochs to meet the level loss, as shown in Table~\ref{tab:torii_training}. We adjust the level loss at level $9$ to $2$, $1.5$, and $1\times 10^{-4}$, and summarize the computational outcomes in Table~\ref{tab:level9-target}. When the target loss matches the one used for level 8, MLSG shows stable runtimes and epoch counts; however, when the level loss is tightened to $1\times 10^{-4}$, it takes longer to satisfy the stopping criterion. Hence, Table~\ref{tab:level9-target} supports the explanation for the long waiting time needed for convergence at level $9$.}}

\begin{table}[htbp]
  \centering
  \caption{MLSG training statistics for the linked-tori example.}
  \label{tab:torii}
  \medskip

  \begin{subtable}{\textwidth}
    \centering
    \caption{Per-level statistics under the progressive schedule. Means and standard
      deviations of the epoch counts and wall-clock times are computed over 20 trials.}
    \label{tab:torii_training}
    \medskip
    \begin{tabular}{c r r r r r r}
      \toprule
      & & & \multicolumn{2}{c}{\# epochs} & \multicolumn{2}{c}{Time (sec)} \\
      \cmidrule(lr){4-5}\cmidrule(lr){6-7}
      Level & Problem size & Target loss & Mean & Std & Mean & Std \\
      \midrule
      1  & $1\,024 \times 2$  & $4\times10^{-4}$ & $570.30$    & $25.97$ & $3.59$      & $0.44$  \\
      2  & $1\,600 \times 2$  & $4\times10^{-4}$ & $687.40$    & $34.09$ & $4.56$      & $0.26$  \\
      3  & $2\,500 \times 2$  & $4\times10^{-4}$ & $523.55$    & $32.46$ & $4.54$      & $0.29$  \\
      4  & $4\,096 \times 2$  & $2\times10^{-4}$ & $878.00$    & $45.10$ & $12.98$     & $0.68$  \\
      5  & $6\,561 \times 2$  & $2\times10^{-4}$ & $889.60$    & $54.76$ & $26.92$     & $1.66$  \\
      6  & $10\,404 \times 2$ & $2\times10^{-4}$ & $552.45$    & $45.85$ & $44.16$     & $3.69$  \\
      7  & $16\,384 \times 2$ & $2\times10^{-4}$ & $414.90$    & $23.41$ & $70.00$     & $3.97$  \\
      8  & $25\,921 \times 2$ & $2\times10^{-4}$ & $891.05$    & $60.00$ & $399.13$    & $26.97$ \\
      9  & $41\,209 \times 2$ & $1\times10^{-4}$ & $1\,204.00$ & $76.58$ & $1\,310.11$ & $83.38$ \\
      10 & $65\,536 \times 2$ & $1\times10^{-4}$ & $297.25$    & $28.50$ & $785.59$    & $75.34$ \\
      \midrule
      \multicolumn{3}{l}{\textbf{Total}}
      & $\mathbf{6\,908.50}$ & $\mathbf{276.39}$
      & $\mathbf{2\,661.57}$ & $\mathbf{144.71}$ \\
      \bottomrule
    \end{tabular}
  \end{subtable}

  \bigskip

  \begin{subtable}{\textwidth}
    \centering
    \small
    \setlength{\tabcolsep}{3.5pt}
    \caption{Statistics at levels 9 and 10 under different level-9 target losses. Means and
      standard deviations are computed over 5 trials for the targets $2\times10^{-4}$ and
      $1.5\times10^{-4}$, and over 20 trials for $1\times10^{-4}$.}
    \label{tab:level9-target}
    \medskip
    \begin{tabular}{c rrrr rrrr rrrr}
      \toprule
      & \multicolumn{4}{c}{Level-9 target $2\times 10^{-4}$}
      & \multicolumn{4}{c}{Level-9 target $1.5\times 10^{-4}$}
      & \multicolumn{4}{c}{Level-9 target $1\times 10^{-4}$} \\
      \cmidrule(lr){2-5}\cmidrule(lr){6-9}\cmidrule(lr){10-13}
      & \multicolumn{2}{c}{\# epochs} & \multicolumn{2}{c}{Time (sec)}
      & \multicolumn{2}{c}{\# epochs} & \multicolumn{2}{c}{Time (sec)}
      & \multicolumn{2}{c}{\# epochs} & \multicolumn{2}{c}{Time (sec)} \\
      \cmidrule(lr){2-3}\cmidrule(lr){4-5}
      \cmidrule(lr){6-7}\cmidrule(lr){8-9}
      \cmidrule(lr){10-11}\cmidrule(lr){12-13}
      Level & Mean & Std & Mean & Std
            & Mean & Std & Mean & Std
            & Mean & Std & Mean & Std \\
      \midrule
      9  & 832.80 & 73.55 & 901.40 & 79.67
         & 973.20 & 46.71 & 1\,054.40 & 50.48
         & 1\,204.00 & 76.58 & 1\,310.11 & 83.38 \\
      10 & 307.20 & 1.94 & 807.73 & 5.01
         & 308.20 & 0.98 & 810.45 & 2.46
         & 297.25 & 28.50 & 785.59 & 75.34 \\
      \bottomrule
    \end{tabular}
  \end{subtable}
\end{table}
\begin{figure}[h]
    \centering
    \begin{subfigure}[t]{0.27\linewidth}
        \centering        \includegraphics[width = \textwidth]{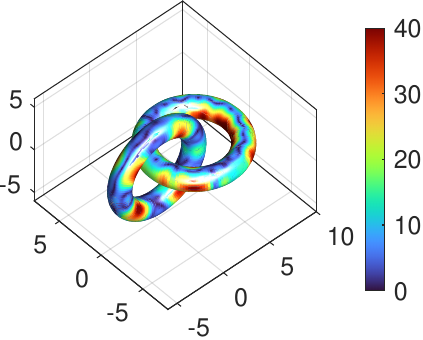}
        \caption{Computed density $|\rho_\theta^{\mathrm{param}}|$.}
        \label{fig:tori_density}
    \end{subfigure}
    \hfill
    \begin{subfigure}[t]{0.27\linewidth}
        \centering
        \includegraphics[width = \textwidth]{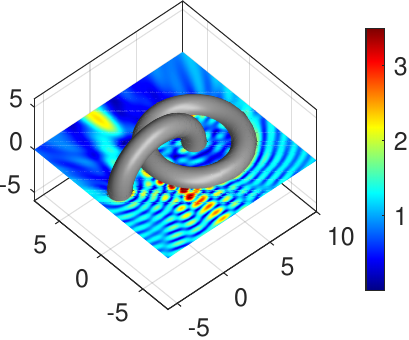}
        \caption{Total field $|u+u^{\mathrm{inc}}|$ on a constant-$z$ plane.}
        \label{fig:tori_zslice}
    \end{subfigure}
    \hfill
    \begin{subfigure}[t]{0.27\linewidth}
        \centering
        \includegraphics[width = \textwidth]{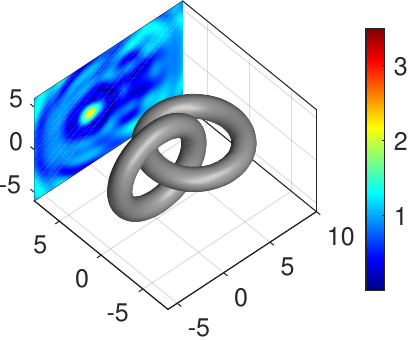}
        \caption{Total field $|u+u^{\mathrm{inc}}|$ on a constant-$y$ plane.}
        \label{fig:tori_yslice}
    \end{subfigure}
    \caption{Linked-tori Helmholtz problem ($\kappa = 4$) with a zero Neumann boundary condition and an incident plane wave in the $y$-direction.}
    \label{fig:tori}
\end{figure}
\paragraph{Epoch count versus prescribed tolerance.}

To examine how the optimization effort changes with tighter stopping
criteria, we fix the training strategy and vary only the finest-level target
$\varepsilon_{\ell_{\mathrm F}}$.  Across levels the targets satisfy
$\varepsilon_\ell=1.5\,\varepsilon_{\ell+1}$, while the quadrature
resolution increases by a factor of $2^{1/3}$ per level.  Using the
Lipschitz sampling scale $m=1$ gives
$(h_\ell/h_{\ell+1})^{2m}=2^{2/3}$, close to the imposed tolerance ratio.

We run this study on the linked tori and coiled torus.  The linked-tori
hierarchy has ten levels and reaches $n=256$ per angular variable
($N=2n^2=131\,072$) at the finest level, with
$\kappa\in\{1,2,3,4\}$.  The coiled-torus hierarchy has thirteen levels
and reaches $n=512$ ($N=512^2=262\,144$), with
$\kappa\in\{1,2,3,4,5\}$.

Figure~\ref{fig:tol_workload}, with panels
\ref{fig:torii_tol} and~\ref{fig:coil_tol}, plots the total epoch count,
summed over all levels, against the achieved finest-level loss.  Over the
tested range, the dependence on inverse loss is clearly sublinear.  The
dashed guides show the reference laws
$\varepsilon_{\ell_{\mathrm F}}^{-0.2}$ and
$\varepsilon_{\ell_{\mathrm F}}^{-0.5}$; they are included to indicate
the scale of the observed growth rather than to assert an asymptotic
complexity exponent.  In particular, substantially lower achieved losses
are obtained without a proportional increase in total epochs.
\begin{figure}[h]
    \centering
    \begin{subfigure}[t]{0.4\linewidth}
        \centering
        \includegraphics[width=\linewidth]{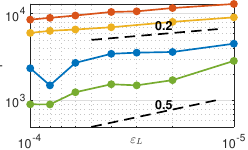}
        \caption{Two linked tori, $10$ levels, $n=256$ at the finest
        level; $\kappa=1,2,3,4$ (blue, green, yellow, red).}
        \label{fig:torii_tol}
    \end{subfigure}
    \hspace{2em}
    \begin{subfigure}[t]{0.4\linewidth}
        \centering
        \includegraphics[width=\linewidth]{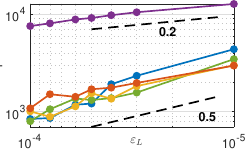}
        \caption{Coiled torus, $13$ levels, $n=512$ at the finest level;
        $\kappa=1,2,3,4,5$ (blue, green, yellow, red, purple).}
        \label{fig:coil_tol}
    \end{subfigure}
    \caption{Total epoch count summed over all refinement levels versus
    the achieved finest-level loss $\varepsilon_{\ell_{\text{F}}}$ (log--log;
    $\varepsilon_{\ell_{\text{F}}}$ decreases to the right as the tolerance tightens).
    The dashed lines indicate the reference slopes $\varepsilon_{\ell_{\text{F}}}^{-0.2}$
    and $\varepsilon_{\ell_{\text{F}}}^{-0.5}$.}
    \label{fig:tol_workload}
\end{figure}
\paragraph{Scattering by a genus-three surface.}

Finally, we solve the exterior Helmholtz problem on a genus-three surface
represented by a dense point cloud and discretized with IBIM, and compare the performance between using the cyclic schedule and the progressive schedule.  The incident
field is the plane wave $u^{\mathrm{inc}}=\exp(i\kappa x)$.

\begin{figure}[h]
    \centering
    \begin{subfigure}[t]{0.24\linewidth}
        \centering
        \includegraphics[width=\linewidth]{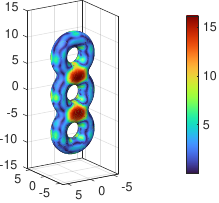}
        \caption{Computed density $|\rho_\theta^{\mathrm{IBIM}}|$.}
        \label{fig:g3_obj}
    \end{subfigure}
    \hfill
    \begin{subfigure}[t]{0.30\linewidth}
        \centering
        \includegraphics[width=\linewidth]{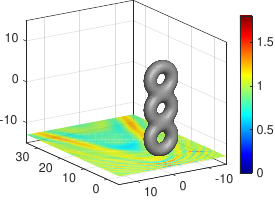}
        \caption{Total field $|u+u^{\mathrm{inc}}|$ on a constant-$z$ plane.}
        \label{fig:g3_zslice}
    \end{subfigure}
    \hfill
    \begin{subfigure}[t]{0.30\linewidth}
        \centering
        \includegraphics[width=\linewidth]{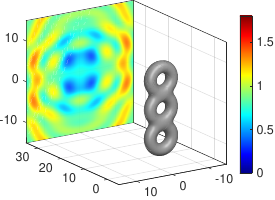}
        \caption{Total field $|u+u^{\mathrm{inc}}|$ on a constant-$x$ plane.}
        \label{fig:g3_xslice}
    \end{subfigure}
    \caption{{Helmholtz problem ($\kappa = 3$) with a zero Neumann boundary condition and an incident plane wave in the $x$-direction.}}
    \label{fig:g3}
\end{figure}

Figure~\ref{fig:g3_obj} shows the learned density for $\kappa=3$ with the
incident wave propagating in the $x$-direction.  The Cartesian grid has
$n=128$ points per coordinate direction ($h\approx0.2197$), and the
network has $10$ hidden layers of $500$ neurons with SIREN activation and
$\omega_0=10$.

Table~\ref{tab:genus3_training_combined} compares the cyclic and
progressive schedules.  The cyclic schedule has a smaller spread in both
the epoch count and the final-level wall-clock time.

\begin{table}[!h]
\centering
\caption{{MLSG training on the genus-three surface using the cyclic and progressive schedules (20 trials). Boxes span the interquartile range with the median marked; whiskers extend to the most extreme values within $1.5\times\mathrm{IQR}$, and circles mark outliers. Boldface levels in the cyclic schedule indicate revisited levels trained at the tightened target.}}
\label{tab:genus3_training_combined}
\medskip
\begin{tabular}{c r r c c c}
\toprule
Level & $n$ & $N$& $\varepsilon$ & \#\,epochs & Time (sec) \\
\cmidrule(lr){5-5}\cmidrule(lr){6-6}
 & & & & \begin{tikzpicture}[x=4.5cm, y=1.6ex, baseline=0.5ex]\draw[axisline] (0,0) -- (1,0);\draw[axisline] (0,0) -- (0,0.8);\node[tick] at (0,2.6) {0};\draw[axisline] (0.2,0) -- (0.2,0.8);\node[tick] at (0.2,2.6) {50};\draw[axisline] (0.4,0) -- (0.4,0.8);\node[tick] at (0.4,2.6) {100};\draw[axisline] (0.6,0) -- (0.6,0.8);\node[tick] at (0.6,2.6) {150};\draw[axisline] (0.8,0) -- (0.8,0.8);\node[tick] at (0.8,2.6) {200};\draw[axisline] (1,0) -- (1,0.8);\node[tick] at (1,2.6) {250};\end{tikzpicture} & \begin{tikzpicture}[x=4.5cm, y=1.6ex, baseline=0.5ex]\draw[axisline] (0,0) -- (1,0);\draw[axisline] (0,0) -- (0,0.8);\node[tick] at (0,2.6) {0};\draw[axisline] (0.2,0) -- (0.2,0.8);\node[tick] at (0.2,2.6) {15};\draw[axisline] (0.4,0) -- (0.4,0.8);\node[tick] at (0.4,2.6) {30};\draw[axisline] (0.6,0) -- (0.6,0.8);\node[tick] at (0.6,2.6) {45};\draw[axisline] (0.8,0) -- (0.8,0.8);\node[tick] at (0.8,2.6) {60};\draw[axisline] (1,0) -- (1,0.8);\node[tick] at (1,2.6) {75};\end{tikzpicture} \\[2pt]
\midrule
\multicolumn{6}{c}{{\textbf{(a) Cyclic schedule}}} \\
\midrule
1 & $32^3$ & 2\,232 & $5\times10^{-3}$ & \begin{tikzpicture}[x=4.5cm, y=1.6ex, baseline=-0.5ex]\draw[axisline] (0,0) -- (1,0);\draw[bxp] (0.2480,0) -- (0.2600,0);\draw[bxp] (0.3110,0) -- (0.3840,0);\draw[bxp] (0.2480,-0.7) -- (0.2480,0.7);\draw[bxp] (0.3840,-0.7) -- (0.3840,0.7);\filldraw[boxfill,draw=black,line width=0.5pt] (0.2600,-1.1) rectangle (0.3110,1.1);\draw[medline] (0.2720,-1.1) -- (0.2720,1.1);\draw[outl] (0.4,0) circle (1.6pt);\end{tikzpicture} & \begin{tikzpicture}[x=4.5cm, y=1.6ex, baseline=-0.5ex]\draw[axisline] (0,0) -- (1,0);\draw[bxp] (0.0053,0) -- (0.0053,0);\draw[bxp] (0.0067,0) -- (0.0080,0);\draw[bxp] (0.0053,-0.7) -- (0.0053,0.7);\draw[bxp] (0.0080,-0.7) -- (0.0080,0.7);\filldraw[boxfill,draw=black,line width=0.5pt] (0.0053,-1.1) rectangle (0.0067,1.1);\draw[medline] (0.0053,-1.1) -- (0.0053,1.1);\draw[outl] (0.0200,0) circle (1.6pt);\end{tikzpicture} \\[3pt]
2 & $40^3$ & 3\,512 & $5\times10^{-3}$ & \begin{tikzpicture}[x=4.5cm, y=1.6ex, baseline=-0.5ex]\draw[axisline] (0,0) -- (1,0);\draw[bxp] (0.1520,0) -- (0.1560,0);\draw[bxp] (0.1640,0) -- (0.1720,0);\draw[bxp] (0.1520,-0.7) -- (0.1520,0.7);\draw[bxp] (0.1720,-0.7) -- (0.1720,0.7);\filldraw[boxfill,draw=black,line width=0.5pt] (0.1560,-1.1) rectangle (0.1640,1.1);\draw[medline] (0.1600,-1.1) -- (0.1600,1.1);\end{tikzpicture} & \begin{tikzpicture}[x=4.5cm, y=1.6ex, baseline=-0.5ex]\draw[axisline] (0,0) -- (1,0);\draw[bxp] (0.0040,0) -- (0.0040,0);\draw[bxp] (0.0040,0) -- (0.0040,0);\draw[bxp] (0.0040,-0.7) -- (0.0040,0.7);\draw[bxp] (0.0040,-0.7) -- (0.0040,0.7);\filldraw[boxfill,draw=black,line width=0.5pt] (0.0040,-1.1) rectangle (0.0040,1.1);\draw[medline] (0.0040,-1.1) -- (0.0040,1.1);\end{tikzpicture} \\[3pt]
3 & $50^3$ & 5\,628 & $5\times10^{-3}$ & \begin{tikzpicture}[x=4.5cm, y=1.6ex, baseline=-0.5ex]\draw[axisline] (0,0) -- (1,0);\draw[bxp] (0.1400,0) -- (0.1400,0);\draw[bxp] (0.1480,0) -- (0.1600,0);\draw[bxp] (0.1400,-0.7) -- (0.1400,0.7);\draw[bxp] (0.1600,-0.7) -- (0.1600,0.7);\filldraw[boxfill,draw=black,line width=0.5pt] (0.1400,-1.1) rectangle (0.1480,1.1);\draw[medline] (0.1440,-1.1) -- (0.1440,1.1);\end{tikzpicture} & \begin{tikzpicture}[x=4.5cm, y=1.6ex, baseline=-0.5ex]\draw[axisline] (0,0) -- (1,0);\draw[bxp] (0.0053,0) -- (0.0053,0);\draw[bxp] (0.0053,0) -- (0.0053,0);\draw[bxp] (0.0053,-0.7) -- (0.0053,0.7);\draw[bxp] (0.0053,-0.7) -- (0.0053,0.7);\filldraw[boxfill,draw=black,line width=0.5pt] (0.0053,-1.1) rectangle (0.0053,1.1);\draw[medline] (0.0053,-1.1) -- (0.0053,1.1);\draw[outl] (0.0067,0) circle (1.6pt);\end{tikzpicture} \\[3pt]
\midrule
4 & $64^3$ & 9\,264 & $3\times10^{-3}$ & \begin{tikzpicture}[x=4.5cm, y=1.6ex, baseline=-0.5ex]\draw[axisline] (0,0) -- (1,0);\draw[bxp] (0.1080,0) -- (0.1120,0);\draw[bxp] (0.1160,0) -- (0.1200,0);\draw[bxp] (0.1080,-0.7) -- (0.1080,0.7);\draw[bxp] (0.1200,-0.7) -- (0.1200,0.7);\filldraw[boxfill,draw=black,line width=0.5pt] (0.1120,-1.1) rectangle (0.1160,1.1);\draw[medline] (0.1160,-1.1) -- (0.1160,1.1);\draw[outl] (0.1440,0) circle (1.6pt);\draw[outl] (0.1480,0) circle (1.6pt);\end{tikzpicture} & \begin{tikzpicture}[x=4.5cm, y=1.6ex, baseline=-0.5ex]\draw[axisline] (0,0) -- (1,0);\draw[bxp] (0.0107,0) -- (0.0107,0);\draw[bxp] (0.0107,0) -- (0.0107,0);\draw[bxp] (0.0107,-0.7) -- (0.0107,0.7);\draw[bxp] (0.0107,-0.7) -- (0.0107,0.7);\filldraw[boxfill,draw=black,line width=0.5pt] (0.0107,-1.1) rectangle (0.0107,1.1);\draw[medline] (0.0107,-1.1) -- (0.0107,1.1);\draw[outl] (0.0120,0) circle (1.6pt);\draw[outl] (0.0133,0) circle (1.6pt);\draw[outl] (0.0147,0) circle (1.6pt);\end{tikzpicture} \\[3pt]
5 & $81^3$ & 15\,198 & $3\times10^{-3}$ & \begin{tikzpicture}[x=4.5cm, y=1.6ex, baseline=-0.5ex]\draw[axisline] (0,0) -- (1,0);\draw[bxp] (0.1040,0) -- (0.1080,0);\draw[bxp] (0.1120,0) -- (0.1120,0);\draw[bxp] (0.1040,-0.7) -- (0.1040,0.7);\draw[bxp] (0.1120,-0.7) -- (0.1120,0.7);\filldraw[boxfill,draw=black,line width=0.5pt] (0.1080,-1.1) rectangle (0.1120,1.1);\draw[medline] (0.1120,-1.1) -- (0.1120,1.1);\end{tikzpicture} & \begin{tikzpicture}[x=4.5cm, y=1.6ex, baseline=-0.5ex]\draw[axisline] (0,0) -- (1,0);\draw[bxp] (0.0187,0) -- (0.0187,0);\draw[bxp] (0.0200,0) -- (0.0200,0);\draw[bxp] (0.0187,-0.7) -- (0.0187,0.7);\draw[bxp] (0.0200,-0.7) -- (0.0200,0.7);\filldraw[boxfill,draw=black,line width=0.5pt] (0.0187,-1.1) rectangle (0.0200,1.1);\draw[medline] (0.0200,-1.1) -- (0.0200,1.1);\end{tikzpicture} \\[3pt]
6 & $102^3$ & 24\,034 & $3\times10^{-3}$ & \begin{tikzpicture}[x=4.5cm, y=1.6ex, baseline=-0.5ex]\draw[axisline] (0,0) -- (1,0);\draw[bxp] (0.0920,0) -- (0.0920,0);\draw[bxp] (0.0920,0) -- (0.0920,0);\draw[bxp] (0.0920,-0.7) -- (0.0920,0.7);\draw[bxp] (0.0920,-0.7) -- (0.0920,0.7);\filldraw[boxfill,draw=black,line width=0.5pt] (0.0920,-1.1) rectangle (0.0920,1.1);\draw[medline] (0.0920,-1.1) -- (0.0920,1.1);\draw[outl] (0.0880,0) circle (1.6pt);\draw[outl] (0.0960,0) circle (1.6pt);\draw[outl] (0.0960,0) circle (1.6pt);\draw[outl] (0.0960,0) circle (1.6pt);\draw[outl] (0.1560,0) circle (1.6pt);\end{tikzpicture} & \begin{tikzpicture}[x=4.5cm, y=1.6ex, baseline=-0.5ex]\draw[axisline] (0,0) -- (1,0);\draw[bxp] (0.0360,0) -- (0.0360,0);\draw[bxp] (0.0360,0) -- (0.0360,0);\draw[bxp] (0.0360,-0.7) -- (0.0360,0.7);\draw[bxp] (0.0360,-0.7) -- (0.0360,0.7);\filldraw[boxfill,draw=black,line width=0.5pt] (0.0360,-1.1) rectangle (0.0360,1.1);\draw[medline] (0.0360,-1.1) -- (0.0360,1.1);\draw[outl] (0.0347,0) circle (1.6pt);\draw[outl] (0.0373,0) circle (1.6pt);\draw[outl] (0.0373,0) circle (1.6pt);\draw[outl] (0.0373,0) circle (1.6pt);\draw[outl] (0.0613,0) circle (1.6pt);\end{tikzpicture} \\[3pt]
\midrule
\textbf{5} & $81^3$ & 15\,198 & $1\times10^{-3}$ & \begin{tikzpicture}[x=4.5cm, y=1.6ex, baseline=-0.5ex]\draw[axisline] (0,0) -- (1,0);\draw[bxp] (0.1720,0) -- (0.1920,0);\draw[bxp] (0.2080,0) -- (0.2160,0);\draw[bxp] (0.1720,-0.7) -- (0.1720,0.7);\draw[bxp] (0.2160,-0.7) -- (0.2160,0.7);\filldraw[boxfill,draw=black,line width=0.5pt] (0.1920,-1.1) rectangle (0.2080,1.1);\draw[medline] (0.2040,-1.1) -- (0.2040,1.1);\end{tikzpicture} & \begin{tikzpicture}[x=4.5cm, y=1.6ex, baseline=-0.5ex]\draw[axisline] (0,0) -- (1,0);\draw[bxp] (0.0307,0) -- (0.0333,0);\draw[bxp] (0.0360,0) -- (0.0373,0);\draw[bxp] (0.0307,-0.7) -- (0.0307,0.7);\draw[bxp] (0.0373,-0.7) -- (0.0373,0.7);\filldraw[boxfill,draw=black,line width=0.5pt] (0.0333,-1.1) rectangle (0.0360,1.1);\draw[medline] (0.0353,-1.1) -- (0.0353,1.1);\end{tikzpicture} \\[3pt]
\textbf{6} & $102^3$ & 24\,034 & $1\times10^{-3}$ & \begin{tikzpicture}[x=4.5cm, y=1.6ex, baseline=-0.5ex]\draw[axisline] (0,0) -- (1,0);\draw[bxp] (0.2400,0) -- (0.2480,0);\draw[bxp] (0.2680,0) -- (0.2760,0);\draw[bxp] (0.2400,-0.7) -- (0.2400,0.7);\draw[bxp] (0.2760,-0.7) -- (0.2760,0.7);\filldraw[boxfill,draw=black,line width=0.5pt] (0.2480,-1.1) rectangle (0.2680,1.1);\draw[medline] (0.2580,-1.1) -- (0.2580,1.1);\draw[outl] (0.3120,0) circle (1.6pt);\draw[outl] (0.3320,0) circle (1.6pt);\draw[outl] (0.3360,0) circle (1.6pt);\end{tikzpicture} & \begin{tikzpicture}[x=4.5cm, y=1.6ex, baseline=-0.5ex]\draw[axisline] (0,0) -- (1,0);\draw[bxp] (0.0947,0) -- (0.0973,0);\draw[bxp] (0.1053,0) -- (0.1093,0);\draw[bxp] (0.0947,-0.7) -- (0.0947,0.7);\draw[bxp] (0.1093,-0.7) -- (0.1093,0.7);\filldraw[boxfill,draw=black,line width=0.5pt] (0.0973,-1.1) rectangle (0.1053,1.1);\draw[medline] (0.1013,-1.1) -- (0.1013,1.1);\draw[outl] (0.1227,0) circle (1.6pt);\draw[outl] (0.1307,0) circle (1.6pt);\draw[outl] (0.1320,0) circle (1.6pt);\end{tikzpicture} \\[3pt]
\midrule
7 & $128^3$ & 38\,265 & $1\times10^{-3}$ & \begin{tikzpicture}[x=4.5cm, y=1.6ex, baseline=-0.5ex]\draw[axisline] (0,0) -- (1,0);\draw[bxp] (0.4120,0) -- (0.4230,0);\draw[bxp] (0.4320,0) -- (0.4400,0);\draw[bxp] (0.4120,-0.7) -- (0.4120,0.7);\draw[bxp] (0.4400,-0.7) -- (0.4400,0.7);\filldraw[boxfill,draw=black,line width=0.5pt] (0.4230,-1.1) rectangle (0.4320,1.1);\draw[medline] (0.4280,-1.1) -- (0.4280,1.1);\draw[outl] (0.4080,0) circle (1.6pt);\draw[outl] (0.7280,0) circle (1.6pt);\draw[outl] (0.8080,0) circle (1.6pt);\end{tikzpicture} & \begin{tikzpicture}[x=4.5cm, y=1.6ex, baseline=-0.5ex]\draw[axisline] (0,0) -- (1,0);\draw[bxp] (0.3947,0) -- (0.4040,0);\draw[bxp] (0.4163,0) -- (0.4240,0);\draw[bxp] (0.3947,-0.7) -- (0.3947,0.7);\draw[bxp] (0.4240,-0.7) -- (0.4240,0.7);\filldraw[boxfill,draw=black,line width=0.5pt] (0.4040,-1.1) rectangle (0.4163,1.1);\draw[medline] (0.4107,-1.1) -- (0.4107,1.1);\draw[outl] (0.6947,0) circle (1.6pt);\draw[outl] (0.7787,0) circle (1.6pt);\end{tikzpicture} \\[3pt]
\addlinespace[3pt]
\midrule
\multicolumn{6}{c}{{\textbf{(b) Progressive schedule}}} \\
\midrule
1 & $32^3$ & 2\,232 & $5\times10^{-3}$ & \begin{tikzpicture}[x=4.5cm, y=1.6ex, baseline=-0.5ex]\draw[axisline] (0,0) -- (1,0);\draw[bxp] (0.2440,0) -- (0.2640,0);\draw[bxp] (0.3240,0) -- (0.3800,0);\draw[bxp] (0.2440,-0.7) -- (0.2440,0.7);\draw[bxp] (0.3800,-0.7) -- (0.3800,0.7);\filldraw[boxfill,draw=black,line width=0.5pt] (0.2640,-1.1) rectangle (0.3240,1.1);\draw[medline] (0.2780,-1.1) -- (0.2780,1.1);\end{tikzpicture} & \begin{tikzpicture}[x=4.5cm, y=1.6ex, baseline=-0.5ex]\draw[axisline] (0,0) -- (1,0);\draw[bxp] (0.0053,0) -- (0.0053,0);\draw[bxp] (0.0067,0) -- (0.0080,0);\draw[bxp] (0.0053,-0.7) -- (0.0053,0.7);\draw[bxp] (0.0080,-0.7) -- (0.0080,0.7);\filldraw[boxfill,draw=black,line width=0.5pt] (0.0053,-1.1) rectangle (0.0067,1.1);\draw[medline] (0.0053,-1.1) -- (0.0053,1.1);\draw[outl] (0.0213,0) circle (1.6pt);\end{tikzpicture} \\[3pt]
2 & $40^3$ & 3\,512 & $5\times10^{-3}$ & \begin{tikzpicture}[x=4.5cm, y=1.6ex, baseline=-0.5ex]\draw[axisline] (0,0) -- (1,0);\draw[bxp] (0.1480,0) -- (0.1560,0);\draw[bxp] (0.1640,0) -- (0.1680,0);\draw[bxp] (0.1480,-0.7) -- (0.1480,0.7);\draw[bxp] (0.1680,-0.7) -- (0.1680,0.7);\filldraw[boxfill,draw=black,line width=0.5pt] (0.1560,-1.1) rectangle (0.1640,1.1);\draw[medline] (0.1580,-1.1) -- (0.1580,1.1);\end{tikzpicture} & \begin{tikzpicture}[x=4.5cm, y=1.6ex, baseline=-0.5ex]\draw[axisline] (0,0) -- (1,0);\draw[bxp] (0.0040,0) -- (0.0040,0);\draw[bxp] (0.0040,0) -- (0.0040,0);\draw[bxp] (0.0040,-0.7) -- (0.0040,0.7);\draw[bxp] (0.0040,-0.7) -- (0.0040,0.7);\filldraw[boxfill,draw=black,line width=0.5pt] (0.0040,-1.1) rectangle (0.0040,1.1);\draw[medline] (0.0040,-1.1) -- (0.0040,1.1);\end{tikzpicture} \\[3pt]
3 & $50^3$ & 5\,628 & $5\times10^{-3}$ & \begin{tikzpicture}[x=4.5cm, y=1.6ex, baseline=-0.5ex]\draw[axisline] (0,0) -- (1,0);\draw[bxp] (0.1400,0) -- (0.1440,0);\draw[bxp] (0.1480,0) -- (0.1520,0);\draw[bxp] (0.1400,-0.7) -- (0.1400,0.7);\draw[bxp] (0.1520,-0.7) -- (0.1520,0.7);\filldraw[boxfill,draw=black,line width=0.5pt] (0.1440,-1.1) rectangle (0.1480,1.1);\draw[medline] (0.1460,-1.1) -- (0.1460,1.1);\end{tikzpicture} & \begin{tikzpicture}[x=4.5cm, y=1.6ex, baseline=-0.5ex]\draw[axisline] (0,0) -- (1,0);\draw[bxp] (0.0053,0) -- (0.0053,0);\draw[bxp] (0.0053,0) -- (0.0053,0);\draw[bxp] (0.0053,-0.7) -- (0.0053,0.7);\draw[bxp] (0.0053,-0.7) -- (0.0053,0.7);\filldraw[boxfill,draw=black,line width=0.5pt] (0.0053,-1.1) rectangle (0.0053,1.1);\draw[medline] (0.0053,-1.1) -- (0.0053,1.1);\end{tikzpicture} \\[3pt]
\midrule
4 & $64^3$ & 9\,264 & $3\times10^{-3}$ & \begin{tikzpicture}[x=4.5cm, y=1.6ex, baseline=-0.5ex]\draw[axisline] (0,0) -- (1,0);\draw[bxp] (0.1080,0) -- (0.1120,0);\draw[bxp] (0.1160,0) -- (0.1200,0);\draw[bxp] (0.1080,-0.7) -- (0.1080,0.7);\draw[bxp] (0.1200,-0.7) -- (0.1200,0.7);\filldraw[boxfill,draw=black,line width=0.5pt] (0.1120,-1.1) rectangle (0.1160,1.1);\draw[medline] (0.1160,-1.1) -- (0.1160,1.1);\draw[outl] (0.1680,0) circle (1.6pt);\end{tikzpicture} & \begin{tikzpicture}[x=4.5cm, y=1.6ex, baseline=-0.5ex]\draw[axisline] (0,0) -- (1,0);\draw[bxp] (0.0107,0) -- (0.0107,0);\draw[bxp] (0.0107,0) -- (0.0107,0);\draw[bxp] (0.0107,-0.7) -- (0.0107,0.7);\draw[bxp] (0.0107,-0.7) -- (0.0107,0.7);\filldraw[boxfill,draw=black,line width=0.5pt] (0.0107,-1.1) rectangle (0.0107,1.1);\draw[medline] (0.0107,-1.1) -- (0.0107,1.1);\draw[outl] (0.0160,0) circle (1.6pt);\end{tikzpicture} \\[3pt]
5 & $81^3$ & 15\,198 & $3\times10^{-3}$ & \begin{tikzpicture}[x=4.5cm, y=1.6ex, baseline=-0.5ex]\draw[axisline] (0,0) -- (1,0);\draw[bxp] (0.1080,0) -- (0.1080,0);\draw[bxp] (0.1120,0) -- (0.1160,0);\draw[bxp] (0.1080,-0.7) -- (0.1080,0.7);\draw[bxp] (0.1160,-0.7) -- (0.1160,0.7);\filldraw[boxfill,draw=black,line width=0.5pt] (0.1080,-1.1) rectangle (0.1120,1.1);\draw[medline] (0.1120,-1.1) -- (0.1120,1.1);\end{tikzpicture} & \begin{tikzpicture}[x=4.5cm, y=1.6ex, baseline=-0.5ex]\draw[axisline] (0,0) -- (1,0);\draw[bxp] (0.0187,0) -- (0.0187,0);\draw[bxp] (0.0200,0) -- (0.0200,0);\draw[bxp] (0.0187,-0.7) -- (0.0187,0.7);\draw[bxp] (0.0200,-0.7) -- (0.0200,0.7);\filldraw[boxfill,draw=black,line width=0.5pt] (0.0187,-1.1) rectangle (0.0200,1.1);\draw[medline] (0.0200,-1.1) -- (0.0200,1.1);\end{tikzpicture} \\[3pt]
\midrule
6 & $102^3$ & 24\,034 & $1\times10^{-3}$ & \begin{tikzpicture}[x=4.5cm, y=1.6ex, baseline=-0.5ex]\draw[axisline] (0,0) -- (1,0);\draw[bxp] (0.3640,0) -- (0.3880,0);\draw[bxp] (0.4050,0) -- (0.4280,0);\draw[bxp] (0.3640,-0.7) -- (0.3640,0.7);\draw[bxp] (0.4280,-0.7) -- (0.4280,0.7);\filldraw[boxfill,draw=black,line width=0.5pt] (0.3880,-1.1) rectangle (0.4050,1.1);\draw[medline] (0.3960,-1.1) -- (0.3960,1.1);\draw[outl] (0.3600,0) circle (1.6pt);\end{tikzpicture} & \begin{tikzpicture}[x=4.5cm, y=1.6ex, baseline=-0.5ex]\draw[axisline] (0,0) -- (1,0);\draw[bxp] (0.1440,0) -- (0.1533,0);\draw[bxp] (0.1597,0) -- (0.1667,0);\draw[bxp] (0.1440,-0.7) -- (0.1440,0.7);\draw[bxp] (0.1667,-0.7) -- (0.1667,0.7);\filldraw[boxfill,draw=black,line width=0.5pt] (0.1533,-1.1) rectangle (0.1597,1.1);\draw[medline] (0.1553,-1.1) -- (0.1553,1.1);\draw[outl] (0.1427,0) circle (1.6pt);\draw[outl] (0.1693,0) circle (1.6pt);\end{tikzpicture} \\[3pt]
\midrule
7 & $128^3$ & 38\,265 & $1\times10^{-3}$ & \begin{tikzpicture}[x=4.5cm, y=1.6ex, baseline=-0.5ex]\draw[axisline] (0,0) -- (1,0);\draw[bxp] (0.5,0) -- (0.5150,0);\draw[bxp] (0.8240,0) -- (0.8480,0);\draw[bxp] (0.5,-0.7) -- (0.5,0.7);\draw[bxp] (0.8480,-0.7) -- (0.8480,0.7);\filldraw[boxfill,draw=black,line width=0.5pt] (0.5150,-1.1) rectangle (0.8240,1.1);\draw[medline] (0.8220,-1.1) -- (0.8220,1.1);\end{tikzpicture} & \begin{tikzpicture}[x=4.5cm, y=1.6ex, baseline=-0.5ex]\draw[axisline] (0,0) -- (1,0);\draw[bxp] (0.4827,0) -- (0.4977,0);\draw[bxp] (0.7950,0) -- (0.8173,0);\draw[bxp] (0.4827,-0.7) -- (0.4827,0.7);\draw[bxp] (0.8173,-0.7) -- (0.8173,0.7);\filldraw[boxfill,draw=black,line width=0.5pt] (0.4977,-1.1) rectangle (0.7950,1.1);\draw[medline] (0.7853,-1.1) -- (0.7853,1.1);\end{tikzpicture} \\[3pt]
\addlinespace[3pt]
\midrule
\multicolumn{4}{r}{{\textbf{Final-level median (cyclic)}}}
& \textbf{107} & \textbf{30.8} \\
\multicolumn{4}{r}{{\textbf{Final-level median (progressive)}}}
& \textbf{205} & \textbf{58.9} \\
\bottomrule
\end{tabular}
\end{table}
Every stage of the cyclic schedule reaches its target loss.  The schedule
first visits level~6 at a looser tolerance, then revisits levels~5 and~6
with tighter targets (boldface rows in
Table~\ref{tab:genus3_training_combined}) before ascending to the finest
grid.  Relative to the progressive schedule, these revisits reduce both the
IQR and the total epoch count; at the finest and most expensive level, the
cyclic schedule uses roughly half as many epochs.  In this example, the
cyclic schedule therefore improves both efficiency and run-to-run
consistency.

\subsection{Ring-shaped horizon data in five-dimensional general relativity}
\label{sec:Examples:blackring}

The final example applies MLSG to the exterior Robin problem for the
four-dimensional Laplacian on a ring-shaped hypersurface.

\paragraph{Physical origin.}
In four-dimensional spacetime, the familiar black-hole horizon has spherical
topology.  With one additional spatial dimension, other topologies become
possible: Emparan and Reall constructed five-dimensional vacuum solutions
with horizon topology $S^1\times S^2$~\cite{EmparanReall2002,EmparanReall2006}.
A constant-time slice of such a spacetime is four-dimensional.  For a
momentarily static, conformally flat slice with spatial metric
$\psi^2\delta$ outside the horizon, the constraint equations reduce to
Problem~3: $\psi$ is harmonic in the exterior, and the Robin condition with
$\beta=\tfrac13(\kappa_1+\kappa_2+\kappa_3)$ makes the boundary a marginal
surface~\cite{IdaNakao2002,Maxwell2005}.  The example below computes this
type of momentarily static data on an $S^1\times S^2$ hypersurface.

\paragraph{Geometry.}

Write $x=(x_1,x_2,x_3,x_4)\in\R^4$ and fix $0<r<R$.  Let

\begin{equation}\label{eq:ring-circle}
    C \;:=\; \bigl\{\, x \in \R^4 \;:\; x_1^2 + x_2^2 = R^2,\;
                       x_3 = x_4 = 0 \,\bigr\}
\end{equation}

be the radius-$R$ circle in the $(x_1,x_2)$-plane.  The ring surface is the
set of points at distance $r$ from $C$, equivalently the zero level set of

\begin{equation}\label{eq:ring-sdf}
    d_\Gam(x) \;=\; \sqrt{\bigl(\varrho(x) - R\bigr)^{2} + x_3^{2} + x_4^{2}} \;-\; r,
    \qquad
    \varrho(x) \;:=\; \sqrt{x_1^{2} + x_2^{2}},
\end{equation}
and bounds the solid ring $\Omega=\{x:d_\Gam(x)<0\}$.  Here $R$ is the
ring radius and $r$ the tube radius.  The condition $r<R$ makes $\Gam$ a
smooth embedded hypersurface diffeomorphic to $S^1\times S^2$: the normal
fiber over each point of $C$ is a two-sphere.  The closest-point projection
and principal curvatures are available in closed form
(Appendix~\ref{app:ring}), so the IBIM calculation uses exact geometric data.

\paragraph{The Robin coefficient.}
We set $\beta=\tfrac13(\kappa_1+\kappa_2+\kappa_3)$, one third of the
additive mean curvature.  For the ring, this coefficient reduces to
\eqref{eq:ring-K-closed} and is positive everywhere precisely when
$r/R<2/3$; the configuration used below satisfies this condition.  With
$\rho=\psi|_\Gam$, the second-kind equation has kernel
\eqref{eq:ring-kernel} and right-hand side $g\equiv1$.  Appendix~\ref{app:ring}
records the marginal-surface reduction and the closed-form ring geometry.

\paragraph{Steklov resonance.}
The formulation also has a resonance caveat.  The relevant obstruction is a
Robin, or Steklov-type, resonance: injectivity fails if the corresponding
homogeneous exterior Robin problem admits a nontrivial decaying
solution~\cite{kress2014linear}.  For a constant Robin coefficient this
reduces to the usual exterior Steklov spectral condition.  As in the
Helmholtz experiments, we choose a nonresonant test configuration; near
resonance, ill-conditioning would appear in the discrete operator $A_N$.
{{%
\paragraph{Computational results.}
We discretize the second-kind equation with IBIM, using tube half-width
$\epsilon=1.5h$ at every level.  The kernel is
\[
    k_{\mathrm{Lap},\R^4}(x,y)
    \;=\;-\left(
    \beta(y)\,G_{0,4}(x,y)
    + \frac{\partial G_{0,4}(x,y)}{\partial \mathbf{n}_y}\right).
\]

The training hierarchy refines the per-coordinate Cartesian resolution from
$n=32$ to $n=230$.  At the finest level it
contains {\color{black}$N=3\,288\,976$} tube quadrature points; every
optimization step uses a batch of $1024$ rows.  Table~\ref{tab:ring_training}
reports the quadrature size, target and achieved loss, and wall-clock time at
each level.

\begin{table}[htbp]
\centering
\caption{
MLSG training statistics for the exterior Robin problem on a ring-shaped
hypersurface in $\R^4$ under the progressive schedule
($R=1$, $r=1/2$; SIREN with $10$ hidden layers of $500$ neurons,
$2{,}007{,}001$ parameters, $\omega_0=30$; batch size $1024$).
Means and standard deviations of the achieved loss and wall-clock time are
computed over $10$ independent runs.  Here $n$ is the per-coordinate
Cartesian resolution, $h$ the grid spacing, and $N$ the number of tube
quadrature points, equal to the dimension of the dense BIE system; only the
Nystr\"om rows required by each mini-batch are evaluated on demand.}
\label{tab:ring_training}\medskip
\begin{tabular}{c r r r c c r r}
\toprule
\multicolumn{6}{c}{}
& \multicolumn{2}{c}{Time (sec)} \\
\cmidrule(lr){7-8}
Level
& {\color{black}$n$}
& $h$
& {\color{black}$N$} (matrix)
& Target loss
& Achieved loss
& Mean
& Std \\
\midrule
1
& $32$
& $0.1935$
& $8\,792$
& $5.0\times10^{-4}$
& $(3.66\pm0.94)\times10^{-4}$
& $3.64$
& $0.45$ \\

2
& $40$
& $0.1538$
& $17\,632$
& $5.0\times10^{-4}$
& $(4.21\pm0.59)\times10^{-4}$
& $13.15$
& $5.76$ \\

3
& $48$
& $0.1277$
& $30\,072$
& $5.0\times10^{-4}$
& $(4.14\pm0.58)\times10^{-4}$
& $3.68$
& $2.02$ \\

4
& $64$
& $0.0952$
& $70\,432$
& $5.0\times10^{-4}$
& $(4.19\pm0.52)\times10^{-4}$
& $19.32$
& $13.46$ \\

5
& $81$
& $0.0750$
& $142\,848$
& $3.0\times10^{-4}$
& $(2.46\pm0.28)\times10^{-4}$
& $32.15$
& $22.65$ \\

6
& $102$
& $0.0594$
& $283\,436$
& $2.0\times10^{-4}$
& $(1.50\pm0.23)\times10^{-4}$
& $86.23$
& $33.79$ \\

7
& $128$
& $0.0472$
& $561\,904$
& $2.0\times10^{-4}$
& $(7.48\pm3.97)\times10^{-5}$
& $256.22$
& $57.25$ \\

8
& $161$
& $0.0375$
& $1\,123\,712$
& $2.0\times10^{-4}$
& $(6.19\pm3.98)\times10^{-5}$
& $1\,220.93$
& $406.49$ \\

9
& $203$
& $0.0297$
& $2\,259\,292$
& $2.0\times10^{-4}$
& $(1.71\pm1.29)\times10^{-5}$
& $2\,057.12$
& $0.51$ \\

10
& $230$
& $0.0262$
& $3\,288\,976$
& $1.0\times10^{-4}$
& $(1.56\pm1.28)\times10^{-5}$
& $4\,338.34$
& $0.88$ \\
\midrule
\multicolumn{6}{l}{\textbf{Total}}
& $\mathbf{8\,030.78}$
& $\mathbf{414.14}$ \\
\bottomrule
\end{tabular}
\end{table}
\begin{figure}[htbp]
    \centering
    \begin{subfigure}[t]{0.3\linewidth}
        \centering
        \includegraphics[width=\linewidth]{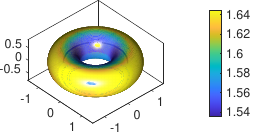}
        \caption{Density $\rho_\theta$ on the slice
        $\Gam \cap \{x_4 = 0\}$, in the $(x_1, x_2, x_3)$-space.}
        \label{fig:ring_density_torus}
    \end{subfigure}
    \hspace{1.2em}
    \begin{subfigure}[t]{0.3\linewidth}
    \centering
    \includegraphics[width=\linewidth]{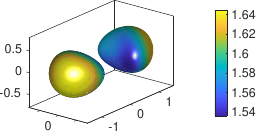}
    \caption{Density $\rho_\theta$ on the tilted slice, drawn in the $(x_1,x_2,x_4)$-space (so the vertical axis is $x_4$).}
    \label{fig:ring_density_tilt50}
    \end{subfigure}
    \hspace{1.2em}
    \begin{subfigure}[t]{0.30\linewidth}
        \centering
        \includegraphics[width=\linewidth]{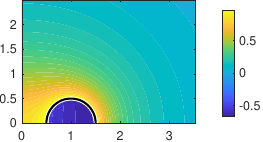}
        \caption{Reconstructed potential $u$ on the $(q,s)$-plane.}
        \label{fig:ring_u_qs}
    \end{subfigure}
    \caption{{The exterior Robin problem on a ring-shaped hypersurface in $\R^4$ with $R=1$, $r=\tfrac12$, $\beta=\tfrac13(\kappa_1+\kappa_2+\kappa_3)$, and $g\equiv1$, computed at finest resolution $n=230$ ($N=3\,288\,976$ quadrature points).  Biaxial rotational symmetry reduces the reconstructed four-dimensional solution in the right panel to the $(q,s)$-plane.}}
    \label{fig:ring_results}
\end{figure}

Figure~\ref{fig:ring_results} summarizes the learned density and reconstructed
field.  The left panel restricts $\rho_\theta$ to the slice $x_4=0$,
for which
\begin{equation*}
    \Gam \cap \{x_4 = 0\} = \left\{(x_1, x_2, x_3):\Bigl(\sqrt{x_1^2 + x_2^2} - R\Bigr)^{2} + x_3^2 \;=\; r^2\right\},
\end{equation*}

is the standard three-dimensional torus with $R=1$ and $r=\tfrac12$.
Because $\rho_\theta$ is an ambient network on $\R^4$, this restriction is
evaluated directly, without interpolation from the quadrature nodes.  The
middle panel shows the same density on the tilted slice
$\Gam\cap\{x_3=x_1\cot\theta\}$ with $\theta=50^\circ$; this section has two
disjoint components.

The right panel reconstructs the solution from the trained density on the
$(q,s)$-plane, where

\begin{equation*}
    q \;:=\; \sqrt{x_1^2 + x_2^2},
    \qquad
    s \;:=\; \sqrt{x_3^2 + x_4^2}
\end{equation*}

are the two rotational radii.  In these coordinates the ring is the circle
$(q-R)^2+s^2=r^2$, shown in black.  The surface, the Robin coefficient, and
the data are invariant under independent rotations in the
$(x_1,x_2)$- and $(x_3,x_4)$-planes, so the solution depends only on
$(q,s)$.  The quarter-plane plot therefore represents the full
four-dimensional field.

\section{Concluding remarks}\label{sec:Conclusion}

In this work, we introduced a multilevel stochastic-gradient neural solver (MLSG) that efficiently resolves second-kind boundary integral equations by coupling continuous network representations with a coarse-to-fine optimization strategy.

For a uniformly stable second-kind discretization, the strongly nonuniform contraction in neural residual minimization stems from the empirical neural tangent kernel (NTK) geometry, not from a resolution-dependent deterioration of the BIE operator. Equation~\eqref{eq:bie-ntk-spectral-equivalence} formalizes this: the spectrum of $A_NT_{\theta,N}A_N^\ast$ is equivalent to that of the empirical NTK up to grid-independent constants.

Our experiments validate the corresponding multilevel mechanism. Within a given level, parameter motion reshapes the finite-width NTK, while refinement re-samples this continuous representation onto a richer quadrature. Residual projections confirm that unresolved BIE residuals align with these expanding tangent directions, supporting a cascadic interpretation of MLSG. Furthermore, our cross-level estimate provides a theoretical basis for avoiding over-solving on coarse levels. By bounding the refined initial loss against the achieved coarse-level tolerance and discretization errors, Equation~\eqref{eq:warm-start-bound} motivates the matched scale $\varepsilon_\ell\asymp h_\ell^{2m}$.

Computationally, MLSG effectively shifts expensive optimization steps away from the finest resolutions and maps highly efficiently to modern GPU architectures. Using on-demand batched evaluation of Nystr\"om rows without hierarchical fast-summation machinery, a single-GPU implementation successfully scaled to million-unknown BIEs (e.g., $N=1~048~576$ for the coiled torus and $N=3{,}288{,}976$ for the $\R^4$ Robin problem). While the substantial wall-clock speedups over our 64-core MATLAB dense-GMRES reference are not hardware-normalized algorithmic claims, they strongly demonstrate that regular mini-batched kernel operations are exceptionally well-suited for current accelerators.

Finally, the $\R^4$ Robin example highlights a major advantage of this formulation: it remains dimension-agnostic, seamlessly extending to four ambient dimensions without altering the algorithmic structure. Realizing this methodology's full potential in high-dimensional settings will require advances in accurate singular quadrature and fast operator application. Future work must also focus on developing a quantitative theory for the evolving finite-width NTK under Adam, clarifying the emergence of effective directions during refinement, and constructing an end-to-end complexity model.

\section*{Acknowledgements}

Bing-Ze Lu received support from the National Science and Technology
Council, Taiwan, through Grants 113-2917-I-564-033 and
114-2115-M-194-007-MY3. 
Richard Tsai is partially supported by National Science Foundation grant
DMS-2513857.  Part of this research was also supported by the Swedish
Research Council under grant no.\ 2021-06594 while Tsai was in residence
at Institut Mittag-Leffler in Djursholm, Sweden, during the Fall 2025
semester.
This work used the Texas Advanced Computing Center (TACC) at The
University of Texas at Austin for development and testing.

\appendix

\section{Surface representations and discretization}\label{app:surf-disc}

This appendix gives the discretization details for the parametric and IBIM
surface representations used in Section~\ref{sec:Examples}.

\subsection{Parametric atlas}\label{app:surf-disc:param}

The atlas
$\{(D_\alpha, \varphi_\alpha)\}_{\alpha=1}^{M_{\mathrm{chart}}}$ of
Section~\ref{sec:Examples} satisfies
\begin{equation}\label{eq:atlas}
    \Gam \;=\; \bigcup_{\alpha=1}^{M_{\mathrm{chart}}} \Gam_\alpha,
    \qquad
    \varphi_\alpha \,:\, D_\alpha \subset \R^{d-1} \;\longrightarrow\;
    \Gam_\alpha \subset \Gam,
\end{equation}
with each $\varphi_\alpha$ a regular parametrization of the chart image
$\Gam_\alpha$ and the chart images pairwise disjoint up to a set of
$dS$-measure zero (their common boundaries). The surface integral
splits across charts as
\begin{equation}\label{eq:atlas-integral}
    \int_\Gam f\, dS
    \;=\; \sum_{\alpha=1}^{M_{\mathrm{chart}}}
            \int_{D_\alpha} f(\varphi_\alpha(\xi))\,
                                  \mu_\alpha(\xi)\,d\xi,
    \qquad
    \mu_\alpha(\xi) \,:=\,
    \sqrt{\det\!\bigl(D\varphi_\alpha(\xi)^{\!\top}
                      D\varphi_\alpha(\xi)\bigr)},
\end{equation}
where $\mu_\alpha$ is the chart-level surface element. The density on
$\Gam$ is represented by one MLP per chart, written as the chart-indexed
collection
\begin{equation}\label{eq:rho-param-mlps}
    \rho_\theta^{\mathrm{param}}
    \;=\; \bigl(\rho_{\theta(1)}, \ldots, \rho_{\theta(M_{\mathrm{chart}})}\bigr),
    ~~\rho_{\theta(\alpha)} \,:\, D_\alpha \,\longrightarrow\, \C,
    ~~
    \theta(\alpha) \in \R^{p_\alpha},
    ~~\alpha = 1, \ldots, M_{\mathrm{chart}},
\end{equation}
where $\theta(\alpha)$ collects the parameters of the chart-$\alpha$
network and the chart index is written in parentheses to keep it
separate from the per-component subscript $\theta_k$ of
Section~\ref{sec:loss}. The full parameter vector
$\theta = \bigl(\theta(1), \ldots, \theta(M_{\mathrm{chart}})\bigr) \in
\R^p$ has dimension
$p = \sum_{\alpha=1}^{M_{\mathrm{chart}}} p_\alpha$. The surface density
and the BIE data on chart $\alpha$ are
\begin{equation}\label{eq:rho-param}
    \rho(\varphi_\alpha(\xi)) \;=\; \rho_{\theta(\alpha)}(\xi),
    \qquad
    g_\alpha(\xi) \,:=\, g(\varphi_\alpha(\xi)),
    \qquad \xi \in D_\alpha.
\end{equation}

\paragraph{Atlas-level quadrature.}
Each chart carries its own quadrature on the parameter domain. On chart
$\alpha$ we choose nodes
$\{\xi_i^\alpha\}_{i=1}^{N_\alpha} \subset D_\alpha$ and positive
weights $\{\omega_i^\alpha\}_{i=1}^{N_\alpha}$ such that
\[
    \int_{D_\alpha} \phi(\xi)\,\mu_\alpha(\xi)\,d\xi
    \;\approx\; \sum_{i=1}^{N_\alpha} w_i^\alpha\,\phi(\xi_i^\alpha),
    \qquad
    w_i^\alpha \,:=\, \omega_i^\alpha\,\mu_\alpha(\xi_i^\alpha),
\]
which transports through $\varphi_\alpha$ to a quadrature on
$\Gam_\alpha$ with surface nodes
$x_i^\alpha := \varphi_\alpha(\xi_i^\alpha)$ and weights $w_i^\alpha$.
The atlas-level node set, sample size, and weight matrix are
\[
    \GamN \,=\, \bigsqcup_{\alpha=1}^{M_{\mathrm{chart}}}
                \{x_i^\alpha\}_{i=1}^{N_\alpha},
    \qquad
    N \,=\, \sum_{\alpha=1}^{M_{\mathrm{chart}}} N_\alpha,
    \qquad
    W \,=\, \mathrm{diag}\bigl(\{w_i^\alpha\}\bigr) \in \R^{N\times N}.
\]

Substituting \eqref{eq:rho-param} into \eqref{eq:BIE} and applying the
chart quadratures gives, for each
$\alpha = 1, \ldots, M_{\mathrm{chart}}$ and
$i = 1, \ldots, N_\alpha$,
\begin{equation}\label{eq:nystrom-charts}
    \tfrac{1}{2}\,\rho_{\theta(\alpha)}(\xi_i^\alpha)
    \;+\; \sum_{\beta=1}^{M_{\mathrm{chart}}} \sum_{j=1}^{N_\beta} w_j^\beta\,
        k\bigl(\varphi_\alpha(\xi_i^\alpha),\,\varphi_\beta(\xi_j^\beta)\bigr)\,
        \rho_{\theta(\beta)}(\xi_j^\beta)
    \;=\; g_\alpha(\xi_i^\alpha),
\end{equation}
the discrete equations on the atlas. For surfaces homeomorphic to
circles (in 2D) or to tori (in 3D), the natural quadrature on each
chart is a corrected truncated trapezoidal rule, chosen to handle
the on-chart singularity of $k$ consistently.

\paragraph{The two-chart system.}
We illustrate \eqref{eq:nystrom-charts} for $M_{\mathrm{chart}} = 2$.
Collect the per-chart sample vectors and pulled-back data,
\[
    \rho_{N_\alpha}^{\mathrm{param}}
    \,:=\, \bigl(\rho_{\theta(\alpha)}(\xi_i^\alpha)\bigr)_{i=1}^{N_\alpha}
    \in \C^{N_\alpha},
    \qquad
    g_{N_\alpha} \,:=\, \bigl(g_\alpha(\xi_i^\alpha)\bigr)_{i=1}^{N_\alpha}
    \in \C^{N_\alpha},
    \qquad \alpha = 1, 2,
\]
and define the inter-chart kernel blocks
\[
    K^{\alpha\beta} \,\in\, \C^{N_\alpha \times N_\beta},
    \qquad
    (K^{\alpha\beta})_{ij} \,:=\, w_j^\beta\,
    k\bigl(\varphi_\alpha(\xi_i^\alpha),\,\varphi_\beta(\xi_j^\beta)\bigr),
    \qquad \alpha, \beta \in \{1, 2\}.
\]
Then \eqref{eq:nystrom-charts} is the block linear system
\begin{equation}\label{eq:two-chart-system}
    \underbrace{
    \begin{pmatrix}
        \tfrac{1}{2} I_{N_1} + K^{11} & K^{12} \\[4pt]
        K^{21} & \tfrac{1}{2} I_{N_2} + K^{22}
    \end{pmatrix}
    }_{=\,A_N}
    \begin{pmatrix} \rho_{N_1}^{\mathrm{param}} \\[4pt]
                    \rho_{N_2}^{\mathrm{param}} \end{pmatrix}
    \;=\;
    \begin{pmatrix} g_{N_1} \\[4pt] g_{N_2} \end{pmatrix},
\end{equation}
with $W = \mathrm{diag}(W_1, W_2)$,
$W_\alpha := \mathrm{diag}(w_1^\alpha, \ldots, w_{N_\alpha}^\alpha)$,
giving the weighted inner product on $\C^N$, $N = N_1 + N_2$. The
diagonal blocks $K^{\alpha\alpha}$ contain the on-chart singular
interactions and are evaluated with the singularity-aware quadrature;
the off-diagonal blocks $K^{\alpha\beta}$, $\alpha \neq \beta$, are
smooth in $(\xi_i^\alpha, \xi_j^\beta)$, so any standard product rule
on $D_\alpha \times D_\beta$ suffices. The solver drives
\eqref{eq:two-chart-system} through residual minimization on the
concatenated parameters $\theta = \bigl(\theta(1), \theta(2)\bigr)$,
with $\rhoN$ in \eqref{eq:res-theta} replaced by the block vector
$(\rho_{N_1}^{\mathrm{param}}, \rho_{N_2}^{\mathrm{param}})^{\!\top}$.

\subsection{IBIM volumetric representation and tubular rule}\label{app:surf-disc:ibim}
\label{app:ibim}

When $\Gam \subset \R^d$ is a closed $C^2$ surface given
non-parametrically as the zero level set of a signed distance function
$d_\Gam$, we adopt the implicit boundary integral method (IBIM) of
\cite{KublikTanushevTsai2013, ChenTsai2017}. The density on $\Gam$ is
the trace of a single ambient MLP,
\begin{equation}\label{eq:rho-ibim}
    \rho_\theta^{\mathrm{IBIM}} \,:\, \R^d \,\longrightarrow\, \C,
    \qquad
    \rho(x) \;=\; \rho_\theta^{\mathrm{IBIM}}(x), \quad x \in \Gam,
\end{equation}
optimized by the solver of Section~\ref{sec:Algs}.

Let
$T_\epsilon := \{x \in \R^d : |d_\Gam(x)| < \epsilon\}$, with
$\epsilon$ smaller than the reach of the closed $C^2$ surface $\Gam$, so
that the closest-point projection
$P_\Gam(z) := z - d_\Gam(z)\,\nabla d_\Gam(z)$ is well defined on
$T_\epsilon$.  For every integrable $f$ on $\Gam$, the IBIM formulation uses
the exact tubular-neighborhood identity
\begin{equation}\label{eq:ibim-volume-identity}
    \int_\Gam f\,dS
    \;=\;
    \int_{T_\epsilon}
        f\bigl(P_\Gam(z)\bigr)\,
        \delta_\epsilon\bigl(d_\Gam(z)\bigr)\,J(z)\,dz.
\end{equation}
Applying the Cartesian trapezoidal rule with grid spacing $h$ gives the
IBIM quadrature
\begin{equation}\label{eq:ibim-quad}
    \int_\Gam f\,dS
    \;\approx\;
    \sum_{z_j \in T_\epsilon^h}
        h^d\,\delta_\epsilon\bigl(d_\Gam(z_j)\bigr)\,J(z_j)\,
        f\bigl(P_\Gam(z_j)\bigr),
\end{equation}
where $T_\epsilon^h$ is the set of Cartesian grid nodes inside
$T_\epsilon$, and
$\delta_\epsilon(r) := \frac{1}{\epsilon}\,\eta(r/\epsilon)$
is a regularized delta formed from a fixed mollifier $\eta \in C^1_c([-1,1])$
with $\int_{-1}^{1}\eta = 1$ --- for instance the raised cosine
$\eta(s) = \tfrac12(1+\cos\pi s)$, giving
$\delta_\epsilon(r) = \tfrac{1}{2\epsilon}\bigl(1+\cos(\pi r/\epsilon)\bigr)$
on $[-\epsilon,\epsilon]$ --- and
{$J(z) := \prod_{i=1}^{d-1}\bigl(1-d_\Gam(z)\,\kappa_i(z)\bigr)$,
where $\kappa_i(z)$ are the principal curvatures of the signed-distance
level set through the tubular node $z$ (equivalently, the nonzero eigenvalues
of $D^2d_\Gam(z)$ with the stated orientation).  This is the Jacobian of the
closest-point map used by the IBIM rule; in particular, the curvatures in
this formula are evaluated at $z$, not at $P_\Gam(z)$}~\cite{KublikTanushevTsai2013}.
Thus the volumetric integral in \eqref{eq:ibim-volume-identity} is
identical to the surface integral; the numerical approximation enters only
when that volume integral is discretized on the Cartesian grid.  With
$\epsilon \gtrsim h$, the trapezoidal rule \eqref{eq:ibim-quad} converges
to this common value as $h\to0$, at an order set by the smoothness and
moments of $\eta$.  The IBIM node set and weights are therefore
\begin{equation}\label{eq:ibim-nodes}
    \GamN \,=\, \bigl\{ x_j = P_\Gam(z_j) \,:\, z_j \in T_\epsilon^h \bigr\},
    \qquad
    w_j \,:=\, h^d\,\delta_\epsilon(d_\Gam(z_j))\,J(z_j),
\end{equation}
and the IBIM-discretized BIE operator
$(\mathcal{K}\rho)(x) = \int_\Gam k(x,y)\,\rho(y)\,dS(y)$ at quadrature
nodes $x_i = P_\Gam(z_i)$ is
\begin{equation}\label{eq:ibim-K}
    (K_N\rho)_i
    \;=\;
    \sum_{z_j \in T_\epsilon^h} w_j\,
        k\bigl(P_\Gam(z_i),\,P_\Gam(z_j)\bigr)\,
        \rho_\theta^{\mathrm{IBIM}}\bigl(P_\Gam(z_j)\bigr),
\end{equation}
with near-diagonal entries (those for which
$\|P_\Gam(z_i) - P_\Gam(z_j)\|$ is comparable to $h$) replaced by
curvature-dependent constants following
\cite{KublikTanushevTsai2013, ChenTsai2017}. High-order corrected
trapezoidal IBIM rules are developed in \cite{IzzoRunborgTsai2022}; the
experiments of Section~\ref{sec:Examples} use the basic rule
\eqref{eq:ibim-quad}.

\paragraph{Regularization of kernel singularities.}

In the volumetric examples, the layer kernels are singular, undefined, or
direction-dependent on the diagonal.  We replace the unresolved
near-diagonal contribution by a local constant determined from the surface
geometry at the target point.

\emph{Conventions.}
\(\Gamma\subset\mathbb R^{d}\) is oriented by the unit normal
\(n=\nabla d_\Gamma\), pointing out of \(\Omega\); the principal curvatures
\(\kappa_1,\dots,\kappa_{d-1}\) and the shape operator \(S_x\) are taken with
respect to \(n\), so that the sphere of radius \(a\) carries
\(\kappa_i=1/a\). They are computed from second derivatives of the signed
distance function. We write
\[
  H:=\frac{1}{d-1}\sum_{i=1}^{d-1}\kappa_i,
  \qquad
  \mathcal H:=\sum_{i=1}^{d-1}\kappa_i=(d-1)H
\]
for the averaged and the additive mean curvature. 

For \(d=2,3\) the regularizing constants are obtained by averaging the kernel
over an osculating surface above the tangent disc
\[
  U_T(x,r_0)
  =
  \left\{
    y\in\Gamma:
    \left|P_{T_x\Gamma}(y-x)\right|\le r_0
  \right\}.
\]
For \(d=4\) the quantity being regularized is instead the polar-weighted
integrand in tangent-space polar coordinates.

\emph{(i) Laplace in \(\mathbb R^2\) and \(\mathbb R^3\): double-layer kernel
\[
  k_{\mathrm{Lap}}(x,y)
  =
  -\,\frac{\partial G_{0,d}}{\partial n_y}(x,y).
\]}
In \(\mathbb R^2\), the double-layer kernel has the finite diagonal limit
\[
  k_{\mathrm{Lap}}(x,x)
  =
  \frac{\kappa(x)}{4\pi}
  =
  \frac{H(x)}{4\pi},
\]
under the curvature and normal conventions fixed above, and hence no
\(r_0\)-dependent regularization is required.

In \(\mathbb R^3\), averaging the kernel over the osculating paraboloid above
\(U_T(x,r_0)\) gives \cite[eq.~(29)]{KublikTanushevTsai2013}
\begin{equation}\label{eq:reg-laplace}
  \widetilde k_{\mathrm{Lap}}(x)
  =
  H\left[
    \frac{1}{4\pi r_0}
    -
    \frac{5\bigl(3H^2-\kappa_1\kappa_2\bigr)}{192\pi}\,r_0.
  \right]
\end{equation}

\emph{(ii) Helmholtz in \(\mathbb R^3\), exterior Neumann: adjoint
double-layer kernel
\[
  k_{\mathrm{Helm}}(x,y)
  =
  -\,\frac{\partial G_{\kappa,3}}{\partial n_x}(x,y),
\]
with \(n_x\) the exterior unit normal.}
Re-expanding the Helmholtz kernel over the same osculating paraboloid and
tangent-disc patch used in \eqref{eq:reg-laplace} gives
\begin{equation}\label{eq:reg-helmholtz}
  \widetilde k_{\mathrm{Helm}}(x)
  =
  H\left[
    \frac{1}{4\pi r_0}
    -
    \frac{5\bigl(3H^2-\kappa_1\kappa_2\bigr)}{192\pi}\,r_0
    +
    \frac{\kappa^2}{24\pi}\,r_0
  \right].
\end{equation}

Equation \eqref{eq:reg-helmholtz} is the re-derived tangent-disc formula used
here in place of the coefficients printed in
\cite[eq.~(2.8)]{ChenTsai2017}.

\emph{(iii) Laplace in \(\mathbb R^4\), exterior Robin (black ring):
combined kernel
\[
  k_{\mathrm{Lap},\mathbb R^4}(x,y)
  =
  -\left(
    \beta(y)\,G_{0,4}(x,y)
    +
    \frac{\partial G_{0,4}}{\partial n_y}(x,y)
  \right),
  \qquad
  \beta=\frac{\mathcal H}{3}.
\]}
Let \(x\in\Gamma\subset\mathbb R^4\), where \(\Gamma\) is a smooth
three-dimensional hypersurface. In tangent-space polar coordinates
\(y=y(t,\omega)\), \(t>0\), \(\omega\in\mathbb S^2\), the surface measure has
the local form
\(
  dS_y=t^2\bigl(1+\mathcal O(t)\bigr)\,dt\,d\omega .
\)
Since \(G_{0,4}(x,y)=\bigl(4\pi^2|x-y|^2\bigr)^{-1}\), the factor \(t^2\) in
the polar measure cancels the \(|x-y|^{-2}\) singularity, and the single-layer
term has the direction-independent limit
\begin{equation}\label{eq:reg-ring-sl}
  \lim_{t\to0} t^2\,\bigl(-\beta(y)G_{0,4}(x,y(t,\omega))\bigr)
  =
  -\frac{\beta(x)}{4\pi^2}.
\end{equation}
For the normal-derivative term the corresponding directional limit is
\begin{equation}\label{eq:reg-ring-dl-directional}
  \lim_{t\to0}
  t^2\left(-\frac{\partial G_{0,4}}{\partial n_y}\right)
  \bigl(x,y(t,\omega)\bigr)
  =
  \frac{\left\langle S_x\omega,\omega\right\rangle}{4\pi^2},
\end{equation}
where \(S_x\) is the shape operator at \(x\); the same limit is obtained with
\(n_y\) replaced by \(n_x\), again because
\((y-x)\cdot n_y=(x-y)\cdot n_x+\mathcal O(t^3)\). Therefore the directional
polar-diagonal limit of the combined Robin kernel is
\begin{equation}\label{eq:reg-ring-combined-directional}
  \lim_{t\to0}
  t^2\,k_{\mathrm{Lap},\mathbb R^4}\bigl(x,y(t,\omega)\bigr)
  =
  \frac{1}{4\pi^2}
  \Bigl[
    \left\langle S_x\omega,\omega\right\rangle
    -
    \beta(x)
  \Bigr].
\end{equation}
The angular average of the quadratic form of the shape operator satisfies
\begin{equation}\label{eq:shape-average-r4}
  \frac{1}{|\mathbb S^2|}
  \int_{\mathbb S^2}
  \left\langle S_x\omega,\omega\right\rangle\,d\omega
  =
  \frac{\operatorname{tr}S_x}{3}
  =
  \frac{\kappa_1+\kappa_2+\kappa_3}{3}
  =
  \frac{\mathcal H(x)}{3}.
\end{equation}
With \(\beta=\mathcal H/3\) the two averaged contributions cancel,
\[
  \widetilde k_{\mathrm{Lap},\mathbb R^4}(x)
  =
  0.
\] 

{\color{black}{
\section{Details for the ring-shaped horizon example}\label{app:ring}

Throughout this appendix, $\Omega\subset\R^4$ is the solid ring of
Section~\ref{sec:Examples:blackring}, $D=\R^4\setminus\overline\Omega$
is its exterior, and $\mathbf n$ denotes the unit normal on $\Gam$
pointing into $D$.

\paragraph{From the Hamiltonian constraint to the Robin condition.}
For a time-symmetric slice of a five-dimensional vacuum spacetime, the
momentum constraint holds trivially, and the Hamiltonian constraint
is the vanishing of the scalar curvature of the spatial metric. In
spatial dimension four, the conformal ansatz $g = \psi^2\delta$ gives
$R[\psi^2\delta] = -6\,\psi^{-3}\Delta\psi$, so the constraint is the
flat Laplace equation, and asymptotic flatness fixes $\psi \to 1$ at
infinity. A marginally trapped surface of such a slice is a minimal
surface of $g$; since the additive mean curvature of a
three-dimensional hypersurface transforms under
$\hat g = e^{2\phi} g$ as
$\widehat{\mathcal H}=e^{-\phi}\bigl(\mathcal H+3\,\partial_\nu\phi\bigr)$,
the condition $\widehat{\mathcal H}=0$ with $\phi=\log\psi$ is exactly
the Robin condition of Problem~3 with $\beta=\mathcal H/3=H$. On the sphere $|x| = a$
the exterior problem with constant coefficient $\beta$ has the exact
solution
\begin{equation}\label{eq:ring-sphere-exact}
    \psi(x) \;=\; 1 + \frac{c}{|x|^{2}},
    \qquad
    c \;=\; \frac{\beta a^{3}}{2 - \beta a},
    \qquad \beta a \neq 2,
\end{equation}
and the horizon value $\beta = 1/a$ gives $c = a^2$, the
isotropic-coordinate form of the four-dimensional
Schwarzschild--Tangherlini initial data. The solution
\eqref{eq:ring-sphere-exact} provides an exact reference for
validating the four-dimensional pipeline independently of the ring
geometry.

\paragraph{Closed-form geometry.}
For the surface \eqref{eq:ring-sdf}, every quantity entering the
tubular rule \eqref{eq:ibim-quad} is explicit. Off the plane
$\varrho = 0$, which the tube never meets when $\epsilon < R - r$,
the point of the circle $C$ nearest to $z$ is
$c(z) = (R/\varrho(z))\,(z_1, z_2, 0, 0)$, the closest-point
projection is $P_\Gam(z) = c(z) + r\,(z - c(z))/|z - c(z)|$, and
$\nabla d_\Gam(z) = (z - c(z))/|z - c(z)|$ exactly. The principal
curvatures with respect to $\mathbf{n} = \nabla d_\Gam$ and the
resulting Robin coefficients are
\begin{equation}\label{eq:ring-K-closed}
    \kappa_1 = \kappa_2 = \frac{1}{r}, \qquad
    \kappa_3(y) = \frac{\varrho(y) - R}{r\,\varrho(y)}, \qquad
    \beta(y) \;=\; \frac{\mathcal H(y)}{3} \;=\; H(y) \;=\;
    \frac{1}{3r}\left(3 - \frac{R}{\varrho(y)}\right).
\end{equation}
{In the tubular Jacobian of \eqref{eq:ibim-quad}, the
principal curvatures are evaluated at the Cartesian tube node $z$, as in
Appendix~\ref{app:surf-disc:ibim}, rather than at $P_\Gam(z)$.  Thus
\[
  J(z)=\prod_{i=1}^{3}\bigl(1-d_\Gam(z)\,\kappa_i(z)\bigr).
\]
For a sphere of radius $a$, $\kappa_i(z)=1/(a+d_\Gam(z))$ on the offset
level set, giving $J(z)=(a/(a+d_\Gam(z)))^3$, the expected closest-point
Jacobian.}  Exact surface coordinates are available for visualization and
for uniform sampling:
\begin{equation}\label{eq:ring-param}
    x(\varphi, \nu)
    \;=\;
    \bigl((R + r\nu_1)\cos\varphi,\ (R + r\nu_1)\sin\varphi,\
    r\nu_2,\ r\nu_3\bigr),
    \qquad \varphi \in [0, 2\pi),\ \nu \in S^2,
\end{equation}
with $dS = r^2 (R + r\nu_1)\,d\sigma(\nu)\,d\varphi$ and
$|\Gam| = 8\pi^2 R\,r^2$ exactly; the latter serves as a unit test
for the tube rule.

}} 
\bibliographystyle{abbrv}
\bibliography{reference}
\end{document}